\documentclass{article}

\usepackage[preprint]{neurips_2026}

\usepackage{tikz}
\usetikzlibrary{decorations.pathreplacing}
\usepackage{threeparttable}
\usepackage{amsthm}
\usepackage{amsmath}
\usepackage{amssymb}
\usepackage{mathrsfs}
\usepackage[utf8]{inputenc} 
\usepackage[T1]{fontenc}    
\usepackage{hyperref}       
\usepackage{url}            
\usepackage{booktabs}       
\usepackage{amsfonts}       
\usepackage{nicefrac}       
\usepackage{microtype}      
\usepackage{xcolor}         

\newcommand{\R}{\mathbb{R}}

\newcommand{\E}{\mathbb{E}}

\newcommand{\dd}{\text{d}}

\newtheorem{thm}{Theorem}[section]
\newtheorem{defn}[thm]{Definition}
\newtheorem{lem}[thm]{Lemma}
\newtheorem{prop}[thm]{Proposition}
\newtheorem{cor}[thm]{Corollary}
\newtheorem{rem}[thm]{Remark}

\newtheorem{assumption}[thm]{Assumption}

\title{Enhanced convergence guarantees of score-based generative models in $\mathcal{W}_2$-distance beyond log-concavity%
\thanks{Accepted at the Fortieth Annual Conference on Neural Information Processing Systems (NeurIPS 2026).}}

\author{%
  Xiaoyan Zhang
  \\
  Central South University\\
  \texttt{x.y.zhang@csu.edu.cn} \\
   \And
   Chenxu Pang \\
   Central South University\\
   \texttt{c.x.pang@csu.edu.cn} \\
   \And
   Xiaojie Wang\thanks{Corresponding author}  \\
   Central South University\\
   \texttt{x.j.wang7@csu.edu.cn} \\
}

\begin{document}

\maketitle

\begin{abstract}
%
Although score-based generative models (SGMs) have achieved remarkable success in real-world sample generation tasks, it is still far from sufficient to
understand them from a mathematical perspective.
In this paper, we aim to provide enhanced convergence guarantees for two existing SGMs in $\mathcal{W}_2$-distance beyond log-concavity. More precisely, we develop a novel framework of {\color{black}discretization} error analysis for
Euler-Maruyama (EM) and Poisson midpoint (PM) time discretization schemes for SGMs. Under a non-log-concavity condition, we show that $\tilde{\mathcal{O}}(\sqrt{d}/\epsilon)$ iterations suffice to approximate the target distributions in $\epsilon$-accuracy for the classical EM-based SGM, considerably improving upon the existing iteration complexity $\tilde{\mathcal{O}}(d/\epsilon^2)$. Notably, the novel framework of error analysis also enables us to establish an $\tilde{\mathcal{O}}(\sqrt{d}/\epsilon^{2/3})$ iteration complexity for the PM-based SGM in $\epsilon$-accuracy, significantly outperforming current state-of-the-art Wasserstein convergence guarantees for SDE-based diffusion samplers. 
\end{abstract}

\section{Introduction}
Diffusion models such as denoising diffusion probabilistic models (DDPMs) \cite{ho2020denoising} have achieved state-of-the-art
performance in data generation across a wide range of applications, including image generation \cite{dhariwal2021diffusion,ho2020denoising,songs2021Denoising}, natural language processing \cite{popov2021grad,li2022diffusion}, inverse problems \cite{chung2023diffusion,cardoso2024monte} and many other areas.
In many cases, they are shown to outperform other
generative models such as generative adversarial networks (GANs)
\cite{Goodfellow2014generative}, variational autoencoders (VAEs)
\cite{cardoso2014Auto}, and normalizing flows
\cite{rezende2015variational}.

%
%
Through
stochastic differential equations (SDEs), Song et al. \cite{songs2020core} originally
formulated a general theoretical framework for diffusion models, referred to as score-based generative models (SGMs),
%
%
consisting of two processes: the forward process and the reverse process.
%
Via the density flow \( \{ p_t\}_{t \geq 0}\), the forward process gradually (as $t \to \infty$) transforms the target distribution \( p_0=\mu\) into a tractable prior distribution \(p_{\infty}\), typically the standard Gaussian law.
This step is commonly referred to as the diffusion process. 
Along these diffusion trajectories, the gradient of the log-density function (i.e., \(\nabla\log p_t\)), known as the score function and generally not available in a closed form, is estimated using score matching techniques, such as denoising score matching \cite{Vincent2011Aconnection} or sliced score matching \cite{song2020sliced}.
%
%
By leveraging the learned score function \(s_{\theta^\star}\), one can rely on a time-discretization scheme (e.g., the classical Euler-Maruyama, EM) for the reverse process to convert random noise back into new samples that approximately follow the target distribution $\mu$.
Despite the remarkable empirical success of SGMs across various applications, their
theoretical understanding is still far from enough and the present paper attempts to provide
enhanced convergence guarantees of SGMs in $\mathcal{W}_2$-distance beyond log-concavity.
%


\subsection{Related works on convergence guarantees of SGMs}
{\color{black}

In recent years, a growing body of research has focused on the theoretical understanding of SDE-based SGMs, see, e.g., 
\cite{debortoli2022convergence,chen2023improved,chen2023sampling,lee2023convergence,kwon2022score,jiao2025optimal,li2024improved,li2024accelerating,liu2022let,lee2022convergence,ligen2025d/t,zhang2025sublinear}.
%
In early works
\cite{debortoli2022convergence,liu2022let},  convergence guarantees often suffer from exponential dependencies on parameters such as dimension and smoothness constant. 
This was overcome by \cite{chen2023sampling}, where a $\tilde{\mathcal{O}}(d/\epsilon^2)$ iteration complexity was obtained in total variation (TV) distance under the  spatial Lipschitz condition on the score functions $\nabla \log p_t(x)$ (with respect to $x$).  
Further, the authors of \cite{chen2023improved} established $\tilde{\mathcal{O}}(d/\epsilon^2)$ 
and $\tilde{\mathcal{O}}(d^2/\epsilon^2)$  iteration complexity bounds in KL divergence
under the spatial Lipschitz condition on $\nabla \log p_t(x)$ and $\nabla \log p_0(x)$,
respectively.
These error bounds scale polynomially and even have a logarithmic dependence on relevant problem parameters, significantly improving prior works.
%
More recently in \cite{ligen2025d/t}, the authors showed  $\tilde{\mathcal{O}}(d/\epsilon)$  iterations for the SDE-based sampler suffice to achieve a given precision level $\epsilon$ in TV distance, under the assumption that the target distribution has finite moments.
%

During the preparation of the present manuscript, 
we were aware of two interesting papers \cite{jiao2025optimal,zhang2025sublinear} on convergence guarantees of SDE-based samplers.
More accurately, the authors of \cite{jiao2025optimal} obtained an improved iteration complexity of order $\tilde{\mathcal{O}}(\sqrt{d}/\epsilon)$ in TV distance and KL divergence under the assumption that the score function is spatially Lipschitz. 
Also, the authors of \cite{zhang2025sublinear} proposed a denoising diffusion randomized midpoint (DDRaM) method for the time discretization of the reverse process and derived a $\tilde{\mathcal{O}}(\sqrt{d}/\epsilon)$ iteration complexity in KL divergence, under the condition that the estimated score function is spatial Lipschitz (not uniform-in-time).
Besides, we would like to mention a series of recent works \cite{li2024improved,li2024accelerating,li2024sharp,chen2023probability,Huang2025Convergence},  providing theoretical error bounds in TV distance for ODE-based diffusion samplers.

\textbf{Convergence in Wasserstein distance with and without log-concavity.}
Note that the above-mentioned convergence guarantees are all obtained in TV distance and KL divergence. In contrast, convergence guarantees in Wasserstein distance are less studied in the literature.
To analyze the error in $\mathcal{W}_2$-distance, previous works often
imposed the log-concavity condition on the target distribution (see, e.g., \cite{bruno2025diffusion,Gao2025Wasserstein,gao2025convergence,strasman2025analysis,yu2025advancing} and references therein), which can be also utilized to ensure the above frequently used Lipschitz condition on the score function.
}

%
%
%
%
Nevertheless, the log-concavity condition is too restrictive and precludes complex and multi-modal real-world target distributions.
To remedy it, some researchers try to analyze the error in $\mathcal{W}_2$- distance under certain non-log-concavity conditions such as the weak log-concavity \cite{Silveri2025beyond,kremling2025non} and the log-Sobolev inequality (LSI) condition \cite{lee2022convergence}.
{\color{black}
The authors of \cite{bruno2025wasserstein} established convergence guarantees under semi log-concavity and strong log-concavity at infinity allowing for potentially discontinuous gradients, obtaining iteration complexities of $\tilde{\mathcal{O}}(d/\epsilon^2)$ in the general setting and $\tilde{\mathcal{O}}(d/\epsilon)$ under additional regularity assumptions on the score approximation in $\mathcal{W}_2$-distance.
}
In particular, for a class of weakly log-concave target distributions, the authors of \cite{Silveri2025beyond} derived iteration complexity $\tilde{\mathcal{O}}(d/\epsilon^2)$ for the EM scheme to achieve a given precision level $\epsilon$ in $\mathcal{W}_2$-distance, in the setting that the target distribution satisfies the weak log-concavity condition. 
%
For a general class of probability flow ODEs, the authors of \cite{kremling2025non} established an $\tilde{\mathcal{O}}(\sqrt{d}/\epsilon)$ iteration complexity in $\mathcal{W}_2$-distance under the  weak log-concavity condition on the target distribution.

Bearing the above error bounds in mind, the following interesting question thus arises:

{
\it
Beyond log-concavity, can the classical EM scheme for the SDE-based SGMs achieve an
$\tilde{\mathcal O}(\sqrt d/\epsilon)$ iteration complexity in
$\mathcal W_2$-distance?
}

This paper provides an affirmative answer to this question.
As illustrated by Theorem \ref{thm:em}, an $\tilde{\mathcal{O}}(\sqrt{d}/\epsilon)$ iteration complexity in $\mathcal{W}_2$-distance can be derived for the EM method under the weak log-concavity condition.
This matches the best known iteration complexity bound in any metric for the EM discretization of the reverse SDE.

{\color{black}
It is worth mentioning that, the authors of \cite{arsenyan2025assessing} derived the optimal $\tilde{\mathcal O}(\sqrt d/\epsilon)$ iteration complexity for the DDPM (EM) sampler in $\mathcal W_2$-distance under a uniform
posterior-covariance condition for Gaussian perturbations.
%
More recently in \cite{xixian2026wasserstein}, 
the authors also established the
$\tilde{\mathcal O}(\sqrt d/\epsilon)$ iteration complexity in $\mathcal W_2$-distance under a Gaussian-tail assumption on the target distribution. 
Compared to these two works, we rely on different assumptions and proof strategies in the error analysis. Notably, the assumptions and proof techniques used in this paper apply to other sampling algorithms, such as the Poisson midpoint method \cite{kandasamy2024poisson,srinivasan2026poisson}.
}

It is known that, the EM method still suffers from low sampling speed, requiring a large number of score function evaluations (often neural network evaluations) to generate new samples. 
A natural accelerated alternative is Runge-Kutta type method \cite{wu2024stochastic,Huang2025Convergence}, but the price is to impose stronger smoothness assumption on the (estimated) score function. 
Indeed, accelerated convergence guarantees of Runge-Kutta type methods for SGMs in $\mathcal{W}_2$-distance beyond log-concavity are, to the best of our knowledge, still missing in the literature.

To speed up the sampling, in this work we do not rely on Runge-Kutta type methods, but turn to the Poisson midpoint (PM) method \eqref{eq:poisson-method} originally proposed by \cite{kandasamy2024poisson}, where only convergence guarantees in KL divergence of PM-based Langevin Monte Carlo (LMC) were established.
Very recently in \cite{srinivasan2026poisson}, the authors further obtained a $\tilde{\mathcal{O}}(d^{2/3}/\epsilon^{2/3})$ iteration complexity in $\mathcal{W}_2$-distance for the PM-based overdamped LMC under the strongly log-concave condition and the Lipschitz condition of the drift.
Although numerical results in \cite{kandasamy2024poisson} showed that the PM-based SGMs substantially outperform the vanilla DDPM and ODE-based methods,
the theoretical convergence guarantee of the PM-based SGM is unfortunately still absent in the literature.
Therefore, we ask the following question:


{
\it Can an enhanced iteration complexity be obtained for PM-based SGM beyond log-concavity, to validate its outperformance over the other SDE-based samplers?}

In this work, we answer this question in the affirmative and establish an $\tilde{\mathcal{O}}(\sqrt{d}/\epsilon^{2/3})$ iteration complexity in $\mathcal{W}_2$-distance (see Theorem \ref{thm:poisson}) under the weak log-concavity condition. This thus fills the theoretical gap left in \cite{kandasamy2024poisson}. 

\subsection{Contributions of this work}
\label{subsec:Contributions}
%
To conclude, 
our contribution can be summarized as follows:
\begin{itemize}

\item  {\color{black}\textbf{Refined discretization error analysis.}} 
A {\color{black}refined} approach of discretization error analysis is developed to characterize the error propagation of the scheme. This helps us provide enhanced convergence guarantees in $\mathcal{W}_2$-distance for SDE-based SGMs with two existing discretization schemes beyond log-concavity, including the Euler-Maruyama method and Poisson midpoint method.

\item  \textbf{Enhanced iteration complexity for Euler--Maruyama discretization in $\mathcal{W}_2$-distance.}
We obtain a 
$\tilde{\mathcal{O}}(\sqrt{d}/\epsilon)$ iteration complexity of the EM
discretization~\eqref{eq:EM} of the reverse SDE that suffices to achieve
$\mathcal{W}_2$ error $\epsilon$ (see Theorem \ref{thm:em}), significantly improving the previous
$\tilde{\mathcal{O}}(d/\epsilon^2)$ EM bound of \cite{Silveri2025beyond}
under weak log-concavity.

\item \textbf{Sharper iteration complexity via Poisson midpoint discretization in $\mathcal{W}_2$-distance.}
For the Poisson midpoint method \eqref{eq:poisson-method}, we establish an iteration complexity of $\tilde{\mathcal{O}}\Big(\tfrac{\sqrt{d}}{\epsilon^{2/3}}\Big)$ for SGMs in $\mathcal{W}_2$-distance (see Theorem \ref{thm:poisson}), under weak
log-concavity and an additional Lipschitz-in-time condition on the score
(Assumption~\ref{ass:lipschitz-time}). 
    To the best of our knowledge, this bound fills the theoretical gap left in \cite{kandasamy2024poisson} and improves upon existing state-of-the-art Wasserstein convergence guarantees for SDE-based diffusion samplers.

\end{itemize}

In Table \ref{tab:Comparison}, we compare our theoretical findings with the existing results, in terms of the total iterations required to achieve a given precision level $\epsilon>0$.
Clearly, our results improve upon existing state-of-the-art Wasserstein convergence guarantees for SDE-based diffusion samplers.

\begin{table}
\caption{A comparison of iteration complexity for SGMs.}
\label{tab:Comparison}
\centering
\begin{threeparttable}
\begin{tabular}{cccccc}
\toprule
&
Algorithm      
& 
Metric 
& 
Iteration 
&
Log
&
Additional  
\\
&
&
&
complexity
&
-concavity
&
assumption\tnote{1}
\\
\midrule
\cite{chen2023sampling}
& 
EM
&
TV 
& 
$\tilde{\mathcal{O}}(d/\epsilon^2)$
& 
No
&
$\nabla \log p_t(x)$ Lipschitz in $x$
\\ \midrule 
\cite{chen2023improved}
& 
EM 
&
KL \tnote{2}
&
$\tilde{\mathcal{O}}
(d/\epsilon^2)$
& 
No 
&
$\nabla \log p_t(x)$ Lipschitz in $x$
\\
& 
EM 
&
KL \tnote{2}
&
$\tilde{\mathcal{O}}
(d^2/\epsilon^2)$
& 
No 
&
$\nabla \log p_0(x)$ Lipschitz in $x$
\\\midrule 
\cite{Gao2025Wasserstein}
& 
EM
&
$\mathcal{W}_2$
& 
$\tilde{\mathcal{O}}
(d /\epsilon^2)$
& 
Yes
&
$\nabla \log p_t(x)$ Lipschitz in $t$
\\\midrule 
\cite{yu2025advancing}
& 
EM
&
$\mathcal{W}_2$
& 
$\tilde{\mathcal{O}}
(d /\epsilon^2)$
& 
Yes
&
$\nabla \log p_t(x)$ Lipschitz in $t$
\\\midrule 
\cite{zhang2025sublinear} 
& 
DDRaM \tnote{3}
&
KL 
& 
$\tilde{\mathcal{O}}
(\sqrt{d} /\epsilon)$
& 
No 
&
$\tilde{s}_{\theta}(t,x)$ Lipschitz in $x$
\\\midrule 
\cite{kremling2025non}
& 
Euler (ODE)
&
$\mathcal{W}_2$
& 
$\tilde{\mathcal{O}}
(\sqrt{d} /\epsilon)$
& 
No
&
Weak log-concavity 
\\
& 
&
& 
& 
&
$\nabla \log p_t(x)$ Lipschitz in $t$
\\\midrule 
\cite{Silveri2025beyond}
& 
EM
&
$\mathcal{W}_2$
& 
$\tilde{\mathcal{O}}
(d /\epsilon^2)$
& 
No
&
Weak log-concavity 
\\\midrule 
This work
& 
EM
&
$\mathcal{W}_2$
& 
$\tilde{\mathcal{O}}
(\sqrt{d} /\epsilon)$
& 
No
&
Weak log-concavity 
\\\midrule 
This work
& 
PM 
&
$\mathcal{W}_2$
& 
$\tilde{\mathcal{O}}
(\sqrt{d} /\epsilon^{2/3})$
& 
No
&
Weak log-concavity 
\\
& 
&
& 
& 
&
$\nabla \log p_t(x)$ Lipschitz in $t$
\\
\bottomrule
\end{tabular}
\begin{tablenotes}
\item[1]
Smoothness assumptions on the (approximate) score function.
\item[2]
For error bounds in KL divergence, the given precision level is set to be $\epsilon^2$. 
\item[3] 
Denoising diffusion
randomized midpoint method. 
\end{tablenotes}
\end{threeparttable}
\end{table}

The rest of this article is organized as follows. The next section collects some notations, introduces SGMs and two types of methods throughout this paper.
In Section \ref{sec:assumption-result}, we present assumptions and main results of this paper, whose proofs are given in Appendix.
A technical overview of the proofs is given in Section \ref{sec:Technical-Overview}.
Finally, some concluding remarks are given in the last section.

\section{Preliminaries on SGMs}
\label{SGMs}
\subsection{Notation}
Throughout this paper, we use $\mathbb{N}$ to denote the set of all positive integers and let $d\in \mathbb{N}$. 
For any positive integer $m$, we denote $[m]:= \{1,...,m\}$ and $[m]_0:= \{0, 1,...,m\}$.
We use $I_d \in \mathbb{R}^{d\times d}$ to denote the identity matrix.
Let $\|\cdot\|$ and $\langle\cdot, \cdot\rangle$ denote the Euclidean norm and the inner product of vectors in $\mathbb{R}^d$, respectively. 
%
For any matrix $A=(a_{i,j})\in \mathbb{R}^{d\times d}$, the $\mathrm{F}$-norm is defined as
$\|A\|_{\mathrm{F}}
:=
\sqrt{\sum_{i,j=1}^d |a_{i,j}|^2}$ and the operator norm as $\|A\| =\|A\|_{\mathrm{op}} :=\sup_{\|v\|=1}\|Av\|$. It is not difficult to see
$
\|A\|
\leq
\|A\|_{\mathrm{F}}
\leq 
\sqrt{d}
\,
\|A\|.
$
We denote the density of the standard Gaussian distribution on $\mathbb{R}^d$ by $\gamma^d$.
For two numbers $a$ and $b$, we denote $a \vee b:=\max\{a, b\}$ and $a \wedge b:=\min\{a, b\}$. 
We write $\tilde{\mathcal{O}}(\cdot)$ to mean that it hides polylogarithmic factors in $d$ and $1/\epsilon$.

Let $\mathcal{B}(\mathbb{R}^d)$ be the Borel $\sigma$-field on $\mathbb{R}^d$, and let $\mathcal{P}(\mathbb{R}^d)$ denote the set of probability measures on $(\mathbb{R}^d,\mathcal{B}(\mathbb{R}^d))$. 
We write $\mathcal{P}_2(\mathbb{R}^d)$ for the subset of probability measures with finite second moments. 
By $\mathcal{L}(X)$ we denote the probability distribution of the $\mathbb{R}^d$-valued random variable $X$.
Given $\nu_1,\nu_2\in\mathcal{P}_2(\mathbb{R}^d)$, let $\mathcal{D}(\nu_1,\nu_2)$ be the set of all couplings of $\nu_1$ and $\nu_2$, i.e., probability measures on $\mathbb{R}^d\times\mathbb{R}^d$ whose first and second marginals are $\nu_1$ and $\nu_2$, respectively. 
The $\mathcal{W}_2$-distance between $\nu_1$ and $\nu_2$ is then defined by
\begin{equation}
\begin{aligned}
\label{eq:_def_w_2_distance}
\mathcal{W}_2
(\nu_1, \nu_2)
& :=
\inf_{\nu \in \mathcal{D}
(\nu_1, \nu_2)}
\left(
\int_{\mathbb{R}^d} \int_{\mathbb{R}^d}
\left\|
\theta_1-\theta_2
\right\|^2 
\mathrm{\, d} 
\,
\nu
\left(
\theta_1, \theta_2
\right)
\right)^{1/2}.
\end{aligned}
\end{equation}

Let $\mathcal{C}\left(\mathbb{R}^d, \mathbb{R}\right)$ denote the set of all continuous functions from $\mathbb{R}^d$ to $\mathbb{R}$. 
%
For $k \geq 0$, denote by $\mathcal{C}^k\left(\mathbb{R}^d, \mathbb{R}\right)$ the set of functions from $\mathbb{R}^d$ to $\mathbb{R}$ which have continuous $0$-th, $\ldots, k$-th order derivatives.

\subsection{Score-based generative models}

%
In this subsection, we provide a brief summary of the ideas behind the construction of score-based generative models (SGMs) based on diffusion processes \cite{songs2020core}. 
An SGM consists of a forward process and a reverse process. 
The forward process progressively transforms the target distribution
into a simple reference distribution. 
In this paper, we use the Ornstein--Uhlenbeck (OU) diffusion as the forward process. Specifically, let $\{\overrightarrow{X}_t\}_{0\leq t\leq T}$ solve
\begin{equation}
\label{eq:forward-SDE}
    \mathrm{d}\overrightarrow{X}_t
    =
    -\overrightarrow{X}_t\,\mathrm{d}t
    +
    \sqrt{2}\,\mathrm{d}B(t),
    \quad 
    \overrightarrow{X}_0 \sim \mu,
    \quad 0\leq t \leq T,
\end{equation}
where $\mu(\mathrm{d}x)\propto \exp(-V(x))\,\mathrm{d}x$ is the target distribution and $\{B(t)\}_{0\leq t\leq T}$ is an $\mathbb{R}^d$-valued standard Brownian motion. 
Then the OU process \eqref{eq:forward-SDE} admits the explicit representation as below
\begin{equation}
\label{eq:solution_ou}
\overrightarrow{X}_t
\overset{\mathrm{d}}{=}
\mathrm{e}^{-t}\overrightarrow{X}_0
+
\sqrt{1-\mathrm{e}^{-2t}}\,Z,
\quad 
Z \sim \gamma^d,
\end{equation}
where $Z$ is independent of $\overrightarrow{X}_0$. 
Thus, $\overrightarrow{X}_t$ converges exponentially to the standard Gaussian law $\gamma^d$.

SGMs generate new samples by simulating the time reversal of
\eqref{eq:forward-SDE}, which transports the terminal law $\mathcal{L}(\overrightarrow{X}_T)$ back to $\mu$.
Let $\{\overleftarrow{X}_t\}_{0\leq t\leq T}$ denote the corresponding exact reverse process. 
By the time-reversal formula for diffusions \cite{anderson1982reverse,Cattiaux2023Timereversal}, it satisfies
\begin{equation}
\label{eq:backward}
\mathrm{d} 
\overleftarrow{X}_t
=
\big(
\overleftarrow{X}_t
    +
    2 \nabla \log {p}_{T-t}
    (
    \overleftarrow{X}_t)
    \big) 
    \,
    \mathrm{d} t
    +
    \sqrt{2} 
    \,
    \mathrm{d} W(t),
    \quad 
    \overleftarrow{X}_0 \sim 
    \mathcal{L}(\overrightarrow{X}_T),
    \quad \quad 0\leq t \leq T,
\end{equation}
where $p_t$ denotes the density of $\mathcal{L}(\overrightarrow{X}_t)$ with respect to Lebesgue measure, and $\{W(t)\}_{0\leq t\leq T}$ is another $\mathbb{R}^d$-valued standard Brownian motion  defined on a filtered probability space $(\Omega_W,\mathcal{F}^W,$\ $\{\mathcal{F}^W_t\}_{t\geq 0},\mathbb{P}_W)$, satisfying the usual conditions.
Moreover, such a Brownian motion is assumed to be independent of $\{B(t)\}_{0\leq t\leq T}$. 
The drift $\nabla\log p_t$ of the reverse-time SDE  is the gradient of the log-density of $\mathcal{L}(\overrightarrow{X}_t)$, also known as the \emph{score function}.
Since
$\nabla \log \gamma^d
(x) = -x$, the backward SDE \eqref{eq:backward} can be rewritten in the following equivalent form:
\begin{equation}
\label{eq:backward-SDE}
    \mathrm{d} 
    \overleftarrow{X}_t
    =
    \big(
    -\overleftarrow{X}_t
    +
    2 \nabla \log \tilde{p}_{T-t}
    (
    \overleftarrow{X}_t)
    \big) 
    \,
    \mathrm{d} t
    +
    \sqrt{2} 
    \,
    \mathrm{d} W(t),
    \quad 
    \overleftarrow{X}_0 \sim 
    \mathcal{L}(\overrightarrow{X}_T),
    \quad \quad 0\leq t \leq T,
\end{equation}
where, for $(t, x) \in [0, T]\times \mathbb{R}^d$,
\[
\tilde{p}_t(x):= p_t(x)/\gamma^d(x).
\]
This time-reversal technique is extensively used in the literature (see, e.g., \cite{Cattiaux2023Timereversal,Conforti2025KLconvergence,Silveri2025beyond,strasman2025wasserstein,strasman2025analysis}), 
owing to its preservation of the forward Ornstein--Uhlenbeck \eqref{eq:forward-SDE} drift \(-x\) in the reverse dynamics.
However, since the score function $\nabla \log {p}_{t}(x)$, which appears in \eqref{eq:backward}, is not available. 
To obtain a
practical sampler, we replace it by a learned approximation
$s_\theta:[0,T]\times\mathbb{R}^d\to\mathbb{R}^d$ from a parametric
family $\{s_\theta\}_{\theta\in\Theta}$, with parameter $\theta^\star$
obtained by minimizing the denoising score-matching loss
\cite{Vincent2011Aconnection}:
\begin{equation}
\label{eq:score-estimator}
\theta 
\rightarrow
\int_0^T
\E
\Big[
\big\|
s_\theta
(t,\overrightarrow{X}_t)
-
\nabla \log p_t
(
\overrightarrow{X}_t)
\big\|^2
\Big] 
\dd t.
\end{equation}
Since 
$\nabla \log \tilde{p}_{t}(x)
=
x+\nabla \log {p}_{t}(x)
$ for any $x\in \mathbb{R}^d$ and $t \in [0,T]$, therefore, the corresponding estimator of the relative score is then defined by
\[
\tilde{s}_{\theta^{\star}}(t,x):=x+s_{\theta^{\star}}(t,x).
\]
Since $\mathcal{L}(\overrightarrow{X}_T)$ depends on the unknown target $\mu$ and is not directly accessible, we initialize the sampler from the stationary law $\gamma^d$ as shown in \eqref{eq:solution_ou}.
Replacing in addition the relative score in \eqref{eq:backward-SDE} by $\tilde{s}_{\theta^\star}$
yields the following practical sampler
\begin{equation}
\label{eq:reverse-sde-estimator}
\begin{aligned}
\mathrm{d} 
    Y_t
    =
    \big(
    -Y_t
    +
    2 \tilde{s}_{\theta^{\star}}
    (T-t, Y_t)
    \big) 
    \,
    \mathrm{d} t
    +
    \sqrt{2} 
    \,
    \mathrm{d} W(t),
    \quad 
    Y_0 \sim \gamma^d,
    \quad \quad 0\leq t \leq T.
\end{aligned}
\end{equation}

\subsection{Time discretization schemes}
For a general learned score, 
the reverse SDE \eqref{eq:reverse-sde-estimator}
cannot be simulated exactly.
Sampling generation therefore requires a time-discretization scheme for \eqref{eq:reverse-sde-estimator}.
Let $ \pi_h $ be a temporal grid of the form 
$$
\pi_h
:=
\big\{
t_{n}+\tfrac{ih}{k}=n \cdot h+\tfrac{ih}{k},
\quad n\in [N]_0, \, k\in [N],
\quad i\in[k-1]_0 \text { with } h=T / N
\big\}.
$$
We consider two types of discretization methods, which are widely used in existing work.
\begin{itemize}
    \item 
    Euler-Maruyama (EM) method, for any $n \in [N-1]_0, \,N \in \mathbb{N}$:
\begin{equation}
\label{eq:EM}
Y_{t_{n+1}}
=
Y_{t_n}
+
h
\big(
-Y_{t_n}
+
2\tilde{s}_{\theta^{\star}}(T-t_n,Y_{t_n})
\big)
+ 
\sqrt{2}\Delta W_{t_n},
\quad Y_0 \sim \gamma^d,
\end{equation}
where $\Delta W_{t_n}:= W(t_{n+1})-W(t_{n})$.

\item Poisson midpoint (PM) method \cite{srinivasan2026poisson,kandasamy2024poisson}, for any $n \in [N-1]_0, \,i \in [k-1]_0,\,\,N,k \in \mathbb{N}$: 
\begin{equation}
\label{eq:poisson-method}
\begin{aligned}
\tilde{Y}_{t_{n,i}}
& =
Y_{t_n}
+
\tfrac{ih}{k}
\big(
-Y_{t_n}
+
2\tilde{s}_{\theta^{\star}}(T-t_n,Y_{t_n})
\big)
+
\sum_{j=0}^{i-1}
\sqrt{\tfrac{2h}{k}}
Z_{t_{n,j}},\\
Y_{t_{n,i+1}}
& =
Y_{t_{n,i}}
+
\tfrac{h}{k}
\big(
-Y_{t_n}
+
2\tilde{s}_{\theta^{\star}}(T-t_n,Y_{t_n})
\big)
+
h B_{t_{n,i}}
\Big[
\big(
-\tilde{Y}_{t_{n,i}}
\\
& \quad
+
2\tilde{s}_{\theta^{\star}}(T-(t_n+\tfrac{ih}{k}),\tilde{Y}_{t_{n,i}})
\big)
-
\big(
-
Y_{t_n}
+
2\tilde{s}_{\theta^{\star}}(T-t_n,Y_{t_n})
\big)
\Big]
+
\sqrt{\tfrac{2h}{k}}
Z_{t_{n,i}}, \\
Y_{{t_{n+1}}} & =
Y_{{t_{n+1},0}} =Y_{t_{n,k}},
\quad Y_0 \sim \gamma^d,
\end{aligned}
\end{equation} 
where $Z_{t_{n,i}}$ denote independent standard Gaussians and $B_{t_{n,i}}$ denote independent Bernoulli random variables with parameter $1/k$.
\end{itemize}

For simplicity, we denote $t_{n,i} := t_n + \frac{ih}{k}$ for any $n \in [N-1]_0, \,i \in [k-1]_0$ in the sequel. 
The Poisson midpoint method was first introduced in \cite{kandasamy2024poisson} as a discrete variant of the randomized midpoint method proposed by \cite{shen2019randomized}.
Noting that $\E[\sum_{j=0}^{k-1}B_{t_{n,j}}]=1$, we observe that an expected number of $2$ gradient score function evaluations per batch is required, which matches the randomized midpoint method in \cite{shen2019randomized}.

\section{Assumptions and main results}
\label{sec:assumption-result}
We first delineate the assumptions underlying our results. We begin with the following definition of weak convexity \cite{Conforti2024Weaksemiconvexity,Silveri2025beyond}.  
\begin{defn}
[Weak convexity]
\label{def:weak-convexity}
The weak convexity profile of a function  $g \in \mathcal{C}^1(\mathbb{R}^d,\mathbb{R})$  is defined as
\begin{equation}
\kappa_g(r) 
:= 
\inf_{x, y \in \mathbb{R}^d: \|x - y\| = r} 
\left\{ \frac{\langle \nabla g(x) - \nabla g(y), x - y \rangle}{\|x - y\|^2} \right\}, 
\quad r > 0. 
\end{equation}
We say that $g$  is $(\alpha, M)$-weakly convex if
\begin{equation}
\kappa_g(r) 
\geq 
\alpha - 
\frac{1}{r} 
f_M(r) \quad \forall \,r > 0
\end{equation}
for some constants $\alpha, M > 0 $ and
\begin{equation}
f_M(r) 
: = 
2\sqrt{M} 
\tanh\left(\tfrac{1}{2}\sqrt{M}r\right).
\end{equation}
Moreover, we say that \( g \) is \((\alpha, M)\)-weakly concave if \(-g\) is \((\alpha, M)\)-weakly convex.
\end{defn}
First, we make the following weak convexity assumption of the potential function $V$.
\begin{assumption}
[Weak convexity]
\label{ass:weak-convex}
The target distribution \(\mu\) is absolutely continuous w.r.t. Lebesgue measure with \(\mu(\mathrm{d} x) \propto \exp(-V(x)) \mathrm{d} x\), where \(V: \mathbb{R}^d \to \mathbb{R}\) is {\color{black}$(\alpha,M)$-weakly convex for some constants \(\alpha,M>0\).}
\end{assumption}

Assumption~\ref{ass:weak-convex} relaxes the strong log-concavity condition commonly imposed in prior analyses \cite{Gao2025Wasserstein,yu2025advancing,bruno2025diffusion,strasman2025analysis}.
For example, certain Gaussian mixture models satisfy this weak convexity condition while failing to be strongly log-concave (see \cite[Appendix A]{Silveri2025beyond}).

{\color{black}A standard sufficient condition for Assumption \ref{ass:weak-convex} is the
contractivity at infinity condition (see \cite[Assumption 1.1, H2']{Conforti2024Weaksemiconvexity} or \cite[Assumption 1.1]{majka2020nonasymptotic}), which
implies the weak convexity Assumption
\ref{ass:weak-convex} (see \cite[Proposition 5.1]{Conforti2024Weaksemiconvexity} for details).
}

\begin{assumption}
[One-sided Lipschitz]
\label{ass:one-side-lipschitz}
The function $\nabla V: \mathbb{R}^d \rightarrow \mathbb{R}^d$ is one-sided Lipschitz continuous. 
Namely, there exists a constant
$L_V\geq 0$, independent of $d$, such that, for any $ x,y \in \mathbb{R}^d$,
\begin{equation}
\label{eq:one-side-lipschitz}
\langle
\nabla V(x)-\nabla V(y),x-y\rangle
\leq L_V\|x-y\|^2.
\end{equation}
\end{assumption}

All the results rely on the following $L^2$-accuracy of the score estimator. This assumption quantifies
the ability of the neural network architecture to approximate the true score function and serves to
control the score approximation error.

\begin{assumption}
[Score estimation error]
\label{ass:score-function-approximation}
There exist $\epsilon_0 \geq 0$ and $\theta^\star \in \Theta$ such that for any $n \in [N-1]_0, \,i \in [k-1]_0$,
    \begin{equation}
    \big\|
    \nabla \log \tilde{p}_{T-t_{n,i}}
    ({Y}_{t_{n,i}})-\tilde{s}_{\theta^{\star}}(T-t_{n,i}, {Y}_{t_{n,i}})
    \big\|_{L^2(\Omega, \mathbb{R}^d)}
    \leq
    \epsilon_0, 
    \end{equation}
    where ${Y}_{t_{n,i}}$ is produced by EM method \eqref{eq:EM} or PM method \eqref{eq:poisson-method}.
\end{assumption}

{\color{black}
Indeed, Assumption \ref{ass:score-function-approximation} can be replaced by the standard score-matching assumption along the forward noising distribution commonly used in the literature (e.g.,\cite{chen2023sampling,chen2023improved})
and the additional assumptions that $\tilde{s}_{\theta^\star}(T-t_{n,i}, \cdot), \, \nabla \log \tilde{p}_{T-t_{n,i}} (\cdot)$ are Lipschitz in space.
With this substitution, we can still arrive at the same error bounds after very minor modifications to the proof of the main theorem.
In this situation, the required spatial Lipschitz continuity of $\nabla \log \tilde{p}_{T-t_{n,i}} (\cdot)$ can be guaranteed by Lemma \ref{lem:score-lipschitz} and it only remains to assume the spatial Lipschitz of the learned score $\tilde{s}_{\theta^\star}(T-t_{n,i}, \cdot)$, as also used in \cite{chen2023probability}. 

We adopt the score error along the trajectory of the approximate process in Assumption \ref{ass:score-function-approximation} because it leads to the most direct and tractable treatment of the score approximation error within our proof framework. Such assumptions were also employed in existing literature, e.g., \cite{Gao2025Wasserstein,kremling2025non,yu2025advancing} in the manuscript.
}

\subsection{Main result for the Euler-Maruyama method}


In the following theorem, we quantify the $\mathcal{W}_2$-distance between the distribution of EM method \eqref{eq:EM} and the target distribution $\mu$.

\begin{thm}
\label{thm:em}
[Main result: error bounds for EM]
Suppose that Assumptions \ref{ass:weak-convex}, \ref{ass:one-side-lipschitz} and \ref{ass:score-function-approximation} hold. 
Let $\left\{Y_{t_n}\right\}_{n \in [N]_0}$ be produced by the EM method \eqref{eq:EM} with the uniform step-size $h=T/N$. 
For any $\lambda\in(0,1)$, assume that
\[
0 < h \leq \tfrac{2-\lambda}{2(2L+1)^2},
\]
where $L>0$ 
is the  Lipschitz constant of $x\mapsto\nabla\log\tilde p_t(x)$ from Lemma \ref{lem:score-lipschitz}.
Then there exists a constant $C>0$, independent of $d$ and $h$, such that
\begin{equation}
    \begin{aligned}
    \mathcal{W}_2
    \left(\mu, \mathcal{L}\left(Y_{t_N}\right)\right)  \leq
    C
    \mathrm{e}^{C \eta(\alpha, M,(2L+1)^2h+\lambda/2)}
    \Big(
    \underbrace{
    \sqrt{T d}h
    }_{\text{Discretization error}}
    +
    \underbrace{
    \mathrm{e}^{-T}
    \mathcal{W}_2
    (\mu,
    \gamma^d)}_{\text{Initialization error}}
    +
    \underbrace{
    \epsilon_0 T}_{\text{Score-matching error}}
    \Big),
    \end{aligned}
    \end{equation}
where $\epsilon_0$ comes from Assumption \ref{ass:score-function-approximation} and $\eta\big(\alpha, M,(2L+1)^2h+\lambda/2\big)$ is defined in \eqref{eq:eta-def}.
\end{thm}

{\color{black}A brief proof overview is provided in the next section, while the complete proof is deferred to Appendix.}
As a direct consequence, we obtain the following scaling of the iteration number, terminal time $T$, and score approximation error required to achieve a given $\mathcal{W}_2$ accuracy.
The proof is straightforward and omitted here.

\begin{cor}
\label{cor:EM}
    Let Assumptions \ref{ass:weak-convex}, \ref{ass:one-side-lipschitz} and \ref{ass:score-function-approximation} hold. 
    To achieve a given precision level $\epsilon>0$ in $\mathcal{W}_2$-distance, the required number of iterations of the EM method \eqref{eq:EM}
    is of order $\tilde{\mathcal{O}}\big(\tfrac{\sqrt{d}}{\epsilon}\big)$ for the discretization error to be small, the time scale $T$ is of order $\mathcal{O}\big(\log(\tfrac{\sqrt{d}}{\epsilon})\big)$ for the initialization error to be small, and the score-matching error $\epsilon_0$ is of order $\mathcal{O}\Big(\tfrac{\epsilon}{\log(\epsilon^{-1}\sqrt{d})}\Big)$ for the propagated score-matching error to be small.
\end{cor}

{\color{black}
As validated in \cite[Proposition 4.1]{Silveri2025beyond}, in the case of Gaussian mixture distributions with equal variances, 
\begin{equation}
p_0(x)
:=
\sum_{i=1}^{m}
\beta_i
\frac{1}{(2\pi\sigma^2)^{d/2}}
\exp\left(
-\frac{\|x-\mu_i\|^2}{2\sigma^2}
\right),
\qquad x\in\mathbb R^d,
\end{equation}
where $\sigma>0$, $\mu_i\in\mathbb R^d$, and $\beta_i \in [0,1]$ for $i \in [m]$, such that 
$
\sum_{i=1}^{m}\beta_i=1,
$
the $-\log p_0 $ is weakly convex and $\nabla \log p_0$ is Lipschitz. Moreover, the corresponding weak convexity parameters $\alpha_{p_0}=1/\sigma^2,\, \sqrt{M_{p_0}}:= 2m \sum_{i=1}^m \|\mu_i\|/\sigma^2$ and the Lipschitz constant $L_{p_0}:=\alpha_{p_0}+{M_{p_0}}$ is independent of $d$, whenever
$\max_{1\le i\le m}\|\mu_i\|$ is bounded independently of $d$ (see \cite[Appendix A]{Silveri2025beyond} for details). Hence, the stated complexity bound $\tilde{\mathcal{O}}\big(\tfrac{\sqrt{d}}{\epsilon}\big)$ has no hidden dependence on $d$ for fixed 
$\sigma$ and $\mu_i$.

In this more general setting, where the covariance matrices of each mode of the Gaussian mixture differ, the required weak convexity and global score Lipschitz properties may not be satisfied. Even when these properties hold, the associated constants may depend on $d$.

Consequently, for general Gaussian mixture distributions with unequal covariance matrices, our analysis does not guarantee that all hidden constants in the stated complexity bound are independent of the dimension $d$.
}

\subsection{Main result for the Poisson midpoint method}

The PM analysis additionally requires the score to be Lipschitz continuous in time,
which is
a standard regularity assumption in the literature
\cite{kremling2025non,yu2025advancing,Gao2025Wasserstein}.

\begin{assumption}
[Lipschitz in time]
\label{ass:lipschitz-time}
There exists a constant $\tilde{L}>0$ such that for any $n \in [N-1]_0$, it holds that
    \begin{equation}
    \sup_{t_n\leq t,s\leq t_{n+1}}
    \big\|
    \nabla \log {p}_{T-t}
    (x)
    -
    \nabla \log {p}_{T-s}
    (x)
    \big\|
    \leq 
    \tilde{L} 
    h
    (1+\|x\|), 
    \quad \forall x \in \mathbb{R}^d.
    \end{equation}
\end{assumption}

{\color{black}
Next, we present examples satisfying the above assumption, whose proof is deferred to Appendix \ref{app:proof-lipschitz-t}.

\begin{prop}
\label{prop:time-lipschitz-gaussian-mixture}
Let $p_0$ be a Gaussian mixture in $\mathbb{R}^d$ having density law
\[
p_0(x)=
\sum_{i=1}^{m}
\beta_i
\frac{1}{(2\pi\sigma^2)^{d/2}}
\exp\left(
-\frac{\|x-\mu_i\|^2}{2\sigma^2}
\right),
\]
where $\sigma>0, \, \mu_i \in \mathbb{R}^d$ and $\beta_i \in [0,1],$ for $i \in [m]$, such that $\sum_{i=1}^m\beta_i=1$.
Then there exists a constant $\tilde L>0$, independent of $d$, such that
\[
\|\partial_r\nabla\log p_r(x)\|
\le
\tilde L (1+\|x\|),
\qquad
r\ge0,\ x\in\mathbb R^d .
\]
Consequently, for any $n \in [N-1]_0$,
\[
\sup_{t,s\in[t_n,t_{n+1}]}
\|\nabla \log p_{T-t}(x)
-\nabla\log p_{T-s}(x)\|
\le
\tilde L h(1+\|x\|),
\quad
x\in\mathbb R^d.
\]
\end{prop}
}

Thanks to the above assumptions, we are able to derive a non-asymptotic error bound in $\mathcal{W}_2$-distance between the distribution of PM method \eqref{eq:poisson-method} and the target distribution $\mu$. Its sketch of the complete proof is provided in Section \ref{sec:Technical-Overview}.
\begin{thm}
\label{thm:poisson}
[Main result: error bounds for PM]
    Suppose that Assumptions \ref{ass:weak-convex}, \ref{ass:one-side-lipschitz}, \ref{ass:score-function-approximation} and \ref{ass:lipschitz-time} hold. 
    Let $\left\{Y_{t_n}\right\}_{n \in [N]_0}$ be produced by the PM method \eqref{eq:poisson-method} using the uniform step-size. For any $\lambda\in (0,1/4)$, let the step-size satisfy
    \[
    h \leq
    \tfrac{2(1-\lambda)}
    {C(L, \tilde{L})}
    \wedge
    \tfrac{1}{C(L,\alpha)},
    \]
    where $C(L, \tilde{L})$ and $C(L,\alpha)$ are given by \eqref{eq:numerical-bounded-constant-c} and \eqref{eq:C-def}, respectively.
    Then there exists a constant $C$, independent of $d$ and $h$, such that,
    \begin{equation}
    \begin{aligned}
    &
    \mathcal{W}_2
    \left(\mu, \mathcal{L}\left(Y_{t_N}\right)\right) \\
    & \leq
    C
    \mathrm{e}^{C\tilde{\eta}(\alpha, M, \rho_h)}
    \Big(
    \underbrace{
    T^3\big(
    \tfrac{\sqrt{d}h}{k}
    +
    k d h^2
    +
    dh^2
    +
    \sqrt{d}h^{3/2}
    \big)}_{\text{Discretization error}}
    +
    \underbrace{
    \mathrm{e}^{-T}
    \mathcal{W}_2
    (\mu,
    \gamma^d)}_{\text{Initialization error}}
    +
    \underbrace{
    \epsilon_0 T
    }_{\text{Score matching error}}
    \Big).
    \end{aligned}
    \end{equation}
    with $\tilde{\eta}(\alpha, M, \rho_h):=\eta(\alpha, M,C(L, \tilde{L})h/2+\lambda)+\eta(\alpha, M,(3(2L+1)^2 h/2+\lambda) \vee C(L,\alpha)h)$, and $\eta\big(\alpha, M, \rho)$ as defined in \eqref{eq:eta-def}, where $L>0$ is the Lipschitz constant of $x\mapsto\nabla\log\tilde p_t(x)$ from Lemma \ref{lem:score-lipschitz} and $\epsilon_0$ comes from Assumption \ref{ass:score-function-approximation}.
\end{thm}

As a direct consequence, to achieve a given precision level $\epsilon>0$, meaning that 
$\mathcal{W}_2
\left(\mu, \mathcal{L}\left(Y_{t_N}\right)\right)
\leq \epsilon$, we obtain the following result regarding the required number of iterations, the time scale $T$ and the order of the score-matching error, {\color{black} whose proof is postponed to Appendix \ref{app:proof-mixting-time}.}

\begin{cor}
\label{cor:PM}
    Let Assumptions \ref{ass:weak-convex}, \ref{ass:one-side-lipschitz}, \ref{ass:score-function-approximation} and \ref{ass:lipschitz-time} hold. 
    To achieve a given precision level $\epsilon>0$ in $\mathcal{W}_2$-distance, the required number of iterations of the PM method \eqref{eq:poisson-method}
    is of order $\tilde{\mathcal{O}}\big(\tfrac{\sqrt{d}}{\epsilon^{2/3}}\big)$, the batch size $k$ is of order $\tilde{\mathcal{O}}\big(\tfrac{\sqrt{d}h}{\epsilon}\big)$ for the discretization error to be small, the time scale $T$ is of order $\mathcal{O}\big(\log(\tfrac{\sqrt{d}}{\epsilon})\big)$ for the initialization error to be small, and the score-matching error $\epsilon_0$ is of order $\mathcal{O}\Big(\tfrac{\epsilon}{\log(\epsilon^{-1}\sqrt{d})}\Big)$ for the propagated score-matching error to be small.
\end{cor}

\section{Technical Overview}
\label{sec:Technical-Overview}

\subsection{Regime shift property}
For simplicity, in the following we denote $b_t(x):= -x+2\nabla \log \tilde{p}_{T-t}(x)$. Then the backward SDE \eqref{eq:backward-SDE} can be rewritten as follows:
\begin{equation}
\mathrm{d} 
    \overleftarrow{X}_t
    =
    b_t(\overleftarrow{X}_t)
    \,
    \mathrm{d} t
    +
    \sqrt{2} 
    \,
    \mathrm{d} W_t,
    \quad 
    \overleftarrow{X}_0 \sim 
    \mathcal{L}(\overrightarrow{X}_T),
    \quad \quad 0\leq t \leq T.
\end{equation}
In this subsection, we revisit the regime shift results of $p_t$ and $b_t(x)$ established in \cite{Silveri2025beyond}, as they play a central role in proving the convergence guarantees of the two SGMs.

\textbf{Regime shift of $p_t$.}

\begin{itemize}
\item $p_t$ is only weak log-concave for $t \in \big[0, \tau(\alpha, M)\big ) $.

\item $p_t$ is log-concave for $t \in [ \tau(\alpha, M),T ] $.
\end{itemize}

\textbf{Regime shift of $b_t(x)$.}

\begin{itemize}
\item $b_t(x) $ is contractive for $t \in \big[0, \hat{T}(\alpha,M, 0)\big]$.

\item $b_t(x) $ is not (necessarily) contractive for $t \in [ \hat{T}(\alpha,M, 0),T ]$ (i.e., it is contractive in $\big[\hat{T}(\alpha,M,0),T^\star\big]$, while not contractive in $\big(T^\star,T\big]$).
\end{itemize}

Appendix \ref{appen:regime-shift} provides further details and the explicit formulas for $\tau(\alpha, M), \, \hat{T}(\alpha,M ,0)$ and $T^\star$. We mentioned that the introduction of $\hat{T}(\alpha,M ,0)$ as a lower bound on $T^\star$ is not a limitation of our proof technique but a direct consequence of the fact that $T^\star$ is not explicitly computable.

\subsection{Overview of proofs of Theorem \ref{thm:em} and Theorem \ref{thm:poisson}}
\label{subsec:Overview-proofs}

{\color{black}
We sketch the proofs of Theorem~\ref{thm:em} (EM) and Theorem
\ref{thm:poisson} (PM). 
{\color{black}
Before proceeding with the proof, we introduce the following two types of auxiliary process as used in \cite{Silveri2025beyond}.
}

\paragraph{Auxiliary processes.}
We introduce the following auxiliary processes,  for any $n \in [N]_0, \, i \in [k-1]$, 
\begin{itemize}
\item $\{X^N_{t_{n,i}}\}$: discretization of $\{\overleftarrow{X}_t\}_{0 \leq t \leq T}$ with 
$X^N_0 \sim \mathcal{L}(\overleftarrow{X}_0)$;
\item $\{X^\infty_{t_{n,i}}\}$: same discretization as $X^N$ but
with $X^\infty_0 \sim \gamma^d$.
\end{itemize}

For EM, both auxiliary processes follow
\begin{equation}
\label{eq:em-aux}
X^{\star}_{t_{n+1}}
= 
X^{\star}_{t_n}
+ h
b_{t_n}(X^{\star}_{t_n})
+ 
\sqrt{2}
\Delta W_{t_n},
\qquad
\star \in \{N,\infty\}.
\end{equation}
For PM, with refined sub-grid $t_{n,i} := t_n + ih/k$,
\begin{equation}
\label{eq:poisson-aux}
\begin{aligned}
\tilde X^{\star}_{t_{n,i}}
&= X^{\star}_{t_n}
+ 
\tfrac{ih}{k}
b_{t_n}(X^{\star}_{t_n})
+ \sqrt{\tfrac{2h}{k}} 
\sum_{j=0}^{i-1} Z_{t_{n,j}}, \\
X^{\star}_{t_{n,i+1}}
&= X^{\star}_{t_{n,i}}
+ \tfrac{h}{k}
b_{t_n}(X^{\star}_{t_n})
+ h
B_{t_{n,i}}
\Big(
b_{t_{n,i}}(\tilde X^{\star}_{t_{n,i}})
- 
b_{t_n}(X^{\star}_{t_n})
\Big)
+ 
\sqrt{\tfrac{2h}{k}}\, Z_{t_{n,i}},\\
X^{\star}_{{t_{n+1}}} 
& =
X^{\star}_{{t_{n+1},0}} 
=
X^{\star}_{t_{n,k}},
\quad
\star \in \{N,\infty\},
\end{aligned}
\end{equation}
where $\{Z_{t_{n,i}}\}$ are i.i.d.\ $\gamma^d$, $\{B_{t_{n,i}}\}$ are
i.i.d.\ $\mathrm{Bernoulli}(1/k)$, and the two families are mutually
independent. 

\paragraph{Error decomposition.}
By the triangle inequality, 
with $\mu = \mathcal{L}(\overleftarrow{X}_T)$,
\begin{equation}
\label{eq:error-decom}
\mathcal{W}_2
\bigl(
\mu, \mathcal{L}(Y_{t_N})
\bigr)
\leq
\underbrace{
\mathcal{W}_2
\bigl(
\mu, \mathcal{L}(X^N_{t_N})
\bigr)
}_{\mathcal{E}_{\mathrm{disc}}}
+ 
\underbrace{
\mathcal{W}_2
\bigl(
\mathcal{L}(X^N_{t_N}),
\mathcal{L}(X^\infty_{t_N})
\bigr)
}_{
\mathcal{E}_{\mathrm{init}}
}
+ 
\underbrace{
\mathcal{W}_2
\bigl(
\mathcal{L}(X^\infty_{t_N}),
\mathcal{L}(Y_{t_N})
\bigr)
}_{\mathcal{E}_{\mathrm{score}}
}.
\end{equation}
The initialization and score-approximation error, i.e., $\mathcal{E}_{\mathrm{init}}$ and $\mathcal{E}_{\mathrm{score}}$, are treated in the
same way for EM and PM methods, while the discretization term $\mathcal{E}_{\mathrm{disc}}$ is treated
separately.

\paragraph{Step 1: Discretization error $\mathcal{E}_{\mathrm{disc}}$.}
Combining with a general one-step recursion
(Proposition \ref{prop:Upper-mean-square error-bounds} for PM, and a
parallel statement for EM), 
and iterated with the regime-shift drift
bound (Proposition~\ref{prop:regime-shift-drift}) gives, 
\[
\mathcal{E}_{\mathrm{disc}}
=
\begin{cases}
\mathcal{O}\big(\sqrt{Td}\,h\big),
& \text{EM},\\[2pt]
\mathcal{O}\left(T^3\left(
\tfrac{\sqrt{d}\,h}{k}
+ \sqrt{k}d h^{2}
+ \sqrt{d}\,h^{3/2}
\right)\right),
& \text{PM}.
\end{cases}
\]
{\color{black}
Compared with \cite{Silveri2025beyond}, our 
analysis for EM improves the
discretization error to first order in \(h\), while retaining the same dimension dependence.
This is achieved through a
refined estimate of the cross terms in the $L^2$ one-step error recursion.
{\color{black}
This refined estimate isolates the martingale component of the one-step residual, which vanishes due to the basic property of the It\^o integral. }
For PM, we further exploit the rank-one structure of
the centered Bernoulli correction to obtain a sharper mean-square bound (Lemma \ref{lem:key-lem}), which yields the significant speed-up over EM.
{\color{black}
Further algorithmic and analytical intuition behind the improved PM bound, together with the corresponding technical coupling argument, is provided in Appendix \ref{app:B.1}.}


\paragraph{Step 2: Initialization error $\mathcal{E}_{\mathrm{init}}$.}
The processes $X^N$ and $X^\infty$ share the same numerical randomness
and differ only at $t_0$. Choosing
$(X^N_0,X^\infty_0)$ as the $\mathcal{W}_2$-optimal coupling
of $\mathcal{L}(\overleftarrow{X}_0)$ and $\gamma^d$,
\[
\mathcal{E}_{\mathrm{init}}^2
\leq
\mathbb{E}\bigl[\|X^N_{t_N} - X^\infty_{t_N}\|^2\bigr]
\leq
Ce^{-T}\,\mathcal{W}_2^2\bigl(\mu,\gamma^d\bigr),
\]
where the exponential factor comes from the contractivity of the
forward OU process \eqref{eq:forward-SDE}. 

\paragraph{Step 3: Score-approximation error $\mathcal{E}_{\mathrm{score}}$.}
Since
$X^\infty$ and $Y$ differ only in the score used inside the drift,
standard arguments yield
\[
\mathcal{E}_{\mathrm{score}}
\leq C\,T\epsilon_0,
\]
where $\epsilon_0$ is the score-approximation error in
Assumption~\ref{ass:score-function-approximation}.

Thus, for a target $\mathcal{W}_2$-accuracy $\epsilon$, the above
convergence rates give the improved iteration complexity
$\tilde{\mathcal{O}}(\sqrt{d}/\epsilon)$ for EM
(Corollary~\ref{cor:EM}) and the sharper iteration complexity
$\tilde{\mathcal{O}}(\sqrt{d}/\epsilon^{2/3})$ for PM
(Corollary~\ref{cor:PM}).
}

\section{Conclusion and future work}
\label{sec:Conclusion}
In this work, we provide enhanced iteration complexity in $\mathcal{W}_2$ distance for two discretization methods applied to SGMs. First, by carrying out more delicate error estimates, we derive an $\tilde{\mathcal{O}}(\sqrt{d}/\epsilon)$ iteration complexity for the EM-based SGM beyond log-concavity, significantly improving the existing order $\tilde{\mathcal{O}}(d/\epsilon^2)$ obtained in \cite{Silveri2025beyond}. Second, we establish an $\tilde{\mathcal{O}}(\sqrt{d}/\epsilon^{2/3})$ iteration complexity for the PM-based SGM in $\mathcal{W}_2$-distance, filling a theoretical guarantee left in \cite{kandasamy2024poisson} and improving upon existing state-of-the-art Wasserstein convergence guarantees for SDE-based diffusion samplers.

Several interesting avenues remain to be explored in our future works. First, we seek to develop new techniques that improve the dependence on $\alpha, M$ and $L_V$. The dependence we have established is exponential, whereas obtaining a polynomial dependence constitutes our future work. Second, we intend to build on this idea by developing more efficient sampling methods based on critically-damped Langevin diffusion, under weaker assumptions than weak log-concavity, and also for other metrics. {\color{black}Third,  we plan to extend our analysis to sampling on low-dimensional manifolds, seeking convergence bounds that depend on the intrinsic rather than the ambient dimension.
}

{
\small
\bibliography{ref}
\bibliographystyle{plain}
}

\appendix

\section{Propagation of the regularity assumption} 
\label{appen:regime-shift}
The aim of this section is to revisit some regime shift properties established in \cite{Silveri2025beyond}.
The following proposition shows how the constant of weak log-concavity propagates. 
The proof follows the similar argument of \cite[Lemma~5.9]{Conforti2025Projected} and is omitted.

\begin{prop}
[Propagation of weak log-concavity in time]
\label{prop:Propagation-weak-concavity} Let Assumption \ref{ass:weak-convex} hold. Then, for any $t\in[0,T]$, the function $-\log \tilde{p}_{T-t}(x)$ is weakly convex with weak convexity profile $\kappa_{- \log \tilde{p}_{T-t}}$ satisfying 
\begin{equation}
\kappa_{- \log \tilde{p}_{T-t}}(r)
\geq 
\frac{\alpha}
{\alpha + (1 - \alpha)e^{-2(T-t)}} - 1 - \frac{e^{-(T-t)}}{\alpha + (1 - \alpha)e^{-2(T-t)}} 
\frac{1}{r} 
f_M
\left(\tfrac{e^{-(T-t)}}{\alpha + (1 - \alpha)e^{-2(T-t)}} r\right),
\end{equation}
where $\alpha, M$ come from Assumption \ref{ass:weak-convex}.
Moreover, we have, for any $ x, y \in \mathbb{R}^d$,
\begin{equation}
\big\langle 
\nabla \log \tilde{p}_{T-t}(x) 
- \nabla \log \tilde{p}_{T-t}(y),
x - y
\big\rangle 
\leq 
-C_{T-t} 
\|x - y\|^2,
\end{equation}
where $C_t$ is defined as 
\begin{equation}
\label{eq:c-t-def}
C_t 
:= 
\frac{\alpha}{\alpha + (1 - \alpha)e^{-2t}} - 
\frac{e^{-2t}M}{(\alpha + (1 - \alpha)e^{-2t})^2} - 1.    
\end{equation}
\end{prop}
As a direct consequence of $\nabla \log \tilde{p}_{t}(x)
=
x+\nabla \log {p}_{t}(x)
$ for any $x\in \mathbb{R}^d$ and $t \in [0,T]$, 
the weak
log-concavity constant of the density $p_t$ is $C_t+1$.
The estimate for \( (C_t + 1) \) is coherent with the intuition we have on the SGMs.
Such a constant starts from
\(C_0+1=\alpha-M\), corresponding to
the weak log-concavity constant
of
\(\mu\), and
converges to the 
\[
C_t + 1
=
\frac{\alpha}{\alpha+(1-\alpha)e^{-2t}}
-
\frac{e^{-2t} M}{(\alpha+(1-\alpha)e^{-2t})^2}
\longrightarrow 1
\quad \text{as } t\to\infty,
\]
which is the log-concavity constant of the standard Gaussian law \(\gamma^d\).

In what follows, 
we turn to the regime-shift property of $p_t$, which underlies the contraction estimates used in the proof of Theorem \ref{thm:em} and Theorem \ref{thm:poisson}

\begin{prop}
[Regime shift of $p_t$]
\label{prop:regime-shift}
Let Assumption \ref{ass:weak-convex} hold. For any $t \in [0,T]$, it holds that 
\begin{equation}
\begin{cases} 
p_t \text{ is weakly log-concave, } & t \in \big[0, \tau(\alpha, M) \big] \\ 
p_t \text{ is strongly log-concave, } & t \in \big(\tau(\alpha, M), T\big],
\end{cases}
\end{equation}
with
\begin{equation}
\tau(\alpha, M)
:=
\begin{cases}
\log \Big(
\sqrt{\tfrac{\alpha^2+M-\alpha}{\alpha^2}}
\Big)
\wedge T,  
\quad & \text{if} \quad \alpha-M <0,\\
0,    & \text{otherwise.} 
\quad 
\end{cases}
\end{equation}
\end{prop}

{\color{black}
An interesting property
}
of Proposition \ref{prop:regime-shift} is that 
the law with density $p_t$
becomes strongly
log-concave after a finite amount of time, even if $p_0$ is only weakly log-concave. Indeed, when $\alpha-M<0$, we have that 
$\tfrac{\alpha^2+M-\alpha}{\alpha^2}>1$. Therefore, $\tau(\alpha, M)>0$ and two distinct regimes are present.
Intuitively, it
states that, if $\alpha-M\geq0$, i.e. if $p_0$ is strongly log-concave, then $p_t$ is guaranteed to remain strongly
log-concave in the whole interval $[0,T]$, which matches the property of \cite[Theorem 3.7]{Saumard2014Logconcavity}.

We introduce the time thresholds used to separate the contractive and non-contractive parts of the reverse dynamics. 
Denote by $T^\star$ the following time 
\begin{equation}
T^\star 
:= 
\inf
\big\{ t \in [0,T]:
2 \kappa_{-\log \tilde{p}_{T-t}}+1 \leq 0
\big\},
\end{equation}
defining $\inf \varnothing:=T$.
Also, denote by
\begin{equation}
\widehat{T}(\alpha, M, \rho):=\sup \left\{t \in[0, T]: 2 C_{T-t}+1 \geq \rho \right\},
\end{equation}
for $\rho \in[0,1)$, with $\inf \varnothing:=T$. 
Since the weak profile of $-\log \tilde{p_t}$ satisfies $\kappa_{- \log \tilde{p}_{t}}\geq C_t$, with $C_t$ being defined as in \eqref{eq:c-t-def}, we have that $\widehat{T}(\alpha, M, \rho) \leq T^\star$, for any $\rho \in[0,1)$. Next, we define 
\[
N_h
:=
\sup 
\big\{m \in\{0, \ldots, Nk\}: 
mh/k \leq \widehat{T}(\alpha, M, \rho)
= 
T-\eta(\alpha, M,\rho )\big\}, 
\]
with 
\begin{equation}
\begin{aligned}
\label{eq:eta-def}
&
    \eta(\alpha, M, \rho)\\
   & := 
    \begin{cases}
    \tfrac{1}{2} 
    \log 
    \Big(\tfrac{ M+ \rho\alpha(1-\alpha)}
    {(1-\rho)\alpha^2}
    +
    \sqrt{\tfrac{(1+\rho)(1-\alpha)^2}
    {(1-\rho)\alpha^2}
    +
    \big(\tfrac{M+\rho\alpha(1-\alpha)}
    {(1-\rho) \alpha^2}
    \big)^2}\Big) 
    \wedge
    T, 
    & \text { if } 2 \alpha-2 M-1<0 \\ 0, 
    & \text {otherwise.}
    \end{cases}
\end{aligned}
\end{equation}
Note that the condition
$$
\tfrac{ M+\rho\alpha(1-\alpha)}
{ (1-\rho)\alpha^2}
+
\sqrt{\tfrac{(1+\rho)(1-\alpha)^2}{(1-\rho)\alpha^2}
+
\big(\tfrac{ M+\rho\alpha(1-\alpha)}{(1-\rho)\alpha^2}\big)^2}>1
$$
{\color{black}
is equivalent to
}the condition $2 \alpha-2 M-1<0$. This means that $\widehat{T}(\alpha, M, \rho)$ is well-defined and $\widehat{T}(\alpha, M, \rho) \in[0, T]$. 
For clarity, we sketch the above regime shift of the drift $b_t$ of backward SDE \eqref{eq:backward-SDE} in the following Figure \ref{fig:contractive_interval} with the help of above preliminaries.

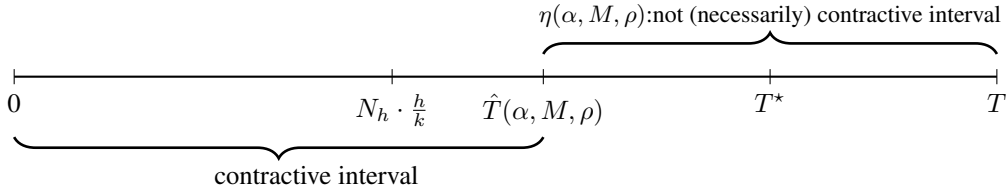
\begin{figure}[htbp]
    \centering
    \begin{tikzpicture}
        \draw[thick] (0,0) -- (13,0);
        
        \draw (0,0.1) -- (0,-0.1) node[below] {0};
        
        \draw (13,0.1) -- (13,-0.1) node[below] {$T$};

        \draw (5,0.1) -- (5,-0.1) node[below] {$N_h\cdot \tfrac{h}{k}$};
        
        \draw (7,0.1) -- (7,-0.1) node[below] {$\hat{T}(\alpha,M,\rho)$};
        
        \draw (10,0.1) -- (10,-0.1) node[below] {$T^\star$};
        
        \draw [decorate, decoration={brace, amplitude=8pt, mirror}, line width=1pt] 
            (0,-0.8) -- (7,-0.8);
        \node at (4,-1.3) {\text{contractive interval}};
        
        \draw [decorate, decoration={brace, amplitude=8pt}, line width=1pt] 
            (7,0.3) -- (13,0.3);
        \node at (10,0.8) {{\small$\eta(\alpha,M,\rho)$:\text{not (necessarily) contractive interval}}};
    \end{tikzpicture}
    \caption{Regime shift of the drift $b_{t}(x)$.}
    \label{fig:contractive_interval}
\end{figure}

\begin{lem}
[Lipschitz condition for the modified score]
\label{lem:score-lipschitz}
Let Assumptions \ref{ass:weak-convex} and \ref{ass:one-side-lipschitz} hold. Then, for any $t\in[0,T),$ it holds that 
\begin{equation}
    \sup_{x \in \mathbb{R}^d} 
    \|\nabla^2 \log \tilde{p}_{T-t}(x)\|
    \leq L_t 
    \leq L,
\end{equation}
where
\[
L_{t}
:= 
\Big( 
\Big(\tfrac{1}{1 - e^{-2(T-t)}} 
\wedge
e^{2(T-t)} L_V 
\Big)
\vee 
\big(
-(C_t + 1)
\big)
\Big)
+ 1, 
\quad 
L 
:= 
(1/\alpha \vee 1)(2 + L_V). 
\]
In particular, the modified score function \( x \mapsto \nabla \log \tilde{p}_t(x) \) is \( L \)-Lipschitz, i.e., for any $x, y \in \mathbb{R}^d$, 
\begin{equation}
\big\|
\nabla \log \tilde{p}_t(x) 
- 
\nabla \log \tilde{p}_t(y)
\big\| 
\leq 
L \|x - y\|. 
\end{equation}
\end{lem}
See \cite[Proposition B.4]{Silveri2025beyond} for a similar proof of this lemma. The following proposition establishes a regime shift result of the drift $b_{t}(x)$ of the time-reversal process.

\begin{prop}
[Regime shift of $b_{t}(x)$]
\label{prop:regime-shift-drift}
Let Assumptions \ref{ass:weak-convex} and \ref{ass:one-side-lipschitz} hold. It holds that, for any $x, y \in \mathbb{R}^d,$ and $t \in[0,T]$,
\begin{equation}
\big\langle
b_{t}(x)-b_{t}(y),x-y
\big\rangle 
\leq
G_{t}\|x-y\|^2, 
\end{equation}
with 
\[
G_{t} :=
\begin{cases}
-(2C_{T-t} +1), 
\quad  &  t \in  \big[0,\widehat{T}(\alpha, M, \rho)\big],\\
2L+1, &  t \in  \big(\widehat{T}(\alpha, M, \rho),T\big].
\end{cases}
\]
\end{prop}

\section{Bound on discretization error}

\subsection{Discretization error of PM method}
\label{app:B.1}
First we prove the following proposition, describing the error propagation between the PM auxiliary process \eqref{eq:poisson-aux} and the backward process SDE \eqref{eq:backward-SDE}.

\begin{prop}
\label{prop:Upper-mean-square error-bounds}
Suppose that Assumptions \ref{ass:weak-convex} and \ref{ass:one-side-lipschitz} hold. 
Let $\{\overleftarrow{X}_t\}_{t \in [0,T]}$ be the solution of SDE \eqref{eq:backward-SDE}. 
For $n \in [N-1]_0$ and $i\in[k-1]_0$, denote
\begin{equation}
\label{eq:defination-delta-w}
\Delta W_{t_{n,i}}:=W(t_{n,i+1})-W(t_{n,i}),
\qquad
\xi_{t_{n,i}}:=\sqrt{k/h}\Delta W_{t_{n,i}}.
\end{equation}
Define the local drift residual
\begin{equation}
\label{eq:defination-R_n}
R_{t_{n,i}}
:= 
\int_{t_{n,i}}^{t_{n,i+1}}
\big[
b_s(\overleftarrow{X}_s)
-
b_{t_{n,i}}(\overleftarrow{X}_{t_{n,i}})
\big]
\,
\mathrm{d} s,
\end{equation}
and the random perturbation
\begin{equation}
\label{eq:defination-M_n}
\begin{aligned}
M_{t_{n,i}}
& :=
-
\sqrt{\tfrac{2h}{k}} 
\bigg[
Z_{t_{n,i}} - \xi_{t_{n,i}} 
+ 
\sqrt{\tfrac{hk}{2}} 
\big( B_{t_{n,i}} - 1/k \big)
\Big( 
b_{t_{n,i}}(\tilde{X}_{t_{n,i}}^N)
-
b_{t_n}(X_{t_n}^N)
\Big) \\
& \quad \quad +
\sqrt{\tfrac{h}{2k}}
\Big( 
b_{t_{n,i}}(\tilde{X}_{t_{n,i}}^N)
-
b_{t_{n,i}}(X_{t_{n,i}}^N)
\Big)
\bigg]. 
\end{aligned}
\end{equation}
Then, for any $\lambda \in (0,1/4)$, we have
\begin{equation}
\begin{aligned}
\mathbb{E}
\Big[
\big\|\overleftarrow{X}_{t_{n,i+1}}
-
X^N_{t_{n,i+1}}
\big\|^2
\Big]
& \leq
\gamma_{t_{n,i}} 
\mathbb{\,E}
\Big[
\big\|
\overleftarrow{X}_{t_{n,i}}
-
X^N_{t_{n,i}}
\big\|^2
\Big]
+
\tfrac{k}
{\lambda h} 
\mathbb{E}
\Big[
\big\|
\mathbb{E}
\big[
R_{t_{n,i}}
\mid 
\mathcal{F}_{t_{n,i}}
\big]
\big\|^2
\Big]\\
& \quad
+
3
\,
\mathbb{E}
\big[
\|
R_{t_{n,i}}\|^2
\big]
+
3
\,
\mathbb{E}
\big[
\|
M_{t_{n,i}}\|^2
\big]
+
\tfrac{(2 L+1)^2h}{\lambda k}
\mathbb{E}
\Big[
\big\|
X_{t_{n,i}}^N
-
\tilde{X}_{t_{n,i}}^N
\big\|^2
\Big],
\end{aligned}
\end{equation}
where 
$L$ comes from Lemma \ref{lem:score-lipschitz} and
\begin{align*}
\gamma_{t_{n,i}}
: =
\begin{cases}
1+
\Big( 
\tfrac{3(2L+1)^2h}{k}-2(2 C_{T-{t_{n,i}}}+1)+2\lambda
\Big)
\tfrac{h}{k},
& \text { for } nk+i\leq N_h 
\\
1+
\big(
{\tfrac{3(2L+1)^2h}{k}}
+2(2L+1)+ 2\lambda\big)
\tfrac{h}{k},
& \text { for } nk+i> N_h. 
\end{cases}
\end{align*}

\end{prop}

\begin{proof}
For any 
$n \in [N], \, i \in [k-1]_0,$
it follows from the definition \eqref{eq:backward-SDE} that
\begin{equation}
\label{eq:rewritten-exact-X}
\overleftarrow{X}_{t_{n,i+1}}
=
\overleftarrow{X}_{t_{n,i}}
+
\tfrac{h}{k}
b_{t_{n,i}}(\overleftarrow{X}_{t_{n,i}})
+
\sqrt{2}
\Delta W_{t_{n,i}}
+
R_{t_{n,i}},
\end{equation}
where $\Delta W_{t_{n,i}}:=W(t_{n,i+1})-W(t_{n,i})$ and  
$R_{t_{n,i}}$ is defined as \eqref{eq:defination-R_n}.
Next, we rewrite \eqref{eq:poisson-aux} as follows:
\begin{equation}
\label{eq:x_0_x_T}
\begin{aligned}
X_{t_{n,i+1}}^N 
& = 
X_{t_{n,i}}^N
+ 
\tfrac{h}{k} 
b_{t_{n,i}}(X_{t_{n,i}}^N) 
+ 
\sqrt{2}
\Delta W_{t_{n,i}}\\
& \quad +
\sqrt{\tfrac{2h}{k}} 
\bigg[
Z_{t_{n,i}} - \xi_{t_{n,i}} 
+ 
\sqrt{\tfrac{hk}{2}} 
\big( B_{t_{n,i}} - 1/k \big)
\Big( 
b_{t_{n,i}}(\tilde{X}_{t_{n,i}}^N)
-
b_{t_n}(X_{t_n}^N)
\Big) \\
& \quad \quad +
\sqrt{\tfrac{h}{2k}}
\Big( 
b_{t_{n,i}}(\tilde{X}_{t_{n,i}}^N)
-
b_{t_{n,i}}(X_{t_{n,i}}^N)
\Big)
\bigg].
\end{aligned}
\end{equation}
Here, we denote $\sqrt{\frac{2h}{k}}\xi_{t_{n,i}}:= \sqrt{2}
\Delta W_{t_{n,i}} $ for sake of notation.
Subtracting \eqref{eq:x_0_x_T} from \eqref{eq:rewritten-exact-X} yields
\begin{equation}
\label{eq:x_n-y_n}
\begin{aligned}
\overleftarrow{X}_{t_{n,i+1}}
-
X^N_{t_{n,i+1}}
& =  
\overleftarrow{X}_{t_{n,i}}
-
X^N_{t_{n,i}}
+ 
\tfrac{h}{k}
\Big(
b_{t_{n,i}}
(\overleftarrow{X}_{t_{n,i}})
-
b_{t_{n,i}}
(X^N_{t_{n,i}})
\Big)
+
R_{t_{n,i}}\\
& \quad -
\sqrt{\tfrac{2h}{k}} 
\bigg[
Z_{t_{n,i}} - \xi_{t_{n,i}} 
+ 
\sqrt{\tfrac{hk}{2}} 
\big( B_{t_{n,i}} - 1/k \big)
\Big( 
b_{t_{n,i}}(\tilde{X}_{t_{n,i}}^N)
-
b_{t_n}(X_{t_n}^N)
\Big) \\
& \quad \quad +
\sqrt{\tfrac{h}{2k}}
\Big( 
b_{t_{n,i}}(\tilde{X}_{t_{n,i}}^N)
-
b_{t_{n,i}}(X_{t_{n,i}}^N)
\Big)
\bigg].
\end{aligned}
\end{equation}
For the sake of convenience in notation, we denote, for any $n \in [N]_0, 
\, i \in [k-1]_0$,
\begin{equation}
\varepsilon_{t_{n,i}}
:= 
\overleftarrow{X}_{t_{n,i}}-X^N_{t_{n,i}},
\end{equation}

and recall $M_{t_{n,i}}$ from \eqref{eq:defination-M_n}.
Squaring both sides of \eqref{eq:x_n-y_n} and taking expectations shows, 
\begin{equation}
\begin{aligned}
\label{eq:x_n-y_n-square-expectation}
\mathbb{E}
\big[
\|\varepsilon_{t_{n,i+1}}\|^2
\big]
& =
\mathbb{E}
\big[
\|\varepsilon_{t_{n,i}}\|^2
\big]
+ 
\tfrac{h^2}{k^2}
\mathbb{\,E}
\Big[
\big\|
b_{t_{n,i}}
(\overleftarrow{X}_{t_{n,i}})
-
b_{t_{n,i}}
(X^N_{t_{n,i}})
\big\|^2
\Big]
+
\mathbb{E}
\Big[
\big\|
R_{t_{n,i}}
\big\|^2
\Big]\\
& \quad
+
\mathbb{E}
\Big[
\big\|
M_{t_{n,i}}
\big\|^2
\Big]
+
\tfrac{2h}{k}
\mathbb{\,E} 
\Big[
\big  \langle
\varepsilon_{t_{n,i}},
b_{t_{n,i}}
(\overleftarrow{X}_{t_{n,i}})
-
b_{t_{n,i}}
(X^N_{t_{n,i}})
\big \rangle
\Big]
+
2
\mathbb{\,E} 
\Big[
\big  \langle
\varepsilon_{t_{n,i}},
R_{t_{n,i}}
\big \rangle
\Big]\\
& \quad
+
2
\,
\mathbb{\,E} 
\Big[
\big  \langle
\varepsilon_{t_{n,i}},
M_{t_{n,i}}
\big \rangle
\Big]
+
\tfrac{2h}{k}
\mathbb{\,E} 
\Big[
\big \langle
b_{t_{n,i}}
(\overleftarrow{X}_{t_{n,i}})
-
b_{t_{n,i}}
(X^N_{t_{n,i}}),
R_{t_{n,i}}
\big\rangle
                                                                                                                                                                                                                                                                                                                                                                                                                                                                                                                                                                                         \Big]
                                                                                                                                                                                                                                                                                                                                                                                                                                                                                                                                                                                        \\
& \quad +
\tfrac{2h}{k}
\mathbb{\,E} 
\Big[
\big \langle
b_{t_{n,i}}
(\overleftarrow{X}_{t_{n,i}})
-
b_{t_{n,i}}
(X^N_{t_{n,i}}),
M_{t_{n,i}}
\big\rangle
\Big]
+
2
\,
\mathbb{\,E} 
\Big[
\big \langle
R_{t_{n,i}},
M_{t_{n,i}}
\big\rangle
\Big].
\end{aligned}
\end{equation}
For the estimate of the fifth term of \eqref{eq:x_n-y_n-square-expectation}, using Proposition \ref{prop:regime-shift-drift} yields,
\begin{equation}
\label{eq:cross-term-est}
\tfrac{2h}{k}
\mathbb{\,E} 
\Big[
\big  \langle
\varepsilon_{t_{n,i}},
b_{t_{n,i}}
(\overleftarrow{X}_{t_{n,i}})
-
b_{t_{n,i}}
(X^N_{t_{n,i}})
\big \rangle
\Big]
\leq 
\begin{cases}
\tfrac{-2
(2C_{T-{t_{n,i}}}+1)
h}
{k}
\,
\mathbb{E}
\big[
\|\varepsilon_{t_{n,i}}\|^2
\big], 
\quad & \text{for} \quad nk+i  \leq N_h,
\\[5pt]
\tfrac{2
(2L+1)
h}{k}
\,
\mathbb{E}
\big[
\|\varepsilon_{t_{n,i}}\|^2
\big],
& \text{for}\quad  nk+i > N_h.
\end{cases}
\end{equation}
Now, we deal with $2
\mathbb{\,E} 
\big[
\langle
\varepsilon_{t_{n,i}},
R_{t_{n,i}}
\rangle
\big].$
In particular, if we let $\mathcal{F}_{t_{n,i}}$ denote the $\sigma$-algebra generated by $\{\overleftarrow{X}_{t_{n,i}}\},\,\{\tilde{X}^N_{t_{n,i}}\},$ and $\{X^N_{t_{n,i}}\}$ for any $n\in [N]_0$ and $i\in[k-1]_0$.
Since $\varepsilon_{t_{n,i}}$ is $\mathcal{F}_{t_{n,i}}$-measurable,
we apply the Young inequality with $0<\lambda<1/4$ to obtain that
\begin{equation}
\begin{aligned}
2
\mathbb{\,E} 
\Big[
\big  \langle
\varepsilon_{t_{n,i}},
R_{t_{n,i}}
\big \rangle
\Big]
& =
2 
        \mathbb{\,E}
        \bigg[
        \Big\langle
        \varepsilon_{t_{n,i}}, 
        \mathbb{E}
        \big[
        R_{t_{n,i}}
        \big|
        \mathcal{F}_{t_{n,i}}
        \big]
        \Big\rangle
        \bigg]\\
        & \leq 
        \tfrac{\lambda h}{k}
        \mathbb{\,E}
        \left[
        \left\|
        \varepsilon_{t_{n,i}}
        \right\|^2
        \right]
        +
        \tfrac{k}
        {\lambda h} 
        \mathbb{E}
        \Big[
        \big\|
        \mathbb{E}
        \big[
        R_{t_{n,i}}
        \mid 
        \mathcal{F}_{t_{n,i}}
        \big]
        \big\|^2
        \Big].
\end{aligned}
\end{equation}
For the term involving $M_{t_{n,i}}$, conditioning on $\mathcal{F}_{t_{n,i}}$ gives
\begin{equation}
\begin{aligned}
\mathbb{E}
\Big[
Z_{t_{n,i}}
-
\xi_{t_{n,i}}
+
\sqrt{\tfrac{hk}{2}}
\big(
B_{t_{n,i}}-1/k
\big)
\Big( 
b_{t_{n,i}}(\tilde{X}_{t_{n,i}}^N)
-
b_{t_n}(X_{t_n}^N)
\Big)
\Big|
\mathcal{F}_{t_{n,i}}
\Big]
=0,
\end{aligned}
\end{equation}
where we use that $Z_{t_{n,i}}$, $\xi_{t_{n,i}}$, and $B_{t_{n,i}}$ are independent of $\mathcal{F}_{t_{n,i}}$.
Hence, for $2
\,
\mathbb{\,E} 
\big[
\langle
\varepsilon_{t_{n,i}},
M_{t_{n,i}}
\rangle
\big]$, the Young inequality and the Lipschitz condition of $b_t(\cdot)$ imply
\begin{equation}
\begin{aligned}
&
2
\,
\mathbb{\,E} 
\Big[
\big  \langle
\varepsilon_{t_{n,i}},
M_{t_{n,i}}
\big \rangle
\Big]
 = 
2 
        \mathbb{\,E}
        \bigg[
        \Big\langle
        \varepsilon_{t_{n,i}}, 
        \mathbb{E}
        \big[
        M_{t_{n,i}}
        \big|
        \mathcal{F}_{t_{n,i}}
        \big]
        \Big\rangle
        \bigg]\\
        & =
        -2
        \sqrt{\tfrac{2h}{k}}
        \mathbb{\,E}
        \bigg[
        \Big\langle
        \varepsilon_{t_{n,i}}, 
        \mathbb{E}
        \Big[
        Z_{t_{n,i}}
        -
        \xi_{t_{n,i}}
        +
\sqrt{\tfrac{hk}{2}}
\big(
B_{t_{n,i}}-1/k
\big)
\Big( 
b_{t_{n,i}}(\tilde{X}_{t_{n,i}}^N)
-
b_{t_n}(X_{t_n}^N)
\Big)
\Big]
        \Big|
        \mathcal{F}_{t_{n,i}}
        \Big]
        \Big\rangle
        \bigg]\\
& \quad +
\tfrac{2h}{k}
        \mathbb{\,E}
        \Big[
        \big\langle
        \varepsilon_{t_{n,i}}, 
b_{t_{n,i}}(X_{t_{n,i}}^N)
-
b_{t_{n,i}}(\tilde{X}_{t_{n,i}}^N)
        \big\rangle
        \Big]\\
         & \leq
\tfrac{\lambda h}{k}
        \,
        \mathbb{E}
        \big[
        \|
        \varepsilon_{t_{n,i}}\|^2
        \big]
        +
        \tfrac{(2 L+1)^2h}{\lambda k}
        \mathbb{E}
        \Big[
        \big\|
        X_{t_{n,i}}^N
         -
        \tilde{X}_{t_{n,i}}^N
        \big\|^2
        \Big].
    \end{aligned}
\end{equation} 
Regarding the last three terms of \eqref{eq:x_n-y_n-square-expectation}, using the Young inequality and Lipschitz condition of 
$b_t(\cdot)$ leads to 
\begin{equation}
\begin{aligned}
\tfrac{2h}{k}
\mathbb{\,E} 
\Big[
\big \langle
b_{t_{n,i}}
(\overleftarrow{X}_{t_{n,i}})
-
b_{t_{n,i}}
(X^N_{t_{n,i}}),
R_{t_{n,i}}
\big\rangle
\Big] 
& \leq 
        \tfrac{(2 L+1)^2
        h^2}{k^2}
         \mathbb{\,E}
         \left[
         \left\|
         \varepsilon_{t_{n,i}}
         \right\|^2
         \right]
         +
         \mathbb{E}
         \Big[
         \big\|
         R_{t_{n,i}}
        \big\|^2
        \Big],\\
\tfrac{2h}{k}
\mathbb{\,E} 
\Big[
\big \langle
b_{t_{n,i}}
(\overleftarrow{X}_{t_{n,i}})
-
b_{t_{n,i}}
(X^N_{t_{n,i}}),
M_{t_{n,i}}
\big\rangle
\Big] 
& \leq 
        \tfrac{(2 L+1)^2
        h^2}{k^2}
         \mathbb{\,E}
         \left[
         \left\|
         \varepsilon_{t_{n,i}}
         \right\|^2
         \right]
         +
         \mathbb{E}
         \Big[
         \big\|
         M_{t_{n,i}}
        \big\|^2
        \Big],
\end{aligned}
\end{equation}
and 
\begin{equation}
\begin{aligned} 
2\,
\mathbb{\,E} 
\Big[
\big \langle
R_{t_{n,i}},
M_{t_{n,i}}
\big\rangle
\Big]
\leq 
\mathbb{E}
\Big[
\big\|
R_{t_{n,i}}
\big\|^2
\Big]
+
\mathbb{E}
\Big[
\big\|
M_{t_{n,i}}
\big\|^2
\Big].
\end{aligned}
\end{equation}
Plugging the above estimates into \eqref{eq:x_n-y_n-square-expectation} yields
\begin{equation}
\label{eq:iteration_error}
\begin{aligned}
\mathbb{E}
\big[
\|\varepsilon_{t_{n,i+1}}
\|^2
\big]
& \leq
\gamma_{t_{n,i}} 
\mathbb{\,E}
\big[
\|\varepsilon_{t_{n,i}}\|^2
\big]
+
\tfrac{k}
{\lambda h} 
\mathbb{E}
\Big[
\big\|
\mathbb{E}
\big[
R_{t_{n,i}}
\mid 
\mathcal{F}_{t_{n,i}}
\big]
\big\|^2
\Big] 
+
3
\,
\mathbb{E}
\Big[
\big\|
R_{t_{n,i}}
\big\|^2
\Big]\\
& \quad +
3
\,
\mathbb{E}
\Big[
\big\|
M_{t_{n,i}}
\big\|^2
\Big]
+
\tfrac{(2 L+1)^2h}{\lambda k}
        \mathbb{E}
        \Big[
        \big\|
        X_{t_{n,i}}^N
         -
        \tilde{X}_{t_{n,i}}^N
        \big\|^2
        \Big],
\end{aligned}
\end{equation}
where 
\begin{align*}
\gamma_{t_{n,i}}
: =
\begin{cases}
1+
\big( 
\tfrac{3(2L+1)^2h}{k}-2(2 C_{T-{t_{n,i}}}+1)+2\lambda
\big)
\tfrac{h}{k},
& \text { for } nk+i\leq N_h 
\\[5pt]
1+
\big(
\tfrac{3(2L+1)^2h}{k}
+2(2L+1)+ 2\lambda\big)
\tfrac{h}{k},
& \text { for } nk+i> N_h. 
\end{cases}
\end{align*}
The proof is completed.
\end{proof}


The remainder of this subsection collects the auxiliary estimates needed to iterate Proposition \ref{prop:Upper-mean-square error-bounds}, starting with moment bounds of the processes $\{\overleftarrow{X}_t\}_{t \in [0, T]}$ and $\{X^N_{t_{n,i}}\}_{n \in [N]_0, i \in [k-1]_0}$ defined by \eqref{eq:backward-SDE} and \eqref{eq:poisson-aux}.

\begin{lem}
\label{lem:target-p-order-moment}
Let Assumption \ref{ass:weak-convex} hold. Then $\mu$ admits a finite $2p$-th order moment for any $p \in [1,2]$. That is, there exists a constant $\mathcal{M}$, independent of $d$ and $t$, satisfying
\begin{equation}
\E
\big[\|x_0\|^{2p}\big]
\leq
\mathcal{M}d^{p}.
\end{equation}
Moreover, the process $\big\{\overleftarrow{X}_t\big\}_{ t \in [0, T]}$ defined by \eqref{eq:backward-SDE} satisfies, for any $ t \in [0,T]$, 
\begin{equation}
\E
\Big[
\big\|
\overleftarrow{X}_t
\big\|^{2p}
\Big]
\leq
\mathcal{M}_1 d^p,
\end{equation}
where $\mathcal{M}_1$ is independent of $d$ and $t$.
\end{lem}


The case $p=1$ is established
in~\cite[Proposition~B.1 and Lemma~D.1]{Silveri2025beyond}; the
extension to $p\in[1,2]$ follows from Assumption~\ref{ass:weak-convex}
by an analogous argument and is therefore omitted. We mentioned that $\gamma_{t_{n,i}}$ and $\Gamma_{t_{n,i}}$ are generic finite constants whose values may vary across different propositions, but will not depend upon the dimension $d$.

\begin{lem}
\label{lem:numerical-p-order-moment}
Let Assumptions \ref{ass:weak-convex}, \ref{ass:one-side-lipschitz} and \ref{ass:lipschitz-time} hold. 
The process $\{{X}^N_{t_{n}}\}_{n\in [N]_0}$ generated by \eqref{eq:poisson-aux} using the uniform step-size. For any $\lambda\in (0,1/4)$, let the step-size satisfy
\begin{equation}
\label{eq:h-condition-numerical-bound}
h \leq
\tfrac{2(1-\lambda)}
{C(L, \tilde{L})},
\end{equation}
where $C(L,  \tilde{L})$ is given by \eqref{eq:numerical-bounded-constant-c}.
Then, there exists a constant $\mathcal{M}_2$, 
satisfying, for any $n \in [N]_0$, 
\begin{equation}
\E
\Big[
\big\|
{X}^N_{t_{n}}
\big\|^4
\Big]
\leq
\mathcal{M}_2 
T^5 d^2,
\end{equation}
where $\mathcal{M}_2:=C\mathrm{e}^
{C
\eta(\alpha, M,C(L, \tilde{L})h/2+\lambda)}$, and $\eta(\alpha, M,C(L, \tilde{L})h/2+\lambda)$ as defined in \eqref{eq:eta-def}, and $C$ is a constant independent of $d$ and $h$.
\end{lem}

\begin{proof}
Recalling \eqref{eq:poisson-aux}, we have, for any $n \in [N-1]_0$,
\begin{equation}
\begin{aligned}
{X}_{t_{n+1}}^N
& =
{X}_{t_n}^N
+
h
b_{t_n}({X}_{t_{n}}^N)
+
h 
\sum_{j=0}^{k-1}
B_{t_{n,j}}
\Big(
b_{t_{n,j}}
(\tilde{X}_{t_{n,j}}^N)
-
b_{t_n}({X}_{t_{n}}^N)
\Big)
+
\sqrt{\tfrac{2h}{k}}
\sum_{j=0}^{k-1}
Z_{t_{n,j}}.
\end{aligned}
\end{equation}
Taking square on both sides shows
\begin{equation}
\begin{aligned}
&
\big\|
{X}_{t_{n+1}}^N
\big\|^2\\
& =
\big\|
{X}_{t_n}^N
\big\|^2
+
h^2 
\big\|
b_{t_n}({X}_{t_{n}}^N)
\big\|^2
+
h^2
\Big\|
\sum_{j=0}^{k-1}
B_{t_{n,j}}
\Big(
b_{t_{n,j}}
(\tilde{X}_{t_{n,j}}^N)
-
b_{t_n}({X}_{t_{n}}^N)
\Big)
\Big\|^2
+
\tfrac{2h}{k}
\Big\|
\sum_{j=0}^{k-1}
Z_{t_{n,j}}
\Big\|^2\\
& \quad 
+
2h 
\Big \langle
{X}_{t_n}^N,
b_{t_n}({X}_{t_{n}}^N)
\Big \rangle 
+
2h 
\Big \langle
{X}_{t_n}^N,
\sum_{j=0}^{k-1}
B_{t_{n,j}}
\Big(
b_{t_{n,j}}
(\tilde{X}_{t_{n,j}}^N)
-
b_{t_n}({X}_{t_{n}}^N)
\Big)
\Big \rangle 
\\
& \quad +
2
\sqrt{\tfrac{2h}{k}}
\Big \langle
{X}_{t_n}^N,
\sum_{j=0}^{k-1}
Z_{t_{n,j}}
\Big \rangle 
+
2h^2 
\Big \langle
b_{t_n}({X}_{t_{n}}^N),
\sum_{j=0}^{k-1}
B_{t_{n,j}}
\Big(
b_{t_{n,j}}
(\tilde{X}_{t_{n,j}}^N)
-
b_{t_n}({X}_{t_{n}}^N)
\Big)
\Big \rangle\\
& \quad
+
2h
\sqrt{\tfrac{2h}{k}}
\Big \langle
b_{t_n}({X}_{t_{n}}^N),
\sum_{j=0}^{k-1}
Z_{t_{n,j}}
\Big \rangle
+
2h 
\sqrt{\tfrac{2h}{k}}
\Big \langle
\sum_{j=0}^{k-1}
B_{t_{n,j}}
\Big(
b_{t_{n,j}}
(\tilde{X}_{t_{n,j}}^N)
-
b_{t_n}({X}_{t_{n}}^N)
\Big),
\sum_{j=0}^{k-1}
Z_{t_{n,j}}
\Big \rangle.
\end{aligned}
\end{equation}
Next using the Young inequality gives 
\begin{equation}
\label{eq:numer-square}
\begin{aligned}
&
\big\|
{X}_{t_{n+1}}^N
\big\|^2 \\
& \leq 
\big\|
{X}_{t_n}^N
\big\|^2
+
3h^2 
\big\|
b_{t_n}({X}_{t_{n}}^N)
\big\|^2
+
3h^2
\Big\|
\sum_{j=0}^{k-1}
B_{t_{n,j}}
\Big(
b_{t_{n,j}}
(\tilde{X}_{t_{n,j}}^N)
-
b_{t_n}({X}_{t_{n}}^N)
\Big)
\Big\|^2
+
\tfrac{6h}{k}
\Big\|
\sum_{j=0}^{k-1}
Z_{t_{n,j}}
\Big\|^2\\
& \quad
+
2h 
\Big \langle
{X}_{t_n}^N,
b_{t_n}({X}_{t_{n}}^N)
\Big \rangle 
+
2h 
\Big \langle
{X}_{t_n}^N,
\sum_{j=0}^{k-1}
B_{t_{n,j}}
\Big(
b_{t_{n,j}}
(\tilde{X}_{t_{n,j}}^N)
-
b_{t_n}({X}_{t_{n}}^N)
\Big)
\Big \rangle \\
& \quad
+
2
\sqrt{\tfrac{2h}{k}}
\Big \langle
{X}_{t_n}^N,
\sum_{j=0}^{k-1}
Z_{t_{n,j}}
\Big \rangle.
\end{aligned}
\end{equation}
Before proceeding further, by the definition of $\tilde{X}_{t_{n,j}}^N$ in \eqref{eq:poisson-aux}, Assumption \ref{ass:lipschitz-time} and the triangle inequality, we have, for any $j \in [k-1]_0$,
\begin{equation}
\label{eq:numer-b-error}
\begin{aligned}
& 
\big\|
b_{t_{n,j}}
(\tilde{X}_{t_{n,j}}^N)
-
b_{t_n}({X}_{t_{n}}^N)
\big\| \\
& \leq 
(2L+1) 
\big\|
\tilde{X}_{t_{n,j}}^N
-
{X}_{t_{n}}^N
\big\|
+
\tilde{L} h
\Big(
\big\|
{X}_{t_{n}}^N
\big\|
+
1
\Big)\\
& \leq 
(2L+1) 
\Big\|
\tfrac{jh}{k}
b_{t_n}(X_{t_n}^N)
+
\sum_{l=0}^{j-1}
\sqrt{\tfrac{2h}{k}} 
Z_{t_{n,l}}
\Big\|
+
\tilde{L} h
\Big(
\big\|
{X}_{t_{n}}^N
\big\|
+
1
\Big)\\
& \leq 
\big((2L+1)^2 +\tilde{L}\big)
h
\big\|
X_{t_n}^N
\big\|
+
(2L+1)
\sqrt{\tfrac{2h}{k}} 
\Big\|
\sum_{l=0}^{j-1}
Z_{t_{n,l}}
\Big\|
+
(2L+1) 
h
\big\|
b_{t_{n}}(0)
\big\|
+
\tilde{L} h, 
\end{aligned}
\end{equation}
where for short we denote 
\begin{equation}
\label{eq:notation-short-I}
\mathcal{I}_{t_n}:=
(2L+1)
\sqrt{\tfrac{2h}{k}} 
\Big\|
\sum_{l=0}^{j-1}
Z_{t_{n,l}}
\Big\|
+
(2L+1)h 
\big\|
b_{t_{n}}(0)
\big\|
+
\tilde{L} h.
\end{equation}
Thanks to the above estimate, the short hand notation \eqref{eq:notation-short-I}, the Cauchy-Schwarz inequality and the Young inequality, one can derive 
\begin{equation}
\begin{aligned}
& 
2h 
\Big \langle
{X}_{t_n}^N,
\sum_{j=0}^{k-1}
B_{t_{n,j}}
\Big(
b_{t_{n,j}}
(\tilde{X}_{t_{n,j}}^N)
-
b_{t_n}({X}_{t_{n}}^N)
\Big)
\Big \rangle \\
& \leq 
2h 
\sum_{j=0}^{k-1}
B_{t_{n,j}}
\big\|
{X}_{t_n}^N
\big\|
\cdot 
\Big\|
b_{t_{n,j}}
(\tilde{X}_{t_{n,j}}^N)
-
b_{t_n}({X}_{t_{n}}^N)
\Big\|\\
& \leq 
\big(
2\big((2L+1)^2 +\tilde{L}\big)
+
1
\big)
h^2
\sum_{j=0}^{k-1}
B_{t_{n,j}}
\big\|
{X}_{t_n}^N
\big\|^2
+ 
\sum_{j=0}^{k-1}
B_{t_{n,j}}
\mathcal{I}^2_{t_n}.
\end{aligned}
\end{equation}
Regarding the third term of \eqref{eq:numer-square}, in view of \eqref{eq:numer-b-error}, one can obtain 
\begin{equation}
\begin{aligned}
&
\Big\|
\sum_{j=0}^{k-1}
B_{t_{n,j}}
\Big(
b_{t_{n,j}}
(\tilde{X}_{t_{n,j}}^N)
-
b_{t_n}({X}_{t_{n}}^N)
\Big)
\Big\|^2 \\
& \leq 
\Big(
\sum_{j=0}^{k-1}
B_{t_{n,j}}
\big\|
b_{t_{n,j}}
(\tilde{X}_{t_{n,j}}^N)
-
b_{t_n}({X}_{t_{n}}^N)
\big\|
\Big)^2\\
& \leq 
\Big(
\big((2L+1)^2 +\tilde{L}\big)
h
\sum_{j=0}^{k-1}
B_{t_{n,j}}
\big\|
X_{t_n}^N
\big\|
+
\sum_{j=0}^{k-1}
B_{t_{n,j}}
\mathcal{I}_{t_n}
\Big)^2\\
& \leq 
2\big((2L+1)^2 +\tilde{L}\big)^2
h^2
\Big(\sum_{j=0}^{k-1}
B_{t_{n,j}}
\Big)^2
\big\|
X_{t_n}^N
\big\|^2
+
2
\Big(\sum_{j=0}^{k-1}
B_{t_{n,j}}
\Big)^2
\mathcal{I}_{t_n}^2.
\end{aligned}
\end{equation}
Moreover, we apply the fact that $\| b_t(x) \| \leq (2L+1)\|x\|+\|b_{t}(0)\| $ to arrive at
\begin{equation}
\big\|
b_{t_n}({X}_{t_{n}}^N)
\big\|^2 
\leq 
2(2L+1)^2
\big\|{X}_{t_n}^N\big \|^2
+
2
\big\|b_{t_n}(0)\big\|^2.
\end{equation}
Now, we estimate $\big \langle
{X}_{t_n}^N,
b_{t_n}({X}_{t_{n}}^N)
\big \rangle $ for two cases: $nk\leq N_h$ and $nk >N_h $. For $nk \leq N_h$, Proposition \ref{prop:regime-shift-drift}, in conjunction with the Young inequality with $\lambda \in (0, 1/4)$, directly yields
\begin{equation}
2h 
\Big \langle
{X}_{t_n}^N,
b_{t_n}({X}_{t_{n}}^N)
\Big \rangle  
\leq 
-2h (2C_{T-t_n} +1-\lambda /2)
\big\|{X}_{t_n}^N\big \|^2
+
\tfrac{h}{\lambda}
\big\|b_{t_n}(0)\big\|^2.
\end{equation}
In the case $nk >N_h $, we apply Proposition \ref{prop:regime-shift-drift} and the Young inequality to get
\begin{equation}
2h 
\Big \langle
{X}_{t_n}^N,
b_{t_n}({X}_{t_{n}}^N)
\Big \rangle  
\leq 
(2L+3/2)h
\big\|{X}_{t_n}^N\big \|^2
+ 
h
\big\|b_{t_n}(0)\big\|^2.
\end{equation}
{\color{black}
Due to \eqref{eq:h-condition-numerical-bound},
\(B_{t_{n,j}}^2=B_{t_{n,j}}\) and \(\mathcal{I}_{t_{n,j}}\geq0\),
we have
\[
6h^2
\Big(
\sum_{j=0}^{k-1}
B_{t_{n,j}}\mathcal{I}_{t_{n,j}}
\Big)^2
+
\sum_{j=0}^{k-1}
B_{t_{n,j}}\mathcal{I}_{t_{n,j}}^2
\leq
(1+6h^2)
\Big(
\sum_{j=0}^{k-1}
B_{t_{n,j}}\mathcal{I}_{t_{n,j}}
\Big)^2
\leq
7
\Big(
\sum_{j=0}^{k-1}
B_{t_{n,j}}\mathcal{I}_{t_{n,j}}
\Big)^2 .
\]
}
Equipped with the above estimates, and using the fact that $\sum_{j=0}^{k-1}
B_{t_{n,j}}
\le
\big(
\sum_{j=0}^{k-1}
B_{t_{n,j}}
\big)^{2}$, one can derive from \eqref{eq:numer-square} that
\begin{equation}
\begin{aligned}
&
1+
\big\|
{X}_{t_{n+1}}^N
\big\|^2\\
& \leq 
\gamma_{t_n}
\Big(
1+
\big\|
{X}_{t_n}^N
\big\|^2
\Big)
+
8
\Big(
2\big((2L+1)^2 +\tilde{L}\big)^2
+1
\Big)
\Big(\sum_{j=0}^{k-1}
B_{t_{n,j}}
\Big)^2
h^2
\big\|
{X}_{t_n}^N
\big\|^2\\
& \quad 
+ 
\big( 
6+C(\lambda)
\big)
h 
\big\|b_{t_n}(0)\big\|^2
+ 
{\color{black}
7
\Big(
\sum_{j=0}^{k-1}
B_{t_{n,j}}\mathcal{I}_{t_{n,j}}
\Big)^2}
+
\tfrac{6h}{k}
\Big\|
\sum_{j=0}^{k-1}
Z_{t_{n,j}}
\Big\|^2
+
2
\sqrt{\tfrac{2h}{k}}
\Big \langle
{X}_{t_n}^N,
\sum_{j=0}^{k-1}
Z_{t_{n,j}}
\Big \rangle
\\
& =: 
\gamma_{t_n}
\Big(
1+
\big\|
{X}_{t_n}^N
\big\|^2
\Big)
+
8
\Big(
2\big((2L+1)^2 +\tilde{L}\big)^2
+1
\Big)
\Big(\sum_{j=0}^{k-1}
B_{t_{n,j}}
\Big)^2
h^2
\big\|
{X}_{t_n}^N
\big\|^2
+
\Xi_{t_{n+1}},
\end{aligned}
\end{equation}
where
\begin{equation}
\begin{aligned}
\Xi_{t_{n+1}}
& :=
\big( 6+C(\lambda)
\big)
h 
\big\|b_{t_n}(0)\big\|^2
+ 
{\color{black}
7
\Big(
\sum_{j=0}^{k-1}
B_{t_{n,j}}\mathcal{I}_{t_{n,j}}
\Big)^2}
+
\tfrac{6h}{k}
\Big\|
\sum_{j=0}^{k-1}
Z_{t_{n,j}}
\Big\|^2\\
& \quad 
+
2
\sqrt{\tfrac{2h}{k}}
\Big \langle
{X}_{t_n}^N,
\sum_{j=0}^{k-1}
{\color{black}Z_{t_{n,j}}}
\Big \rangle
\end{aligned}
\end{equation}
and 
\begin{align*}
\gamma_{t_{n}}
: =
\begin{cases}
1+
\big(
6(2L+1)^2 h
-
2  (2C_{T-t_n} +1)
+ \lambda
\big) 
h,
& \text { for } nk\leq N_h 
\\[5pt]
1+
\big(
6(2L+1)^2+2L+3/2
\big)
h,
& \text { for } nk> N_h.
\end{cases}
\end{align*}
Furthermore, we have 
\begin{equation}
\label{eq:numer-2-square}
\begin{aligned}
&
1+
\big\|
{X}_{t_{n+1}}^N
\big\|^2 \\
& =
\gamma_{t_n}
\Big(
1+
\big\|
{X}_{t_n}^N
\big\|^2
\Big)
+
8
\Big(
2\big((2L+1)^2 +\tilde{L}\big)^2
+1
\Big)
\Big(\sum_{j=0}^{k-1}
B_{t_{n,j}}
\Big)^2 
h^2
\big\|
{X}_{t_n}^N
\big\|^2
+
\Xi_{t_{n+1}}\\
& \leq 
\gamma_{t_n}
\Big(
1+
\big\|
{X}_{t_n}^N
\big\|^2
\Big)
\bigg(
1+
8
\Big(
2\big((2L+1)^2 +\tilde{L}\big)^2
+1
\Big)
\Big(\sum_{j=0}^{k-1}
B_{t_{n,j}}
\Big)^2 
h^2
/\gamma_{t_n}
+
\tfrac{\Xi_{t_{n+1}}}
{\gamma_{t_n}
\big(
1+
\|
{X}_{t_n}^N
\|^2
\big)}
\bigg)\\
& =:
\gamma_{t_n}
\Big(
1+
\big\|
{X}_{t_n}^N
\big\|^2
\Big)
\big(
1+
\mathbb{I}_1
+
\mathbb{I}_2
\big).
\end{aligned}
\end{equation}
Taking square and then expectations of \eqref{eq:numer-2-square} to deduce
\begin{equation}
\begin{aligned}
\E
\bigg[
\Big(1+
\big\|
{X}_{t_{n+1}}^N
\big\|^2 
\Big)^2
\Bigg]
& =
\gamma_{t_n}^2
\E
\bigg[
\Big(
1+
\big\|
{X}_{t_n}^N
\big\|^2
\Big)^2
\big(
1+
\mathbb{I}_1
+
\mathbb{I}_2
\big)^2
\bigg].
\end{aligned}
\end{equation}
With regard to the right term of above equality, we first notice that ${X}_{t_n}^N$ is $\mathcal{F}_{t_n}$-measurable. By further taking conditional expectation with
respect to $\mathcal{F}_{t_n}$, one can see that 
\begin{equation}
\begin{aligned}
\E
\bigg[
\Big(
1+
\big\|
{X}_{t_n}^N
\big\|^2
\Big)^2
\big(
1+
\mathbb{I}_1
+
\mathbb{I}_2
\big)^2
\bigg]
& =
\E
\bigg[
\Big(
1+
\big\|
{X}_{t_n}^N
\big\|^2
\Big)^2
\big(
1+
\mathbb{I}_1
+
\mathbb{I}_2
\big)^2
\big|
\mathcal{F}_{t_n}
\Big]
\bigg]\\
& =
\E
\bigg[
\Big(
1+
\big\|
{X}_{t_n}^N
\big\|^2
\Big)^2
\E 
\Big[
\big(
1+
\mathbb{I}_1
+
\mathbb{I}_2
\big)^2
\big|
\mathcal{F}_{t_n}
\Big]
\bigg].
\end{aligned}
\end{equation}
Next, 
expanding the square and applying the Young inequality to the cross term
\(2\mathbb{I}_1\mathbb{I}_2\) show
\begin{equation}
\begin{aligned}
\E 
\Big[
\big(
1+
\mathbb{I}_1
+
\mathbb{I}_2
\big)^2
\big|
\mathcal{F}_{t_n}
\Big]
& \leq
1
+
2\,
\E 
\Big[
\mathbb{I}^2_1
\big|
\mathcal{F}_{t_n}
\Big]
+
2\,
\E 
\Big[
\mathbb{I}^2_2
\big|
\mathcal{F}_{t_n}
\Big]
+
2\,
\E 
\Big[
\mathbb{I}_1
\big|
\mathcal{F}_{t_n}
\Big]
+
2\,
\E 
\Big[
\mathbb{I}_2
\big|
\mathcal{F}_{t_n}
\Big].
\end{aligned}
\end{equation}
Since $\{B_{t_{n,j}}\}_{j \in [k-1]_0}$ are independent of $\mathcal{F}_{t_{n}}$, and given that 
$\E
\big[
\big(\sum_{j=0}^{k-1}
B_{t_{n,j}}
\big)^2
\big] 
\leq
2 ,\,\E
\big[
\big(\sum_{j=0}^{k-1}
B_{t_{n,j}}
\big)^4
\big] 
\leq
15$, 
one can derive that
\begin{equation}
\begin{aligned}
\E 
\Big[
\mathbb{I}_1
\big|
\mathcal{F}_{t_n}
\Big]
&\leq 
16
\Big(
2\big((2L+1)^2 +\tilde{L}\big)^2
+1
\Big)h^2
/\gamma_{t_n}, 
\\
\E 
\Big[
\mathbb{I}^2_1
\big|
\mathcal{F}_{t_n}
\Big]
&\leq 
{\color{black}960}
\Big(
2\big((2L+1)^2 +\tilde{L}\big)^2
+1
\Big)^2
h^4
/
{\color{black}\gamma^2_{t_n}}.
\end{aligned}
\end{equation}
Next, we recall some properties of the standard Gaussians: for any $n \in  [N]_0, j \in[k-1]_0$, 
\begin{equation}
\E \bigg[
\Big\|
\sum_{l=0}^{j-1}
Z_{t_{n,l}}
\Big\|^2
\big|
\mathcal{F}_{t_n}
\bigg] 
= 
\sum_{l=0}^{j-1}
\E \Big[
\big\|
Z_{t_{n,l}}
\big\|^2
\big|
\mathcal{F}_{t_n}
\Big] 
\leq 
kd, 
\quad 
\E \bigg[
\Big\|
\sum_{l=0}^{j-1}
Z_{t_{n,l}}
\Big\|^4
\big|
\mathcal{F}_{t_n}
\bigg]  
\leq 
k^2d^2,
\end{equation}
and 
\begin{equation}
\E \Big[
\sum_{l=0}^{j-1}
Z_{t_{n,l}}
\big|
\mathcal{F}_{t_n}
\Big]
=0.
\end{equation}
As a result, 
\begin{equation}
\begin{aligned}
\E 
\big[
\mathcal{I}_{t_n}^2
\big|
\mathcal{F}_{t_n}
\big]
& \leq 
\tfrac{6(2L+1)^2h}{k}
\,
\E 
\bigg[
\Big\|
\sum_{l=0}^{j-1}
Z_{t_{n,l}}
\Big\|^2
\big|
\mathcal{F}_{t_n}
\bigg]
+
3(2L+1)h^2
\big\|
b_{t_{n}}(0)
\big\|^2
+
3\tilde{L}^2 h^2\\
& \leq 
6(2L+1)^2d h
+
3(2L+1)^2h^2
\big\|
b_{t_{n}}(0)
\big\|^2
+
3\tilde{L}^2 h^2.
\end{aligned}
\end{equation}
Using the same technique, we have 
\begin{equation}
\label{eq:I_t_n-4-est}
\E 
\big[
\mathcal{I}_{t_n}^4
\big|
\mathcal{F}_{t_n}
\big]
\leq 
108(2L+1)^4 d^2 h^2
+
27(2L+1)^4h^4
\big\|
b_{t_{n}}(0)
\big\|^4
+
27\tilde{L}^4 h^4.
\end{equation}
{\color{black}
Similarly, since
\[
\Big(
\sum_{j=0}^{k-1}
B_{t_{n,j}}\mathcal{I}_{t_{n,j}}
\Big)^2
\leq
\Big(
\sum_{j=0}^{k-1}B_{t_{n,j}}
\Big)
\sum_{j=0}^{k-1}
B_{t_{n,j}}\mathcal{I}_{t_{n,j}}^2,
\]
we have
\[
\E
\bigg[
\Big(
\sum_{j=0}^{k-1}
B_{t_{n,j}}\mathcal{I}_{t_{n,j}}
\Big)^2
\bigm|
\mathcal{F}_{t_n}
\bigg]
\leq
C(L,\tilde{L})
\Big(
dh+h^2\|b_{t_n}(0)\|^2+h^2
\Big).
\]
}
Therefore, we have
\begin{equation}
\begin{aligned}
\E\big[
\Xi_{t_{n+1}}
\big|
\mathcal{F}_{t_n}
\big]
& {\color{black}\leq}
{\color{black}C}\,
\E \Big[
{\color{black}
\Big(
\sum_{j=0}^{k-1}
B_{t_{n,j}}\mathcal{I}_{t_{n,j}}
\Big)^2}
\big|
\mathcal{F}_{t_n}
\Big]
+
\tfrac{6h}{k}
\,
\E \bigg[
\Big\|
\sum_{j=0}^{k-1}
Z_{t_{n,j}}
\Big\|^2
\big|
\mathcal{F}_{t_n}
\bigg]\\
& \quad 
+
2
\sqrt{\tfrac{2h}{k}}
\E \bigg[
\Big \langle
{X}_{t_n}^N,
\sum_{j=0}^{k-1}
Z_{t_{n,j}}
\Big\rangle
\big|
\mathcal{F}_{t_n}
\bigg]
+
\big(
6+C(\lambda)
\big)h 
\big\|b_{t_n}(0)\big\|^2\\
& \leq 
C(L, {\color{black}\tilde{L}},\lambda)
\big(
d + \|b_{t_n}(0)\|^2
\big)h.
\end{aligned}
\end{equation}
Based on the above estimate, we have
\begin{equation}
\begin{aligned}
\E 
\Big[
\mathbb{I}_2
\big|
\mathcal{F}_{t_n}
\Big]
& =
\tfrac{1}
{
\gamma_{t_n}
\big(
1+
\|
{X}_{t_n}^N
\|^2
\big)}
\E
\Big[
\Xi_{t_{n+1}}
\big|
\mathcal{F}_{t_n}
\Big]\\
& \leq 
\tfrac{
C(L, \tilde{L},\lambda)
(d + \|b_{t_n}(0)\|^2)h}
{\gamma_{t_n}
\big(
1+
\|
{X}_{t_n}^N
\|^2
\big)}.
\end{aligned}
\end{equation}
For the estimate of $\E 
\big[
\mathbb{I}^2_2
\big|
\mathcal{F}_{t_n}
\big]$, using a fundamental inequality shows
\begin{equation}
\label{eq:numerical-Xi-square-expansion}
\begin{aligned}
\big(
\Xi_{t_{n+1}}
\big)^2
& \leq 
{\color{black}4}
\Big(
\big(6+C(\lambda)
\big)^2
h^2 \big\|b_{t_n}(0)\big\|^4
+ 
{\color{black}
49
\Big(
\sum_{j=0}^{k-1}
B_{t_{n,j}}\mathcal{I}_{t_{n,j}}
\Big)^4}
+
\tfrac{36h^2}{k^2}
\Big\|
\sum_{j=0}^{k-1}
Z_{t_{n,j}}
\Big\|^4\\
& \quad 
+
\tfrac{8h}{k}
\Big \langle
{X}_{t_n}^N,
\sum_{j=0}^{k-1}
Z_{t_{n,j}}
\Big \rangle^2
\Big)\\
& {\color{black}\leq}
{\color{black}4}
\Big(
\big(6+C(\lambda)\big)^2h^2 
\big\|b_{t_n}(0)\big\|^4
+
49
\Big(
\sum_{j=0}^{k-1}
B_{t_{n,j}}\mathcal{I}_{t_{n,j}}
\Big)^4
+
\tfrac{36h^2}{k^2}
\Big\|
\sum_{j=0}^{k-1}
Z_{t_{n,j}}
\Big\|^4\\
& \quad 
{\color{black}
+
\tfrac{8h}{k}
\big\|
{X}_{t_n}^N
\big\|^2
\Big\|
\sum_{j=0}^{k-1}
Z_{t_{n,j}}
\Big \|^2}
\Big).
\end{aligned}
\end{equation}
{\color{black}
Moreover, for the nonnegative Bernoulli variables \(B_{t_{n,j}}\), we use
\[
\Big(\sum_{j=0}^{k-1} B_{t_{n,j}}\mathcal{I}_{t_{n,j}}\Big)^4
\leq
\Big(\sum_{j=0}^{k-1}B_{t_{n,j}}\Big)^3
\sum_{j=0}^{k-1}B_{t_{n,j}}\mathcal{I}_{t_{n,j}}^4 .
\]
}
{\color{black}
Since the Bernoulli variables are independent of the Gaussian variables, the following estimate hold true,
\begin{equation}
\label{eq:numerical-BI-fourth-bound}
\begin{aligned}
\E 
\bigg[
\Big(
\sum_{j=0}^{k-1}
B_{t_{n,j}}\mathcal{I}_{t_{n,j}}
\Big)^4
\bigm|
\mathcal{F}_{t_n}
\bigg]
& \leq
\sum_{j=0}^{k-1}
\E\bigg[
B_{t_{n,j}}
\Big(\sum_{\ell=0}^{k-1}B_{t_{n,\ell}}\Big)^3
\bigg]\,
\E\big[
\mathcal{I}_{t_{n,j}}^4
\bigm|
\mathcal{F}_{t_n}
\big]\\
& \leq
\E\bigg[
\Big(\sum_{\ell=0}^{k-1}B_{t_{n,\ell}}\Big)^4
\bigg]\,
\sup_{0\leq j\leq k-1}
\E\big[
\mathcal{I}_{t_{n,j}}^4
\bigm|
\mathcal{F}_{t_n}
\big]\\
& \leq
C(L,\tilde{L})
\Big(
d^2h^2+h^4\|b_{t_n}(0)\|^4+h^4
\Big).
\end{aligned}
\end{equation}
}
Moreover,
\begin{equation}
\label{eq:numerical-gaussian-moment-bound}
\E\bigg[
\Big\|\sum_{j=0}^{k-1}Z_{t_{n,j}}\Big\|^2
\bigm|\mathcal{F}_{t_n}
\bigg]=kd,
\qquad
\E\bigg[
\Big\|\sum_{j=0}^{k-1}Z_{t_{n,j}}\Big\|^4
\bigm|\mathcal{F}_{t_n}
\bigg]
=k^2(d^2+2d)
\leq 3k^2d^2 .
\end{equation}
Keeping this in mind and using the Young inequality leads to
\begin{equation}
\begin{aligned}
\E 
\Big[
\big(
\Xi_{t_{n+1}}
\big)^2
\big|
\mathcal{F}_{t_n}
\Big]
& {\color{black}\leq}
{\color{black}32dh}
{\color{black}
\big\|
{X}_{t_n}^N
\big\|^2
+
C (L,{\color{black}\tilde{L}},\lambda)
\big( d^2 + \|b_{t_n}(0)\|^4
\big)
h^2 .}
\end{aligned}
\end{equation}
{\color{black}
Substituting \eqref{eq:numerical-BI-fourth-bound} and
\eqref{eq:numerical-gaussian-moment-bound} into
\eqref{eq:numerical-Xi-square-expansion} yields
}
\begin{equation}
\label{eq:numerical-bounded-third-term}
\begin{aligned}
\E 
\Big[
\big(
\Xi_{t_{n+1}}
\big)^2
\big|
\mathcal{F}_{t_n}
\Big]
& \leq 
196
\,
\E 
\bigg[
\mathcal{I}_{t_n}^4
\E
\Big[
\Big(\sum_{j=0}^{k-1}
B_{t_{n,j}}
\Big)^4
\big|
\mathcal{F}_{t_n}
\Big]
\bigg]
+
\tfrac{144h^2}{k^2}
\,
\E \bigg[
\Big\|
\sum_{j=0}^{k-1}
Z_{t_{n,j}}
\Big\|^4
\big|
\mathcal{F}_{t_n}
\bigg]\\
& \quad 
+
\tfrac{32h}{k}
\big\|
{X}_{t_n}^N
\big\|^2
\E \bigg[
\Big\|
\sum_{j=0}^{k-1}
Z_{t_{n,j}}
\Big \|^2
\big|
\mathcal{F}_{t_n}
\bigg]
+
4
\big
(6+C(\lambda)\big)^2h^2 
\big\|b_{t_n}(0)\big\|^4\\
& \leq 
32 d h 
\big\|
{X}_{t_n}^N
\big\|^2
+ 
C (L, \tilde{L},\lambda)
\big( d^2 + \|b_{t_n}(0)\|^4
\big)
h^2 .
\end{aligned}
\end{equation}
Therefore, one can obtain that 
\begin{equation}
\begin{aligned}
\E 
\Big[
\mathbb{I}^2_2
\big|
\mathcal{F}_{t_n}
\Big]
& =
\tfrac{1}
{
\gamma^2_{t_n}
\big(
1+
\|
{X}_{t_n}^N
\|^2
\big)^2}
\,
\E
\Big[
\big(
\Xi_{t_{n+1}}
\big)^2
\big|
\mathcal{F}_{t_n}
\Big]\\
&
\leq
{\color{black}
\tfrac{32dh}{\gamma^2_{t_n}\big(
1+
\|
{X}_{t_n}^N
\|^2
\big)}}
+
\tfrac{
C(L, {\color{black}\tilde{L}},\lambda)
\big(
d^2 + \|b_{t_n}(0)\|^4
\big)h^2}
{
\gamma^2_{t_n}
\big(
1+
\|
{X}_{t_n}^N
\|^2
\big)^2}.
\end{aligned}
\end{equation}
Equipped with estimates, one can derive that 
\begin{equation}
\begin{aligned}
&
\E 
\Big[
\big(
1+
\mathbb{I}_1
+
\mathbb{I}_2
\big)^2
\big|
\mathcal{F}_{t_n}
\Big]\\
& \leq
1
+
{\color{black}1952}
\Big(
2\big((2L+1)^2 +{\color{black}\tilde{L}}\big)^2
+1
\Big)^2
h^2
/{\color{black}\gamma^2_{t_n}}
+
\tfrac{
C(L, {\color{black}\tilde{L}},\lambda)
\big(
d + \|b_{t_n}(0)\|^2
\big)h}
{\gamma_{t_n}
\big(
1+
\|
{X}_{t_n}^N
\|^2
\big)}\\
& \quad
+
{\color{black}
\tfrac{64dh}{\gamma^2_{t_n}\big(
1+
\|
{X}_{t_n}^N
\|^2
\big)}}
+
\tfrac{
C(L, {\color{black}\tilde{L}},\lambda)
\big(
d^2 + \|b_{t_n}(0)\|^4
\big)h^2}
{
\gamma^2_{t_n}
\big(
1+
\|
{X}_{t_n}^N
\|^2
\big)^2}.
\end{aligned}
\end{equation}
Thus, we get 
\begin{equation}
\label{eq:1+X^4}
\begin{aligned}
\E
\bigg[
\Big(1+
\big\|
{X}_{t_{n+1}}^N
\big\|^2 
\Big)^2
\Bigg]
& {\color{black}\leq}
\gamma^2_{t_n}
\E
\bigg[
\Big(1+
\big\|
{\color{black}{X}_{t_n}^N}
\big\|^2 
\Big)^2
\Bigg]\\
& \quad +
{\color{black}1952}
\Big(
2\big((2L+1)^2 +{\color{black}\tilde{L}}\big)^2
+1
\Big)^2
h^2
\,
\E
\bigg[
\Big(1+
\big\|
{\color{black}{X}_{t_n}^N}
\big\|^2
\Big)^2
\Bigg]\\
& \quad +
\gamma_{t_n}
\E
\bigg[
\Big(
1+
\big\|
{X}_{t_n}^N
\big\|^2
\Big)
C(L, {\color{black}\tilde{L}},\lambda)
\big(
d + \|b_{t_n}(0)\|^2
\big)h
\bigg]\\
& \quad +
{\color{black}64}dh
\,
\E
\bigg[
\Big(
1+
\big\|
{X}_{t_n}^N
\big\|^2
\Big)
\bigg]
+
C(L, {\color{black}\tilde{L}},\lambda)
\big(
d^2 + \|b_{t_n}(0)\|^4
\big)h^2.
\end{aligned}
\end{equation}
Based on the above estimate, using the Young inequality gives 
\begin{equation}
\label{eq:numer-young}
\begin{aligned}
&
\gamma_{t_n}
\E
\bigg[
\Big(
1+
\big\|
{X}_{t_n}^N
\big\|^2
\Big)
C(L, \tilde{L},\lambda)
\big(
d + \|b_{t_n}(0)\|^2
\big)h
\bigg]\\
& \leq 
\lambda h \gamma^2_{t_n} /2
\,
\E
\bigg[
\Big(1+
\big\|
{X}_{t_{n}}^N
\big\|^2 
\Big)^2
\Bigg]
+
C(L, \tilde{L},\lambda)
\big(
d^2 + \|b_{t_n}(0)\|^4
\big)h,
\end{aligned}
\end{equation}
and 
\begin{equation}
\begin{aligned}
{\color{black}64}dh
\,
\E
\bigg[
\Big(
1+
\big\|
{X}_{t_n}^N
\big\|^2
\Big)
\bigg]
\leq 
\lambda h /2
\,
\E
\bigg[
\Big(1+
\big\|
{\color{black}{X}_{t_n}^N}
\big\|^2 
\Big)^2
\Bigg]
+
{\color{black}C(\lambda)d^2h}.
\end{aligned}
\end{equation}

Combining \eqref{eq:numer-young} with \eqref{eq:1+X^4}, one can obtain 
\begin{equation}
\label{eq:1+X4-after-young}
\begin{aligned}
&
\E
\bigg[
\Big(1+
\big\|
{X}_{t_{n+1}}^N
\big\|^2 
\Big)^2
\Bigg]\\
& {\color{black}\leq}
\Big(
\gamma^2_{t_n}
+
{\color{black}1952}
\Big(
2\big((2L+1)^2 +{\color{black}\tilde{L}}\big)^2
+1
\Big)^2
h^2
+
\lambda h (\gamma^2_{t_n}+1)/2
\Big)
\E
\bigg[
\Big(1+
\big\|
{\color{black}{X}_{t_n}^N}
\big\|^2 
\Big)^2
\Bigg]\\
& \quad +
C(L, {\color{black}\tilde{L}},\lambda)
\big(
d^2 + \|b_{t_n}(0)\|^4
\big)h
\\
& {\color{black}\leq} 
\Big(
\gamma^2_{t_n}
+
{\color{black}1952}
\Big(
2\big((2L+1)^2 +{\color{black}\tilde{L}}\big)^2
+1
\Big)^2
h^2
+
\lambda h  (\gamma^2_{t_n}+1)/2
\Big)
\E
\bigg[
\Big(1+
\big\|
{\color{black}{X}_{t_n}^N}
\big\|^2 
\Big)^2
\Bigg]\\
& \quad +
C(L, {\color{black}\tilde{L}},{\color{black}L'},\lambda)
\big(
d^2 + T^4
\big)h,
\end{aligned}
\end{equation}
where, the last inequality is an application of Lemma \ref{lem:score-lipschitz}, Assumption \ref{ass:lipschitz-time}{\color{black}, and the fact \(L'd^{1/2}:=\|\nabla V(0)\|\)}, for any $t \in [t_n,t_{n+1}]$,
\begin{equation}
\label{eq:b_t-0-est}
\begin{aligned}
\big\|
b_{t}(0)
\big\|
&{\color{black}
=2\big\|\nabla \log \tilde p_{T-t}(0)\big\|
=2\big\|\nabla \log {p}_{T-t}(0)\big\|}\\
& {\color{black}\leq C\bigg(}
\big\|
\nabla \log {p}_{T-{t}}(0)
-
\nabla \log {p}_{T-t_n}(0)
\big\|\\
& \quad +
\sum_{j=n}^{N-1}
\big\|
\nabla \log {p}_{T-{t_{j+1}}}(0)
-
\nabla \log {p}_{T-t_j}(0)
\big\|
+
\big\|
\nabla \log {p}_{0}(0)
\big\|
{\color{black}\bigg)}\\
& {\color{black}\leq 
C\big(\tilde{L}T + L' d^{1/2}\big).}
\end{aligned}
\end{equation}

\textbf{Case I.} $nk\leq N_h$.
By noting that if $h$ satisfies \eqref{eq:h-condition-numerical-bound}:
\[
h \leq 
\tfrac{2(1-\lambda)}
{C(L, \tilde{L})},
\]
where 
\begin{equation}
\label{eq:numerical-bounded-constant-c}
{\color{black}C(L,\tilde{L})}
:=
{\color{black}
6(2L+1)^2+
{\color{black}1952}
\Big(
2\big((2L+1)^2 +\tilde{L}\big)^2
+1
\Big)^2,}
\end{equation}
then we can set $\rho^{(1)}_h:= C(L, \tilde{L})h/2+\lambda$, such that $\rho^{(1)}_h \in (\lambda,1)$. 
Then we have:
\[
\gamma^2_{t_n} \leq \gamma_{t_n} \leq 1
\]
and 
\[
\gamma_{t_n}
+ 
1952
\Big(
2\big((2L+1)^2 +\tilde{L}\big)^2
+1
\Big)^2
h^2
+
\lambda h 
\leq 
1.
\]
One can then apply these inequalities iteratively, together with the facts that ${X}_{0}^N=\overleftarrow{X}_T$ and Lemma \ref{lem:target-p-order-moment} to obtain
\begin{equation}
\begin{aligned}
&
\E
\bigg[
\Big(1+
\big\|
{X}_{t_{n+1}}^N
\big\|^2 
\Big)^2
\Bigg] \\
& \quad  \leq 
\Big(
\gamma_{t_n}
+
{\color{black}1952}
\Big(
2\big((2L+1)^2 +{\color{black}\tilde{L}}\big)^2
+1
\Big)^2
h^2
+
\lambda h
\Big)
\E
\bigg[
\Big(1+
\big\|
{\color{black}{X}_{t_n}^N}
\big\|^2 
\Big)^2
\Bigg]\\
& \quad \quad +
C(L, {\color{black}\tilde{L}},{\color{black}L'},\lambda)
\big(
d^2 + T^4
\big)h\\
& \quad := 
\Gamma_{t_n}
\E
\bigg[
\Big(1+
\big\|
{\color{black}{X}_{t_n}^N}
\big\|^2 
\Big)^2
\Bigg]
+
C(L, {\color{black}\tilde{L}},{\color{black}L'},\lambda)
\big(
d^2 + T^4
\big)h\\
& \quad =
\prod_{\ell=0}^{n} 
\Gamma_{t_\ell}
\E
\bigg[
\Big(1+
\big\|
{X}_{0}^N
\big\|^2 
\Big)^2
\Bigg]
+
C(L, {\color{black}\tilde{L}},{\color{black}L'},\lambda)
\big(
d^2 + T^4
\big)h
\Big(
\sum_{m=1}^{n} 
\prod_{\ell=m}^{n}  
\Gamma_{t_\ell}
+1
\Big)\\
& \quad \leq 
{\color{black}C d^2} 
+
C(L, {\color{black}\tilde{L}},{\color{black}L'},\lambda)
\big(
d^2 + T^4
\big)T,
\end{aligned}
\end{equation}
where 
\begin{equation}
\Gamma_{t_n}   
: =
\gamma_{t_n}
+ 
1952
\Big(
2\big((2L+1)^2 +\tilde{L}\big)^2
+1
\Big)^2
h^2
+
\lambda h ,
\end{equation}
and the last inequality follows from
\[
h
\Big(
\sum_{m=1}^{n} 
\prod_{\ell=m}^{n}  
\big(
{\color{black}\Gamma_{t_\ell}}
\big)
+1
\Big)
\leq 
h
{\color{black}\big(
\lfloor N_h/k\rfloor+1 
\big)}
\leq 
{\color{black}\widehat{T}(\alpha,M,0)+h
\leq C T.}
\]

\textbf{Case II.} $nk > N_h$.
By iteration, we acquire
\begin{equation}
\begin{aligned}
&
\E
\bigg[
\Big(1+
\big\|
{X}_{t_{n+1}}^N
\big\|^2 
\Big)^2
\Bigg]\\
& \leq 
\Big(
\gamma^2_{t_n}
+
1952
\gamma_{t_n} 
\Big(
2\big((2L+1)^2 +\tilde{L}\big)^2
+1
\Big)^2
h^2
+
\lambda h (\gamma^2_{t_n}+1)/2  
\Big)
\E
\bigg[
\Big(1+
\big\|
{X}_{t_{n}}^N
\big\|^2 
\Big)^2
\Bigg]\\
& \quad +
C(L, \tilde{L},\lambda)
\big(
d^2 + T^4
\big)h \\
& := 
\tilde{\Gamma}_{t_n}
\E
\bigg[
\Big(1+
\big\|
{X}_{t_{n}}^N
\big\|^2 
\Big)^2
\Bigg]
+
C (L, \tilde{L},L',\lambda)
\big( d^2 + T^4 
\big)h\\
& \leq 
\prod_{\ell=0}^{n} 
\tilde{\Gamma}_{t_\ell}
\E
\bigg[
\Big(1+
\big\|
{X}_{0}^N
\big\|^2 
\Big)^2
\Bigg]
+
C (L, \tilde{L},L',\lambda)
\big( d^2 + T^4 
\big)h
\Big(
\sum_{m=1}^{n} 
\prod_{\ell=m}^{n}  
\tilde{\Gamma}_{t_\ell}
+1
\Big)\\
& \leq 
\Big(
1+
\mathcal{M}_1 d ^2
+
C (L, \tilde{L},L',\lambda)
\big( d^2 + T^4 
\big)
T
\Big)
\mathrm{e}^{
\tilde{C}(L,\tilde{L})
\eta(\alpha, M,\rho^{(1)}_h)},
\end{aligned}
\end{equation}
where the last inequality follows by an argument similar to those used for \eqref{eq:exp} and \eqref{eq:recusion},
and 
$\tilde{C}(L,\tilde{L})$ is independent of $d$ and $h$.
The proof is completed.
\end{proof}

Based on the above preliminaries, we next present some regularity estimates of the process $\{X^N_{t_{n,i}}\}_{n \in [N]_0, i \in [k-1]_0}$ defined by \eqref{eq:poisson-aux}.

\begin{lem}
\label{lem:prior-bound}
Let Assumptions \ref{ass:weak-convex}, \ref{ass:one-side-lipschitz} and \ref{ass:lipschitz-time} hold. Let $\big\{\tilde{X}_{t_{n,i}}\big\}$ and $\big\{{X}^N_{t_{n,i}}\big\}$, for any $n\in [N]_0, i\in [k-1]_0$ be defined as \eqref{eq:poisson-aux}.
Then, there exists a constant $C(\mathcal{M}_2)$, satisfying

\begin{equation}
\label{eq:b-error}
\E
\Big[
\big\|
b_{t_{n,i}}(\tilde{X}_{t_{n,i}}^N)
-
b_{t_n}(X_{t_n}^N)
\big\|^{4}
\Big]
\leq 
C(\mathcal{M}_2) 
T^5 d^2 h^{2},
\end{equation}
and 
\begin{equation}
\label{eq:x-error}
\E
\Big[
\big\|
\tilde{X}_{t_{n,i}}^N
-
X_{t_{n,i}}^N
\big\|^{2}
\Big]
\leq 
C(\mathcal{M}_2)
T^{5/2}d h^3,
\end{equation}
where $\mathcal{M}_2$ comes from Lemma \ref{lem:numerical-p-order-moment}.
\end{lem}

\begin{proof}

First, one can do the following decomposition:
\begin{equation}
\begin{aligned}
\E
\Big[
\big\|
b_{t_{n,i}}(\tilde{X}_{t_{n,i}}^N)
-
b_{t_n}(X_{t_n}^N)
\big\|^{4}
\Big]
& \leq 
8\,
\E
\Big[
\big\|
b_{t_{n,i}}(\tilde{X}_{t_{n,i}}^N)
-
b_{t_{n,i}}(X_{t_n}^N)
\big\|^{4}
\Big]\\
& \quad +
8\,
\E
\Big[
\big\|
b_{t_{n,i}}({X}_{t_{n}}^N)
-
b_{t_{n}}(X_{t_n}^N)
\big\|^{4}
\Big].
\end{aligned}
\end{equation}
Regarding the first term, we use the Lipschitz condition of $b_t(\cdot)$ and Lemma \ref{lem:numerical-p-order-moment} to obtain
\begin{equation}
\begin{aligned}
\E
\Big[
\big\|
b_{t_{n,i}}(\tilde{X}_{t_{n,i}}^N)
-
b_{t_{n,i}}(X_{t_n}^N)
\big\|^{4}
\Big]
& \leq 
(2L+1)^{4}
\,
\E
\Big[
\big\|
\tilde{X}_{t_{n,i}}^N
-
X_{t_n}^N
\big\|^{4}
\Big]\\
& =
 (2L+1)^{4}
\,
\E
\Big[
\big\|
\tfrac{ih}{k}
b_{t_n}(X_{t_n}^N)
+
\sum_{j=0}^{i-1}
\sqrt{\tfrac{2h}{k}} 
Z_{t_{n,j}}
\big\|^{4}
\Big]\\
& \leq 
C(\mathcal{M}_2) 
T^5 d^2 h^2.
\end{aligned}
\end{equation}
With regard to the second term, applying Assumption \ref{ass:lipschitz-time} and Lemma \ref{lem:numerical-p-order-moment}, we have
\begin{equation}
\begin{aligned}
\E
\Big[
\big\|
b_{t_{n,i}}({X}_{t_{n}}^N)
-
b_{t_{n}}(X_{t_n}^N)
\big\|^{4}
\Big]
& \leq 
C(\mathcal{M}_2) 
T^5 d^2 h^4.
\end{aligned}
\end{equation}
The first desired assertion \eqref{eq:b-error} is thus validated. Concerning the second assertion, we recall \eqref{eq:poisson-aux} to get, for any $n \in [N]_0, i \in [k-1]_0$,
\begin{equation}
\tilde{X}_{t_{n,i}}^N
-
X_{t_{n,i}}^N
=
-
h
\sum_{j=0}^{i-1}
B_{t_{n,j}}
\Big(
b_{t_{n,j}}(\tilde{X}_{t_{n,j}}^N)
-
b_{t_n}(X_{t_n}^N)
\Big).
\end{equation}

In view of the above equality, using \eqref{eq:b-error} and the identities
$\E[B_{t_{n,r}}B_{t_{n,\ell}}]=1/k^2$ for $r\neq \ell$ and
$\E[B^2_{t_{n,j}}]=1/k$, we obtain
\begin{equation}
\begin{aligned}
\E
\Big[
\big\|
\tilde{X}_{t_{n,i}}^N
-
X_{t_{n,i}}^N
\big\|^{2}
\Big]
& =
h^2
\,
\E
\bigg[
\Big\|
\sum_{j=0}^{i-1}B_{t_{n,j}}
\Big(
b_{t_{n,j}}(\tilde{X}_{t_{n,j}}^N)
-
b_{t_n}(X_{t_n}^N)
\Big)
\Big\|^{2}
\bigg]\\
& \leq
\underbrace{
\tfrac{2h^2}{k^2}
\sum_{0\leq r<\ell\leq i-1}
\mathbb{E}
\Big[
\big\langle 
b_{t_{n,r}}(\tilde{X}_{t_{n,r}}^N)
-
b_{t_n}(X_{t_n}^N),
b_{t_{n,\ell}}(\tilde{X}_{t_{n,\ell}}^N)
-
b_{t_n}(X_{t_n}^N)
\big\rangle
\Big]}_{:=J_1}\\
& \quad +
\underbrace{
\tfrac{h^2}{k}
\sum_{j=0}^{i-1}
\mathbb{E}
\Big[
\big\|
b_{t_{n,j}}(\tilde{X}_{t_{n,j}}^N)
-
b_{t_n}(X_{t_n}^N)
\big\|^2
\Big]}_{=:J_2}
\end{aligned}
\end{equation}
In the following we deal with the above two items separately. 
Using the Cauchy-Schwarz inequality and \eqref{eq:b-error}, one can derive that
\begin{equation}
\begin{aligned}
J_1 
& \leq 
\tfrac{2h^2}{k^2}
\sum_{0\leq r<\ell\leq i-1}
\mathbb{E}
\Big[
\big\|
b_{t_{n,r}}(\tilde{X}_{t_{n,r}}^N)
-
b_{t_n}(X_{t_n}^N)
\big\|
\cdot
\big\|
b_{t_{n,\ell}}(\tilde{X}_{t_{n,\ell}}^N)
-
b_{t_n}(X_{t_n}^N)
\big\|
\Big]\\
& \leq
h^2
\sup_{0\leq r \leq i-1}
\mathbb{E}
\Big[
\big\|
b_{t_{n,r}}(\tilde{X}_{t_{n,r}}^N)
-
b_{t_n}(X_{t_n}^N)
\big\|^2
\Big]\\
& \leq
C
T^{5/2}d h^3.
\end{aligned}
\end{equation}
Regarding $J_2$, we use \eqref{eq:b-error} to obtain 
\begin{equation}
J_2 \leq 
C(\mathcal{M}_2) 
T^{5/2}d h^3.
\end{equation}
The proof is thus completed.
\end{proof}

{\color{black}
\paragraph{Intuition behind the Poisson midpoint method.}
\leavevmode\par

\textbf{Algorithmically,} the intuition behind PM is to retain the bias reduction of a midpoint correction while avoiding the cost of evaluating the score at all $k$ intermediate points. The key idea is to activate each interior correction $\Delta_i=b_{t_{n,i}}(x_i)-b_{t_n}(y)$ independently with probability $1/k$. 
Writing $\mathbb E_B$ for expectation with respect to the Bernoulli variables alone, and using $\E_B[B_{t_{n,i}}]=1/k$, we obtain, for any $x,y \in \R^d$,
$$
\mathbb E_B
\Big[
h \, b_{t_n}(y)+
h
\sum_{i=0}^{k-1}B_{t_{n,i}}
\Delta_i
\Big]
=
\frac{h}{k}
\sum_{i=0}^{k-1}
b_{t_{n,i}}
(x_i).
$$
The PM update therefore has the same conditional mean as the full $k$-point midpoint update, while $\mathbb E_B[\sum_i B_{t_{n,i}}]=1$ means only one interior score is evaluated on average. Including the left-endpoint score, PM uses two expected score evaluations per macro-step, the same order as EM.

\textbf{Analytically,} 
each PM sub-step can be viewed as an EM step plus a centered random perturbation, which is given as
$$
R_{n,i}
:=
h\left(B_{t_{n,i}}-\tfrac1k\right)\Delta_i ,
\qquad
\mathbb E_B[R_{n,i}]=0 ,
$$
The key observation is that this perturbation is one-dimensional
since the only additional randomness comes from the scalar Bernoulli variable. 
A coupling adapted to this structure makes its Wasserstein contribution enter at a higher order than under a direct synchronous coupling. This is the key reason that the randomized thinning preserves the midpoint acceleration.
}

\paragraph{Rank-one perturbation and optimal coupling.}
Each refined PM sub-step produces, conditionally on $\mathcal{F}_{t_{n,j}}$, a Gaussian increment together with a mean-zero, rank-one perturbation
\[
V(x,y,B_{t_{n,i}})
:=
\sqrt{\tfrac{hk}{2}}
\Big(B_{t_{n,i}}-\tfrac{1}{k}\Big)
\Big(b_{t_{n,j}}(x)-b_{t_n}(y)\Big),
\]
supported on the line spanned by the drift discrepancy $b_{t_{n,j}}(x)-b_{t_n}(y)$. The following Wasserstein bound exploits this rank-one structure \cite{srinivasan2026poisson,kandasamy2024poisson}. 

\begin{lem}[Wasserstein bound for one-dimensional bounded perturbations]
\label{lem:key-lem}
Let $V$ be a random vector in $\mathbb{R}^d$ such that
\begin{enumerate}
\item[(i)] $\|V\|\leq\beta$ almost surely; $\mathbb{E}[V]=0$;
$\mathbb{E}[VV^\top]=\Sigma$;
\item[(ii)] there exists a deterministic unit vector
$u\in\mathbb{S}^{d-1}$ with $V\in\mathrm{span}(u)$ almost surely.
\end{enumerate}
Let $Z\sim\gamma^d$ be independent of $V$, and set
$\nu:=\operatorname{Tr}(\Sigma)=\mathbb{E}\big[\|V\|^2\big]$. Then
\[
\mathcal{W}_2^2\bigl(\mathcal{L}(Z+V),\gamma^d\bigr)
\leq
\tfrac{11}{2}\nu^2
+\mathbf{1}_{\{5\beta^2>1\}}\cdot 2\nu .
\]
\end{lem}

A direct synchronous coupling of $Z+V$ with $Z$ would only give
\[
\mathcal{W}_2^2(\mathcal{L}(Z+V),\gamma^d)
\leq
\mathbb{E}\big[\|V\|^2\big]
=
\nu .
\]
Lemma~\ref{lem:key-lem} sharpens the linear dependence on $\nu$ to a
quadratic one $\nu^2\textbf{}$, 
which
is the main source of the improved PM discretization bound.
Since the error analysis is based on mean-square estimates, we
need a coupling that realizes this Wasserstein distance as an
\(L^2\)-error.
The following lemma
provides such a coupling via Brenier's optimal transport theorem.

\begin{lem}
\label{lem:key-lem-coupling}
Fix $n\in[N]_0$. Let $V_n$ satisfy the assumptions of
Lemma~\ref{lem:key-lem} with parameters $\beta_n$, $\Sigma_n$, and
$\nu_n:=\operatorname{Tr}(\Sigma_n)$, and let $Z_n\sim\gamma^d$ be
independent of $V_n$. Then there exists a random vector
$\xi_n\sim\gamma^d$ such that
\begin{equation}
\label{eq:key-lem-coupling}
\mathbb{E}\bigl[\|Z_n+V_n-\xi_n\|^2\bigr]
=
\mathcal{W}_2^2
\big(
\mathcal{L}(Z_n+V_n),
\gamma^d
\big)
\leq
\tfrac{11}{2}\nu_n^2
+\mathbf{1}_{\{5\beta_n^2>1\}}\cdot 2\nu_n .
\end{equation}
\end{lem}

\begin{proof}
Since $V_n$ is independent of $Z_n\sim\gamma^d$, the convolution
$\mathcal{L}(Z_n+V_n)$ admits a density with respect to Lebesgue
measure. Brenier's
theorem~\cite[Theorem~2.5.10]{figalli2021invitation} therefore yields
an optimal transport map $\mathcal{T}_n=\nabla\phi_n$ from
$\mathcal{L}(Z_n+V_n)$ to $\gamma^d$ with
$\phi_n:\mathbb{R}^d\to\mathbb{R}$ convex. Setting
$\xi_n:=\mathcal{T}_n(Z_n+V_n)$ gives $\xi_n\sim\gamma^d$ and
\[
\mathbb{E}\bigl[\|Z_n+V_n-\xi_n\|^2\bigr]
=
\mathcal{W}_2^2\bigl(\mathcal{L}(Z_n+V_n),\gamma^d\bigr),
\]
to which Lemma~\ref{lem:key-lem} applies and gives
\eqref{eq:key-lem-coupling}.
\end{proof}

\paragraph{Embedding into a Wiener process.}
Since the numerical variables
\(\{Z_{n}\}_{n \in [N]_{0} }\) are independent across refined
sub-steps and the optimal transport maps are measurable, applying
Lemma~\ref{lem:key-lem-coupling} successively over the refined grid yields an i.i.d.\ family
\(\{\xi_{t_{n,j}}\}_{n,j}\) with  law \(\gamma^d\).
By the standard isonormal Gaussian process construction \cite[Section 1.1]{nualart2006malliavin}, these
variables can be embedded into a Wiener process \(W\) such that
\[
\Delta W_{t_{n,j}}
=
\sqrt{\tfrac{h}{k}}\,\xi_{t_{n,j}} ,
\]
which we take as the driving process of the reverse SDE \eqref{eq:backward-SDE} throughout the proof of Theorem~\ref{thm:poisson}.

Up to this stage, we have established sufficient machinery to obtain the discretization error estimate between the PM auxiliary process \eqref{eq:poisson-aux} and the backward process \eqref{eq:backward-SDE} as below.

\begin{prop}
[Discretization error of PM method]
\label{prop:pi-poisson}
Suppose that Assumptions \ref{ass:weak-convex}, \ref{ass:one-side-lipschitz} and \ref{ass:lipschitz-time} hold. 
For any $n \in [N]_0, j\in[k-1]_0$, consider the process $\{X^N_{t_{n,j}}\}$ generated by \eqref{eq:poisson-aux} using the uniform step-size. For any $\lambda\in (0,1/4)$, let the step-size satisfy
\begin{equation}
\label{eq:h-condition}
h \leq
\tfrac{2(1-\lambda)}{3(2L+1)^2}.
\end{equation}
Then there exists a constant $C$, independent of $d$ and $h$, such that,
\begin{equation}
\begin{aligned}
\mathcal{W}_2
\big(
\mu,
\mathcal{L}(X^N_{T})
\big)
& \leq 
C
\mathrm{e}^{C
\eta(\alpha, M,3(2L+1)^2 h/2+\lambda)}
T^3
\big(
\tfrac{\sqrt{d}h}{k}
+ 
k d h^2
+
dh^{2}
+
\sqrt{d}h^{3/2}
\big).
\end{aligned}
\end{equation}
with $\eta\big(\alpha, M, 3(2L+1)^2 h/2+\lambda\big)$ as defined in \eqref{eq:eta-def}, where $L$ comes from Lemma \ref{lem:score-lipschitz}.
\end{prop}
\begin{proof}
By Proposition \ref{prop:Upper-mean-square error-bounds}, it remains to estimate the four residual terms in the one-step recursion.

\paragraph{Step I: estimate of \(R_{t_{n,i}}\).}

Fix $0\leq \epsilon_2 \leq h/k$, define $T$ to denote $T-\epsilon_2$ and $N$ to denote $(T-\epsilon_2)/h$. This transformation allows us to apply Lemma \ref{lem:score-lipschitz} and \cite[Proposition 2]{Conforti2025KLconvergence}.
Using the It\^o formula \cite[Proposition 2]{Conforti2025KLconvergence} yields  
\begin{equation}
    \begin{aligned}
        &\mathrm{d} 
        b_t(
        \overleftarrow{X}_t)  \\
        & =
        -\mathrm{d} 
        \overleftarrow{X}_t
        +
        \mathrm{d}
        \big(
        2 
        \nabla \log \tilde{p}_{T-t}(
        \overleftarrow{X}_t)
        \big)  \\
        & =
        \Big\{
        \overleftarrow{X}_t
        -
        2 \nabla \log \tilde{p}_{T-t}
        (\overleftarrow{X}_t)
        +
        2 \nabla \log \tilde{p}_{T-t}
        (
        \overleftarrow{X}_t
        )
        \Big\} 
        \,
        \mathrm{d} t
        +
        \sqrt{2}
        \big(
        {\color{black}-I_d}
        +
        2 \nabla^2 \log \tilde{p}_{T-t}
        (
        \overleftarrow{X}_t
        )
        \big) 
        \,
        \mathrm{d} W_t \\
        & =
        \overleftarrow{X}_t 
        \,
        \mathrm{d} t
        +
        \sqrt{2}
        \big(
        {\color{black}-I_d}+2 \nabla^2 \log \tilde{p}_{T-t}
        (
        \overleftarrow{X}_t
        )
        \big) 
        \,
        \mathrm{d} W_t.
    \end{aligned}
\end{equation}
Therefore, the integral form can be written as, for any $n \in [N-1]_0, \, i \in [k-1]_0$,
\begin{equation}
\label{eq:b(x,t)-b(s,y)}
         b_s
         (\overleftarrow{X}_s) 
        -
        b_{t_{n,i}}(\overleftarrow{X}_{t_{n,i}})
        =
        \int_{t_{n,i}}^s \overleftarrow{X}_r 
        \mathrm{\,d} r
        +
        \sqrt{2} 
        \int_{t_{n,i}}^s
        \big({\color{black}-I_d}+2 \nabla^2 \log \tilde{p}_{T-r}
        (
        \overleftarrow{X}_r)
        \big) 
        \,
        \mathrm{d} W_r.
\end{equation}
Substituting the above equation into $\mathbb{E}\big[ \|  R_{t_{n,i}} \|^2 \big]$ and using the H\"older inequality, the It\^o isometry \cite[Lemma 5.4]{mao2007stochastic} and Lemma \ref{lem:target-p-order-moment} gives
\begin{equation}
\begin{aligned}
\label{eq:R-error-bound}
& 
\mathbb{E}
\Big[
\big\|
R_{t_{n,i}}
\big\|^2
\Big]\\
& =
\mathbb{E}
\bigg[
\Big\|
\int_{t_{n,i}}^{t_{n,i+1}}
\big[
b_{s}(\overleftarrow{X}_s)
-
b_{t_{n,i}}(\overleftarrow{X}_{t_{n,i}})
\big] 
\mathrm{\,d}s
\Big\|^2
\bigg]\\
& =
\E
\bigg[
\Big\|
\int_{t_{n,i}}^{t_{n,i+1}}
\int_{t_{n,i}}^{s}
\overleftarrow{X}_r 
\,
\mathrm{d} r
+
\sqrt{2} 
\int_{t_{n,i}}^{t_{n,i+1}}
\int_{t_{n,i}}^{s}
\big({\color{black}-I_d}+2 \nabla^2 \log \tilde{p}_{T-r}
(\overleftarrow{X}_r)
\big)
\mathrm{d} W_r
{\color{black}\mathrm{\,d}s}
\Big\|^{2}
\bigg]\\
& \leq
\tfrac{2h^2}{k^2}
\E
\bigg[
\int_{t_{n,i}}^{t_{n,i+1}}
\int_{t_{n,i}}^{s}
\big\|
\overleftarrow{X}_r 
\big\|^{2}
\mathrm{d} r
\bigg]
+
\tfrac{4h}{k}
\,
\mathbb{E}
\bigg[
\int_{t_{n,i}}^{t_{n,i+1}}
\int_{t_{n,i}}^s
\big\|
{\color{black}-I_d}++2 \nabla^2 \log \tilde{p}_{T-r}
(\overleftarrow{X}_r)
\big) 
\big\|^2_{\mathrm{F}}
\, 
{\color{black}\mathrm{\,d}r}
\mathrm{\,d} s
\bigg]\\
& \leq
2
\big(
\mathcal{M}_1 +4(1+4 L^2) \big)
\tfrac{dh^3}{k^3},
\end{aligned}
\end{equation}
where the last inequality follows from the direct result of Lemma \ref{lem:score-lipschitz} that
\begin{align}
\big\|
\nabla^2 \log \tilde{p}_{T-t}
(\overleftarrow{X}_t)
\big\|_{\mathrm{F}} 
    \leq 
    \sqrt{d}
    \,
    \big\|\nabla^2 \log \tilde{p}_{T-t}
(\overleftarrow{X}_t)\big\|
    \leq 
    \sqrt{d}
    \,L,
    \quad 
    \forall  ~ t \in [0,T).
\end{align}
\paragraph{Step II:
estimate of 
$
\mathbb{E}
\big[
\big\|
\mathbb{E}
\big[
R_{t_{n,i}}
\mid 
\mathcal{F}_{t_{n,i}}
\big]
\big\|^2
\big]$.}
Noting that, for any $n\in [N-1]_0,\,i\in [k-1]_0$,
\begin{equation}
    \mathbb{E}
    \bigg[
    \int^s_{t_{n,i}}  
    \big({\color{black}-I_d}+
    +2 \nabla^2 \log \tilde{p}_{T-r}
    (
    \overleftarrow{X}_r
    )
    \big)  
    \mathrm{d} W_r
    \Big|
    \mathcal{F}_{t_{n,i}}
    \bigg]
    =0,
\end{equation}
one can utilize 
\eqref{eq:b(x,t)-b(s,y)}, the Jensen inequality, Lemma \ref{lem:target-p-order-moment} and the properties of the conditional expectation to show that,
\begin{align}
\label{eq:R-conditionded-error-bound}
\mathbb{E}
\Big[
\big\|
\mathbb{E}
\big[
R_{t_{n,i}}
\mid 
\mathcal{F}_{t_{n,i}}
\big]
\big\|^2
\Big]
& =
\mathbb{E}
\bigg[
\Big\|
\mathbb{E}
\Big[
\int_{t_{n,i}}^{t_{n,i+1}} 
\big[
b_s(\overleftarrow{X}_s)
-
b_{t_{n,i}}(\overleftarrow{X}_{t_{n, i}}) 
\big]
\mathrm{\,d} s 
\Big|
\mathcal{F}_{t_{n,i}}
\Big]
\Big\|^2
\bigg]\nonumber\\
& =
\mathbb{E}
\Bigg[
\bigg\|
\mathbb{E}
\Big[
\int_{t_{n,i}}^{t_{n,i+1}} 
\int_{t_{n,i}}^{s} 
\overleftarrow{X}_r
\mathrm{\,d} r 
\mathrm{\,d} s 
\Big|
\mathcal{F}_{t_{n,i}}
\Big]\nonumber\\
& \quad  +
\sqrt{2}
\,
\mathbb{E}
\Big[
\int_{t_{n,i}}^{t_{n,i+1}}  
\int_{t_{n,i}}^{s}  
\big(
{\color{black}-I_d}+2 \nabla^2 \log \tilde{p}_{T-r}
(\overleftarrow{X}_r)
\big)  
\mathrm{\,d} W_r 
\mathrm{\,d} s 
\Big|
\mathcal{F}_{t_{n,i}}
\Big] 
\bigg\|^2
\Bigg]\nonumber\\
& {\color{black}
\leq
\mathbb{E}
\bigg[
\mathbb{E}
\Big[
\Big\|
\int_{t_{n,i}}^{t_{n,i+1}} 
\int_{t_{n,i}}^{s} 
\overleftarrow{X}_r
\mathrm{\,d} r 
\mathrm{\,d} s 
\Big \|^2
\Big|
\mathcal{F}_{t_{n,i}}
\Big]
\bigg] 
} \nonumber\\
& {\color{black}
\leq
\tfrac{h^2}{k^2}
\mathbb{\,E}
\Bigg[
\int_{t_{n,i}}^{t_{n,i+1}} 
\int_{t_{n,i}}^{s} 
\big\|
\overleftarrow{X}_r
\big\|^2
\mathrm{\,d} r 
\mathrm{\,d} s 
\Bigg]
} \nonumber\\
& \leq
\tfrac{\mathcal{M}_1d h^4}{k^4}.
\end{align}
\paragraph{Step III:
estimate of 
$\mathbb{E}
\big[
\|
M_{t_{n,i}}
\|^2
\big]$.}
One can do the following decomposition:
\begin{equation}
\label{eq:R_2_decom}
\begin{aligned}
\mathbb{E}
\Big[
\big\|
M_{t_{n,i}}
\big\|^2
\Big]
& \leq
\tfrac{4h}{k}
\underbrace{
\E
\bigg[
\Big\|
Z_{t_{n,i}}
-
\xi_{t_{n,i}}
+
\sqrt{\tfrac{hk}{2}}
\big(
B_{t_{n,i}}-1/k
\big)
\Big( 
b_{t_{n,i}}(\tilde{X}_{t_{n,i}}^N)
-
b_{t_n}(X_{t_n}^N)
\Big)
\Big\|^2
\bigg]}_{=:\mathbb{J}_1}\\
& \quad +
\underbrace{
\tfrac{2h^2}{k^2}
\E
\Big[
\big\|
b_{t_{n,i}}(\tilde{X}_{t_{n,i}}^N)
-
b_{t_{n,i}}(X_{t_{n,i}}^N)
\big\|^2
\Big]}_{=:\mathbb{J}_2}.
\end{aligned}
\end{equation}

The above two items will be treated one by one. 
We estimate $\mathbb{J}_1$ first. 
By abuse of notation, we denote, for any $n \in [N-1]_0,\,j\in[k-1]_0,$ 
\[
V(x,y,B_{t_{n,i}})
:=
\sqrt{
\tfrac{hk}{2}
}
(
B_{t_{n,j}}
-
{1}/{k}
)
\Big(
b_{t_{n,i}}(x)
-
b_{t_n}(y)
\Big).
\]
Conditionally on $\mathcal{F}_{t_{n,i}}$, the perturbation $V(x,y, B_{t_{n,i}})$
is centered and supported on the line spanned by
$b_{t_{n,i}}(x)-b_{t_n}(y)$. Moreover, the corresponding quantities in Lemma~\ref{lem:key-lem} satisfy
\[
\nu_{t_{n,j}}{\color{black}(x,y)}
:= 
\E\Big[
\tfrac{hk}{2}
(
B_{t_{n,j}}
-
{1}/{k}
)^2
\,
\big\|
b_{t_{n,j}}
(x)
-
b_{t_{n}}(y)
\big\|^2
\Big],
\qquad
\beta_{t_{n,i}}^2 
{\color{black}(x,y)}
:=
\tfrac{hk}{2}
\big\|
b_{t_{n,i}}(x)
-
b_{t_n}(y)
\big\|^2 .
\]
Using Lemma~\ref{lem:key-lem-coupling}, the Cauchy--Schwarz inequality, and Lemma~\ref{lem:prior-bound}, we obtain
\begin{equation}
\begin{aligned}
\mathbb{J}_1
& =
\mathbb{E}
\Big[
\mathbb{E}
\Big[
\big\|
Z_{t_{n,i}}
-
{\xi}_{t_{n,i}}
+
V(x,y,B_{t_{n,i}})
\big\|^2
\big|
\tilde{X}_{t_{n,i}}^N=x, X_{t_n}^N=y
\Big]
\Big]\\
& \leq
\mathbb{E}
\Big[
\mathbf{1}_{\{5\beta_{t_{n,i}}^2>1\}} \cdot 2 
\nu_{t_{n,i}}
+
\tfrac{11}{2}
\nu_{t_{n,i}}^2
\Big]\\
& \leq
h
\sqrt{\mathbb{P}(5\beta_{t_{n,i}}^2>1)}
\Big(
\mathbb{E}
\big[
\big\|
b_{t_{n,i}}(\tilde{X}_{t_{n,i}}^N)
-
b_{t_n}(X_{t_n}^N)
\big\|^4
\big]
\Big)^{1/2}\\
& \quad 
+
\tfrac{11h^2}{8}
\mathbb{E}
\Big[
\big\|
b_{t_{n,i}}(\tilde{X}_{t_{n,i}}^N)
-
b_{t_n}(X_{t_n}^N)
\big\|^4
\Big]\\
& \leq
C(\mathcal{M}_2)
T^5
\big(
k d^{2} h^{4}
+
d^2h^4
\big),
\end{aligned}
\end{equation}
where the last step follows from Markov's inequality:
\[
\mathbb{P}(5\beta_{t_{n,i}}^2>1)
\leq 
\mathbb{E}
\big[(5\beta_{t_{n,i}}^2)^{2}\big]
\le 
C(\mathcal{M}_2)
T^{5}
k^2 d^2 h^{4}.
\]
For $\mathbb{J}_2$, the Lipschitz bound on $b_t$ and Lemma \ref{lem:prior-bound} imply
\begin{equation}
\begin{aligned}
\mathbb{J}_2
& \leq
\tfrac{(2 L+1)^2 h^2}{k^2}
\E
\Big[
\big\|
\tilde{X}_{t_{n,i}}^N
-
X_{t_{n,i}}^N
\big\|^2
\Big]
\leq
\tfrac{C(\mathcal{M}_2)T^{5/2} d h^5}{k^2} 
 \leq 
\tfrac{C(\mathcal{M}_2) T^{5/2}d h^4}{k}.
\end{aligned}
\end{equation}
Combining the estimates of $\mathbb{J}_1$ and $\mathbb{J}_2$ yields
\begin{equation}
\mathbb{E}
\Big[
\big\|
M_{t_{n,i}}
\big\|^2
\Big]
\leq
C(\mathcal{M}_2)
T^{5}
\big(
d^{2} h^{5}
+
\tfrac{d^2 h^5}{k}
+
\tfrac{dh^4}{k}
\big).
\end{equation}
With regard to the last term of \eqref{eq:iteration_error}, we use Lemma \ref{lem:prior-bound} to directly deduce
\begin{equation}
\begin{aligned}
\tfrac{(2 L+1)^2h}{\lambda k}
        \mathbb{E}
        \Big[
        \big\|
        X_{t_{n,i}}^N
         -
        \tilde{X}_{t_{n,i}}^N
        \big\|^2
        \Big]
\leq
\tfrac{C(\mathcal{M}_2)
T^{5} dh^4}
{\lambda k}.
\end{aligned}
\end{equation}

\paragraph{Step IV: global recursion.}

By virtue of Proposition \ref{prop:Upper-mean-square error-bounds} and the above estimates, iterating and invoking the fact that $\overleftarrow{X}_0 = X_0^N$, we obtain
\begin{equation}
\label{eq:iterate-poissson}
\begin{aligned}
&
\mathbb{E}
\big[
\|\varepsilon_{T}\|^2
\big]\\
& \leq
\gamma_{(Nk-1)h/k} 
\mathbb{\,E}
\big[
\|
\varepsilon_{(Nk-1)h/k}
\|^2
\big]
+
C(\mathcal{M}_2)
T^5
\Big(
\tfrac{d h^3}{k^3}
+
d^{2} h^{5}
+
\tfrac{d^2 h^5}{k}
+
\tfrac{dh^4}{k}
\Big)\\
& \leq  
\prod_{\ell=0}^{Nk-1}
\gamma_{\ell h/k}
\mathbb{\,E}
\big[
\|\varepsilon_{0}\|^2
\big]
+
C(\mathcal{M}_2)
T^5
\Big(
\tfrac{d h^3}{k^3}
+
d^{2} h^{5}
+
\tfrac{d^2 h^5}{k}
+
\tfrac{dh^4}{k}
\Big)
\Big(
\sum_{w=1}^{Nk-1} 
\prod_{\ell=w}^{Nk-1}  
\gamma_{\ell h/k}
+1
\Big).
\end{aligned}
\end{equation}
Now we are left with the estimation of $\big(
\sum_{w=1}^{Nk-1} 
\prod_{\ell=w}^{Nk-1}  
\gamma_{\ell h/k}
+1
\big)$. 
%
Note that if $h$ satisfies \eqref{eq:h-condition}, then we can take $\rho^{(2)}_h:= 3(2L+1)^2 h/2+\lambda$, such that $\rho^{(2)}_h \in (\lambda,1).$ 
Otherwise said, we have 
\begin{equation}
\tfrac{h}{k}
\leq   
h 
=
\frac{2\rho_h^{(2)}-2\lambda}
{3(2L+1)^2}
\leq
\frac{2(2C_{T-{t_{n,i}}}+1)-2\lambda}
{3(2L+1)^2},
\end{equation}
i.e., $\gamma_{nh+ih/k }\leq 1 \text { for } nk+i  \leq N_h.$
Whereas, for $nk+i  > N_h $, Lemma \ref{lem:score-lipschitz} yields 
\begin{equation}
    \gamma_{t_n}
    =
    1+
    \big(
    \tfrac{3(2L+1)^2 h}{k}
    +2(2L+1)+ 2\lambda\big)
    \tfrac{h}{k} 
    \leq 
    1+
    \big(
    5(2L+1)^2+2\lambda
    \big)
    \tfrac{h}{k}.
\end{equation}
Combining the previous remarks, we have
\begin{equation}
\label{eq:recusion}
\begin{aligned}
&
\sum_{w=1}^{Nk-1} 
\prod_{\ell=w}^{Nk-1}  
\gamma_{\ell h/k}
+1 \\
& =
\sum_{w=1}^{N_h-1} 
\prod_{\ell=w}^{N_h-1} 
\gamma_{\ell h/k} 
\times 
\prod_{\ell=N_h}^{Nk-1} 
\gamma_{\ell h/k}
+
\sum_{w=N_h}^{Nk-1} 
\prod_{\ell=w}^{Nk-1} 
\gamma_{\ell h/k} 
+1\\
& \leq 
(N_h-1)
\prod_{\ell=N_h}^{Nk-1} 
\gamma_{\ell h/k}
        +
        \sum_{w=N_h}^{Nk-1} \prod_{\ell=w}^{Nk-1} \gamma_{\ell h/k}
        +1
        \\
        & \leq 
        (N_h-1) 
        \Big(
        1+
        \big(
        5(2L+1)^2+2\lambda
        \big)
        \tfrac{h}{k}
        \Big)^{Nk-N_h}
        +
        \sum_{w=N_h}^{Nk-1}
        \Big(
        1+
        \big(
        5(2L+1)^2+2\lambda
        \big)
        \tfrac{h}{k}
        \Big)^{Nk-w} 
        +1\\
        & \leq 
        (N_h-1)
        \Big(
        1+
        \big(
        5(2L+1)^2+2\lambda
        \big)
        \tfrac{h}{k}
        \Big)^{Nk-N_h}
        +
        \sum_{w=1}^{Nk-N_h}
        \Big(
        1+
        \big(
        5(2L+1)^2+2\lambda
        \big)
        \tfrac{h}{k}
        \Big)^w
        +1\\
        & =
        (N_h-1)
        \Big(
        1+
        \big(
        5(2L+1)^2+2\lambda
        \big)
        \tfrac{h}{k}
        \Big)^{Nk-N_h}\\
        & \quad +
        \tfrac{
        \big(
        1+
        (
        5(2L+1)^2+2\lambda
        )
        \frac{h}{k}
        \big)
        \big(1-\big(
        1+
        (
        5(2L+1)^2+2\lambda
        )
        \frac{h}{k}
        \big)^{Nk-N_h}\big)}
        {1-\big(
        1+
        (
        5(2L+1)^2+2\lambda
        )
        \frac{h}{k}
        \big)}
        +1\\
        & \leq
        \left(
        N_h
        +
        \tfrac{k}
        {(5(2L+1)^2+2\lambda)h}
        \right) \times
        \Big(
        1+
        \big(
        5(2L+1)^2+2\lambda
        \big)
        \tfrac{h}{k}
        \Big)^{Nk-N_h+1} \\
        & \leq
        \left(
        \tfrac{(T-\eta(\alpha, M,\rho_h ))k}
        {h}
        +
        \tfrac{k}
        {(5(2L+1)^2+2\lambda)h}
        \right) 
        \mathrm{e}^{
        (
        5(2L+1)^2+2\lambda
        )
        \eta(\alpha, M,\rho_h)}.
    \end{aligned}
\end{equation}
Putting the above estimate into \eqref{eq:iterate-poissson}, we get
\begin{equation}
\begin{aligned}
\mathcal{W}^{2}_2
\big(
\mu,
\mathcal{L}(X_T^N)
\big)
&\leq
\mathbb{E}
\Big[
\big\|
\overleftarrow{X}_{t_N}
-
X^N_{t_N}
\big\|^2
\Big] \\
& \leq
C(\mathcal{M}_2)
T^5
\Big(
\tfrac{d h^3}{k^3}
+
d^{2} h^{5}
+
\tfrac{d^2 h^5}{k}
+
\tfrac{dh^4}{k}
\Big)\\
& \quad
\cdot
\left(
\tfrac{(T-\eta(\alpha, M,\rho^{(2)}_h ))k}
        {h}
        +
        \tfrac{k}
        {(5(2L+1)^2+2\lambda)h}
        \right) 
        \mathrm{e}^{
        (
        5(2L+1)^2+2\lambda
        )
        \eta(\alpha, M,\rho^{(2)}_h)}\\
& \leq
C(\mathcal{M}_2)
T^6
\big(
\tfrac{d h^2}{ k^2}
+
k d^{2} h^{4}
+
d^2 h^4
+
d h^3
\big)
\mathrm{e}^{(5(2L+1)^2+2\lambda)
\eta(\alpha, M,\rho^{(2)}_h)},
\end{aligned}
\end{equation}
where, in the last inequality, we have used the fact that 
\begin{equation}
\hat{T}\big(\alpha,M,\rho^{(2)}_h\big)
=
T-
\eta(\alpha,M,\rho^{(2)}_h)
<
T-
\eta\big(\alpha,M,0\big)
=
\hat{T}(\alpha,M,0)
\leq
T.
\end{equation}
Now, letting  $\epsilon_2 $ tend to zero and applying Fatou's lemma, we complete the proof.
\end{proof}

\subsection{Discretization error of EM method}
\label{appen:Discretization-error-em}
In this subsection, we present an enhanced discretization error bound for the EM method. 
%
{\color{black}
When \(k=1\), the midpoint correction term vanishes, and the PM scheme \eqref{eq:poisson-method}
reduces to the classical EM scheme \eqref{eq:EM}.
}
Therefore, 
{\color{black}it is straightforward to obtain the following proposition with the same technique as in Proposition \ref{prop:pi-poisson}}

\begin{prop}
[Discretization error of EM method]
\label{prop:pi-em}
Suppose that Assumptions \ref{ass:weak-convex} and \ref{ass:one-side-lipschitz} hold. 
For any $n \in [N]_0$, consider the process $\{X^N_{t_{n}}\}$ generated by \eqref{eq:em-aux} using the uniform step-size. For any $\lambda\in (0,1)$, let the step-size satisfy
\begin{equation}
\label{eq:em-h-condition}
h \leq
\tfrac{2-\lambda}{2(2L+1)^2}.
\end{equation}
Then there exists a constant $C$, independent of $d$ and $h$, such that,
\begin{equation}
\begin{aligned}
& 
\mathcal{W}_2
\big(
\mu,
\mathcal{L}(X^N_{t_N})
\big)
\leq 
C
\mathrm{e}^{C
\eta(\alpha, M,(2L+1)^2h+\lambda/2)}
\sqrt{T d}h,
\end{aligned}
\end{equation}
with $\eta(\alpha, M,(2L+1)^2h+\lambda/2)$ as defined in \eqref{eq:eta-def}, where $L$ comes from Lemma \ref{lem:score-lipschitz}.
\end{prop}

\begin{proof}

Using the same arguments as used in Proposition \ref{prop:Upper-mean-square error-bounds}, we first recast the backward SDE \eqref{eq:backward-SDE} as
\begin{equation}
\overleftarrow{X}_{t_{n+1}}
=
\overleftarrow{X}_{t_{n}}
+
h
b_{t_n}(\overleftarrow{X}_{t_n})
+
\sqrt{2}
\Delta W_{t_n}
+
\mathcal{R}_{t_{n}},
\end{equation}
where 
\begin{equation}
\mathcal{R}_{t_{n}}
=
\int_{t_n}^{t_{n+1}}
\big[
b_{s}(\overleftarrow{X}_s)
-
b_{t_n}(\overleftarrow{X}_{t_n})
\big] 
\mathrm{\,d}s.
\end{equation}
Using the same notation as before {\color{black} yields}, 
for any $n \in [N]_0$,
\begin{equation}
\varepsilon_{t_{n}}
:= 
\overleftarrow{X}_{t_{n}}-X^N_{t_{n}}.
\end{equation}
Subtracting \eqref{eq:em-aux} from this and taking expectations on both sides 
{\color{black}lead to}
\begin{equation}
\label{eq:error-em}
\begin{aligned}
\E
\big[
\|
\varepsilon_{t_{n+1}}
\|^2
\big]
& =  
\E
\big[
\|
\varepsilon_{t_{n}}
\|^2
\big]
+ 
h^2
\,
\E
\Big[
\big\|
b_{t_n}
(\overleftarrow{X}_{t_n})
-
b_{t_n}
(X^N_{t_n})
\big\|^2
\Big]
+
\E
\Big[
\big\|
\mathcal{R}_{t_{n}}
\big\|^2
\Big]\\
& \quad +
2h
\mathbb{\,E} 
\Big[
\big  \langle
\varepsilon_{t_{n}},
b_{t_n}
(\overleftarrow{X}_{t_n})
-
b_{t_n}
(X^N_{t_n})
\big \rangle
\Big]
+
2
\mathbb{\,E} 
\Big[
\big  \langle
\varepsilon_{t_{n}},
\mathcal{R}_{t_n}
\big \rangle
\Big]\\
& \quad +
2h
\mathbb{\,E} 
\Big[
\big  \langle
b_{t_n}
(\overleftarrow{X}_{t_n})
-
b_{t_n}
(X^N_{t_n}),
\mathcal{R}_{t_n}
\big \rangle
\Big].
\end{aligned}
\end{equation}
For the estimate of the forth term of \eqref{eq:error-em}, using Proposition \ref{prop:regime-shift-drift} gets, 
\begin{equation}
\begin{aligned}
\label{eq:em-cross-term-est}
2h
\mathbb{\,E} 
\Big[
\big  \langle
\varepsilon_{t_{n}},
b_{t_n}
(\overleftarrow{X}_{t_n})
-
b_{t_n}
(X^N_{t_n})
\big \rangle
\Big]
& \leq 
\begin{cases}
-2
(2C_{T-t_n}+1)
h
\,
\E
\big[
\|
\varepsilon_{t_{n}}
\|^2
\big], 
\quad & \text{for} \quad n \leq N_h,
\\[5pt]
2
(2L+1)
h
\,
\E
\big[
\|
\varepsilon_{t_{n}}
\|^2
\big],
& \text{for}\quad  n > N_h.
\end{cases}
\end{aligned}
\end{equation}
The remaining terms are handled following the same approach as in the proof of Proposition \ref{prop:pi-poisson}, namely by employing the Young inequality and the Lipschitz condition of $b_t(\cdot)$.
Plugging the previous estimates into \eqref{eq:error-em} yields
\begin{equation}
\label{eq:error-em-r}
\begin{aligned}
\E
\big[
\|
\varepsilon_{t_{n+1}}
\|^2
\big]
& \leq  
\gamma_{t_n}
\,
\E
\big[
\|
\varepsilon_{t_{n}}
\|^2
\big]
+
2\,
\E
\big[
\|
\mathcal{R}_{t_{n}}
\|^2
\big]
+
\tfrac{1}{\lambda h} 
\mathbb{E}
\Big[
\big\|
\mathbb{E}
\big[
\mathcal{R}_{t_n} 
\big| 
\mathcal{F}_{t_{
n}}
\big]
\big\|
\Big]^2,
\end{aligned}
\end{equation}
where
\begin{align}
\gamma_{t_{n}}
: =
\begin{cases}
1+
\big(
2(2L+1)^2h-2(2C_{T-t_n}+1) +\lambda\big)h,
& \text { for }
\quad n \leq N_h 
\\[5pt]
1+
\big(
4(2L+1)^2+\lambda\big)h,
& \text { for } 
\quad n > N_h .
\end{cases}
\end{align}
Following the same arguments as used in the estimate of \eqref{eq:R-error-bound} and \eqref{eq:R-conditionded-error-bound} in Proposition \ref{prop:pi-poisson}, one derives 
\begin{equation}
2\,
\E
\big[
\|
\mathcal{R}_{t_{n}}
\|^2
\big]
\leq 
2 
\big(
\mathcal{M}_1 +4(1+4 L^2) \big)
d h^{3},
\quad 
\tfrac{1}{\lambda h}
\mathbb{E}
\Big[
\big\|
\mathbb{E}
[
\mathcal{R}_{t_n}\mid \mathcal{F}_{t_n}
]
\big\|^2
\Big]
\leq 
\tfrac{\mathcal{M}_1d h^3}
{\lambda }.
\end{equation}
Since $h$ satisfies \eqref{eq:em-h-condition}, then we take $\rho^{(3)}_h:=(2(2L+1)^2h+\lambda) /2$, such that $\rho^{(3)}_h \in (\lambda/2,1)$. Then, we have
\[
h= 
\frac{2\rho_h-\lambda}
{2(2L+1)^2}
\leq 
\frac{2(2C_{T-t_n}+1)-\lambda}
{2(2L+1)^2},
\]
i.e., $\gamma_{t_n}\leq 1$ for $n \leq N_h.$
Combining the above estimates with \eqref{eq:error-em-r} yields 
\begin{equation}
\begin{aligned}
\mathcal{W}^{2}_2
\big(
\mu,
\mathcal{L}(X^N_{t_N})
\big) 
&\leq
\E
\big[
\|
\varepsilon_{t_{n+1}}
\|^2
\big] \\
& \leq  
\gamma_{t_n}
\,
\E
\big[
\|
\varepsilon_{t_{n}}
\|^2
\big]
+
C_{\mathcal{M}_1,L,\lambda}d h^3\\
& \leq
\prod_{\ell=0}^{N-1} 
\gamma_{t_\ell}
\,
\E
\big[
\|
\varepsilon_{0}
\|^2
\big]
+
C_{\mathcal{M}_1,L,\lambda}d h^3
\big(
\sum_{m=1}^{N-1} 
\prod_{\ell=m}^{N-1}  
\gamma_{t_\ell}
+1
\big)\\
& \leq
C_{\mathcal{M}_1,L,\lambda}d h^2
\left(
\hat{T}(\alpha,M,0)
+
\tfrac{1}
{C_{L,\lambda}}
\right) 
\mathrm{e}^{(
4(2L+1)^2+\lambda)
\eta(\alpha, M,\rho^{(3)}_h)},
\end{aligned}
\end{equation}
where we also used 
$\overleftarrow{X}_{0} - X^N_{0} = 0$
and an argument similar to the proof of \eqref{eq:recusion}.
The proof of Proposition \ref{prop:pi-em} is completed.
\end{proof}

\section{Bound on initialization error}
Before discussing the initialization error, we mention that we consider the coupling between the continuous time interpolations $\{X^N_{t_{n,j}}\}_{n \in[N]_0, j \in [k-1]_0}$ and $\{X^\infty_{t_{n,j}}\}_{n \in[N]_0, j \in [k-1]_0}$ with initialization satisfying 
\[
\mathcal{W}^2_2
\big(
\mathcal{L}(\overrightarrow{X}_T), \gamma^d
\big)
=
\E
\Big[
\big\|
{X}^N_{0}
-
X^\infty_{0}
\big\|^2\Big].
\]
We are now ready to present an initialization error bound for two types of methods.

\subsection{Initialization error of PM method}
\begin{prop}
[Initialization error of PM method]
\label{prop:pi-X_T}
Suppose that Assumptions \ref{ass:weak-convex} and \ref{ass:one-side-lipschitz} hold. For any $n \in [N]_0, \,j\in[k-1]_0$,
let $\{X^N_{t_{n,j}}\}$ and $\{X^\infty_{t_{n,j}}\}$ be produced by \eqref{eq:poisson-aux} with the uniform step-size 
\begin{equation}
\label{eq:h-condition-initial-error}
0<
h <
\tfrac{1}{C(L,\alpha)},
\end{equation}
where $C(L,\alpha)$ comes from \eqref{eq:C-def}.
Then there exists a constant $C$, independent of $d$ and $h$, such that,
\begin{equation}
\mathcal{W}_2
\big(
\mathcal{L}(X^N_{t_N}),
\mathcal{L}(X^{\infty}_{t_N})
\big)
\leq 
C
\mathrm{e}^{C
\eta(\alpha, M,C(L,\alpha)h)}
\mathrm{e}^{-T}
\mathcal{W}_2(\mathrm{\mu},\gamma^d),
\end{equation}
with $\eta\big(\alpha, M,C(L,\alpha)h\big)$ as defined in \eqref{eq:eta-def}, where $L$ comes from Lemma \ref{lem:score-lipschitz}.
\end{prop}

\begin{proof}
First, according to \eqref{eq:poisson-aux}, we have, for any $n \in [N-1]_0, \, i\in[k-1]_0$,
\begin{equation}
\label{eq:poisson-aux-em}
\begin{aligned}
{X}_{t_{n,i+1}}^N
& =
{X}_{t_n}^N
+
\tfrac{(i+1)h}{k}
b_{t_n}({X}_{t_{n}}^N)
+
h 
\sum_{j=0}^i
B_{t_{n,j}}
\Big(
b_{t_{n,j}}
(\tilde{X}_{t_{n,j}}^N)
-
b_{t_n}({X}_{t_{n}}^N)
\Big)
+
\sqrt{\tfrac{2h}{k}}
\sum_{j=0}^i
Z_{t_{n,j}}\\
& =
{X}_{t_n}^N
+
\Big(
\tfrac{(i+1)h}{k}
-
h
\sum_{j=0}^i
B_{t_{n,j}}
\Big)
b_{t_n}({X}_{t_{n}}^N)
+
h 
\sum_{j=0}^i
B_{t_{n,j}}
b_{t_{n,j}}
(\tilde{X}_{t_{n,j}}^N)
+
\sqrt{\tfrac{2h}{k}}
\sum_{j=0}^i
Z_{t_{n,j}},
\end{aligned}
\end{equation}
and 
\begin{equation}
\label{eq:poisson-x_0_pi-em}
\begin{aligned}
{X}_{t_{n,i+1}}^\infty
& =
{X}_{t_n}^\infty
+
\tfrac{(i+1)h}{k}
b_{t_n}({X}_{t_{n}}^\infty)
+
h 
\sum_{j=0}^i
B_{t_{n,j}}
\Big(
b_{t_{n,j}}
(\tilde{X}_{t_{n,j}}^\infty)
-
b_{t_n}({X}_{t_{n}}^\infty)
\Big)
+
\sqrt{\tfrac{2h}{k}}
\sum_{j=0}^i
Z_{t_{n,j}}\\
& =
{X}_{t_n}^\infty
+
\Big(
\tfrac{(i+1)h}{k}
-
h
\sum_{j=0}^i
B_{t_{n,j}}
\Big)
b_{t_n}({X}_{t_{n}}^\infty)
+
h 
\sum_{j=0}^i
B_{t_{n,j}}
b_{t_{n,j}}
(\tilde{X}_{t_{n,j}}^\infty)
+
\sqrt{\tfrac{2h}{k}}
\sum_{j=0}^i
Z_{t_{n,j}}.
\end{aligned}
\end{equation}
{\color{black}
Taking $i=k-1$ in \eqref{eq:poisson-aux-em} and \eqref{eq:poisson-x_0_pi-em} gives, for any $n\in [N-1]_0$,}
\begin{equation}
\label{eq:x_n-x_infty}
\begin{aligned}
{X}^N_{t_{n+1}}
-
X^\infty_{t_{n+1}}
& =  
{X}^N_{t_n}
-
X^\infty_{t_n}
+ 
h
\Big(
1
-
\sum_{j=0}^{k-1}
B_{t_{n,j}}
\Big)
\Big(
b_{t_n}
({X}^N_{t_n})
-
b_{t_n}
(X^\infty_{t_n})
\Big)\\
& \quad +
h
\sum_{j=0}^{k-1}
B_{t_{n,j}}
\Big(
b_{t_{n,j}}
(\tilde{X}_{t_{n,j}}^N)
-
b_{t_{n,j}}
(\tilde{X}_{t_{n,j}}^\infty)
\Big).
\end{aligned}
\end{equation}
To simplify the notation, we denote, for any $n \in [N]_0, 
\, j \in [k-1]_0,$
\[
e_{t_n}
:=
{X}^N_{t_n}
-
X^\infty_{t_n},
\quad 
\mathcal{M}_{t_{n,j}}
:=
B_{t_{n,j}}
\Big(
b_{t_{n,j}}
(\tilde{X}^N_{t_{n,j}})
-
b_{t_{n,j}}
(\tilde{X}^\infty_{t_{n,j}})
\Big). 
\]
Squaring both sides of the above equation, taking expectations and using the Cauchy-Schwarz inequality show,
\begin{equation}
\begin{aligned}
\E
\big[
\|
e_{t_{n+1}}
\|^2
\big]
& = 
\E
\big[
\|
e_{t_n}
\|^2
\big]
+
h^2
\mathbb{\,E}
\Big[
\big(
1
-
\sum_{j=0}^{k-1}
B_{t_{n,j}}
\big)^2
\Big]
\mathbb{\,E}
\Big[
\big\|
b_{t_{n}}
({X}^N_{t_{n}})
-
b_{t_{n}}
(X^\infty_{t_{n}})
\big\|^2
\Big]\\
& \quad +
h^2
\,
\mathbb{E}
\Big[
\big\|
\sum_{j=0}^{k-1}
\mathcal{M}_{t_{n,j}}
\big\|^2
\Big]  
+
2h
\mathbb{\,E} 
\Big[
\big \langle
e_{t_n},
\sum_{j=0}^{k-1}
\mathcal{M}_{t_{n,j}}
\big \rangle
\Big]\\
& \quad +
2h
\mathbb{\,E} 
\bigg[
\Big  \langle
e_{t_n},
\Big(
1
-
\sum_{j=0}^{k-1}
B_{t_{n,j}}
\Big)
\Big(
b_{t_{n}}
({X}^N_{t_{n}})
-
b_{t_{n}}
(X^\infty_{t_{n}})
\Big)
\Big \rangle
\bigg]\\
& \quad
+
2h^2
\mathbb{\,E} 
\bigg[
\Big \langle
\Big(
1
-
\sum_{j=0}^{k-1}
B_{t_{n,j}}
\Big)
\Big(
b_{t_{n}}
({X}^N_{t_{n}})
-
b_{t_{n}}
(X^\infty_{t_{n}})
\Big),
\sum_{j=0}^{k-1}
\mathcal{M}_{t_{n,j}}
\Big\rangle
\bigg]\\
& \leq
\E
\big[
\|
e_{t_n}
\|^2
\big]
+
2h^2
\mathbb{\,E}
\Big[
\big(
1
-
\sum_{j=0}^{k-1}
B_{t_{n,j}}
\big)^2
\Big]
\mathbb{\,E}
\Big[
\big\|
b_{t_{n}}
({X}^N_{t_{n}})
-
b_{t_{n}}
(X^\infty_{t_{n}})
\big\|^2
\Big]\\
& \quad +
2h^2
\,
\mathbb{E}
\Big[
\big\|
\sum_{j=0}^{k-1}
\mathcal{M}_{t_{n,j}}
\big\|^2
\Big]  
+
2h
\mathbb{\,E} 
\Big[
\big \langle
e_{t_n},
\sum_{j=0}^{k-1}
\mathcal{M}_{t_{n,j}}
\big \rangle
\Big]\\
& \quad +
2h
\mathbb{\,E} 
\bigg[
\Big  \langle
e_{t_n},
\Big(
1
-
\sum_{j=0}^{k-1}
B_{t_{n,j}}
\Big)
\Big(
b_{t_{n}}
({X}^N_{t_{n}})
-
b_{t_{n}}
(X^\infty_{t_{n}})
\Big)
\Big \rangle
\bigg].
\end{aligned}
\end{equation}
Noting that $\E[\sum_{j=0}^{k-1}B_{t_{n,j}}]=1$ and 
$\E
\big[
\big(\sum_{j=0}^{k-1}
B_{t_{n,j}}
\big)^2
\big] 
\leq
2$, one directly derives that
\begin{equation}
\begin{aligned}
\mathbb{\,E}
\bigg[
\Big(
1
-
\sum_{j=0}^{k-1}
B_{t_{n,j}}
\Big)^2
\bigg]
& =
1-
2\,
\E\Big[\sum_{j=0}^{k-1}
B_{t_{n,j}}\Big]
+
\E
\bigg[
\Big(\sum_{j=0}^{k-1}
B_{t_{n,j}}
\Big)^2
\bigg]
\leq 1,
\end{aligned}
\end{equation}
and 
\begin{equation}
\begin{aligned}
&
2h
\mathbb{\,E} 
\bigg[
\Big  \langle
e_{t_n},
\Big(
1
-
\sum_{j=0}^{k-1}
B_{t_{n,j}}
\Big)
\big(
b_{t_{n}}
({X}^N_{t_{n}})
-
b_{t_{n}}
(X^\infty_{t_{n}})
\big)
\Big \rangle
\bigg]\\
& =
2h
\mathbb{\,E} 
\bigg[
\Big  \langle
e_{t_n},
\big(
b_{t_{n}}
({X}^N_{t_{n}})
-
b_{t_{n}}
(X^\infty_{t_{n}})
\big)
\mathbb{\,E} 
\Big[
\big(
1
-
\sum_{j=0}^{k-1}
B_{t_{n,j}}
\big)
\big|
\sigma
\big({X}^\infty_{t_n},X^N_{t_n}\big)
\Big]
\Big \rangle
\bigg]\\
& =
0.
\end{aligned}
\end{equation}
It remains to estimate 
$2h
\,
\mathbb{\,E} 
\Big[
\big  \langle
e_{t_n},
\sum_{j=0}^{k-1}
\mathcal{M}_{t_{n,j}}
\big \rangle
\Big]$, which can be divided into two cases depending on different ranges
of $n$.

\textbf{Case I:} $(n+1)k \leq N_h$. 
In the case $(n+1)k \leq N_h$, 
we do the following decomposition:
\begin{equation}
\begin{aligned}
&
2h
\,
\mathbb{\,E} 
\Big[
\big  \langle
e_{t_n},
\sum_{j=0}^{k-1}
\mathcal{M}_{t_{n,j}}
\big \rangle
\Big]\\
& =
\underbrace{
2h
\sum_{j=0}^{k-1}
\mathbb{\,E} 
\Big[
\big  \langle
\tilde{X}^N_{t_{n,j}}
-
\tilde{X}^\infty_{t_{n,j}},
B_{t_{n,j}}
\big(
b_{t_{n,j}}
(\tilde{X}^N_{t_{n,j}})
-
b_{t_{n,j}}
(\tilde{X}^\infty_{t_{n,j}})
\big)
\big \rangle
\Big]}_{=:\mathbb{T}_1}\\
& \quad +
\underbrace{
2h
\sum_{j=0}^{k-1}
\mathbb{\,E} 
\Big[
\Big  \langle
{X}^N_{t_n}
-
X^\infty_{t_n}
-
\tilde{X}^N_{t_{n,j}}
+
\tilde{X}^\infty_{t_{n,j}},
B_{t_{n,j}}
\big(
b_{t_{n,j}}
(\tilde{X}^N_{t_{n,j}})
-
b_{t_{n,j}}
(\tilde{X}^\infty_{t_{n,j}})
\big)
\Big \rangle
\Big]}_{=:\mathbb{T}_2}.
\end{aligned}
\end{equation}
The definitions in \eqref{eq:poisson-aux} give, for any $n\in [N]_0$ and $j\in [k-1]_0$, 
\begin{equation}
\label{eq:hat-x-tilde-x}
\tilde{X}^N_{t_{n,j}}
-
\tilde{X}^\infty_{t_{n,j}}
=
{X}^N_{t_{n}}
-
{X}^\infty_{t_{n}}
+
\tfrac{jh}{k}
\Big(
b_{t_n}
({X}^N_{t_{n}})
-
b_{t_n}
({X}^\infty_{t_{n}})
\Big).
\end{equation}
Regarding $\mathbb{T}_1$, 
{\color{black}
using
the identity $\E[B_{t_{n,j}}]=1/k$, the inequality $(a+b)^2\geq \tfrac12 a^2-b^2$, and Proposition~\ref{prop:regime-shift-drift} implies
}
\begin{equation}
\label{eq:T1-est}
\begin{aligned}
\mathbb{T}_1 
& \leq
-\tfrac{2h}{k} 
\sum_{j=0}^{k-1}(2C_{T-{t_{n,j}}}+1)
\mathbb{\,E} 
\Big[
\big\|
\tilde{X}^N_{t_{n,j}}
-
\tilde{X}^\infty_{t_{n,j}}
\big\|^2
\Big]\\
& \leq 
-\tfrac{2h}{k} 
\min_{0\leq j\leq k-1} 
(2C_{T-{t_{n,j}}}+1)
\sum_{j=0}^{k-1}
\Big(
\tfrac{1}{2}
\,
\mathbb{E}
\big[
\|e_{t_n}\|^2
\big]
-
(\tfrac{jh}{k})^2
\mathbb{\,E}
\Big[
\big\|
b_{t_n}
({X}^N_{t_{n}})
-
b_{t_n}
({X}^\infty_{t_{n}})
\big\|^2
\Big]
\Big)\\
& \leq
-h
\min_{0\leq j\leq k-1} (2C_{T-{t_{n,j}}}+1)
\,
\mathbb{E}
\big[
\|e_{t_n}\|^2
\big]
+
2h^3 
(2L+1)^2
(1 \vee 2\alpha -1)
\,
\mathbb{E}
\big[
\|e_{t_n}\|^2
\big],
\end{aligned}
\end{equation}
where the last inequality follows from the fact that, for any $t\in [0,T]$,
\begin{align}
2C_t+1
& =
\frac{2\alpha}
{\alpha + (1 - \alpha)e^{-2t}} - 
\frac{2e^{-2t}M}
{(\alpha + (1 - \alpha)e^{-2t})^2} - 1
\nonumber \\
& \leq
\frac{2\alpha}
{\alpha + (1 - \alpha)e^{-2t}}
-
1 \nonumber\\
& \leq 
1 \vee (2\alpha -1).
\end{align}
For the estimate of $\mathbb{T}_2$, using \eqref{eq:hat-x-tilde-x}, the Young inequality and Lipschitz condition of $b_t(\cdot)$ leads to 
\begin{equation}
\label{eq:T2-est}
\begin{aligned}
\mathbb{T}_2 
& =
2h
\sum_{j=0}^{k-1}
\mathbb{\,E} 
\Big[
\big  \langle
\tfrac{jh}{k}
\big(
b_{t_n}
({X}^\infty_{t_{n}})
-
b_{t_n}
({X}^N_{t_{n}})
\big),
B_{t_{n,j}}
\big(
b_{t_{n,j}}
(\tilde{X}^N_{t_{n,j}})
-
b_{t_{n,j}}
(\tilde{X}^\infty_{t_{n,j}})
\big)
\big \rangle
\Big]\\
& =
2h
\mathbb{\,E} 
\bigg[
\Big  \langle
\big(
b_{t_n}
({X}^\infty_{t_n})
-
b_{t_n}
({X}^N_{t_n})
\big),
\sum_{j=0}^{k-1}
\tfrac{jh}{k}
B_{t_{n,j}}
\big(
b_{t_{n,j}}
(\tilde{X}^N_{t_{n,j}})
-
b_{t_{n,j}}
(\tilde{X}^\infty_{t_{n,j}})
\big)
\Big \rangle
\bigg]\\
& \leq
h^2 
(2L+1)^2
\,
\mathbb{E}
\big[
\|
e_{t_n}\|^2
\big]
+
h^2
\,
\mathbb{E}
\Big[
\big\|
\sum_{j=0}^{k-1}
{\color{black}\tfrac{j}{k}}
B_{t_{n,j}}
\big(
b_{t_{n,j}}
(\tilde{X}^N_{t_{n,j}})
-
b_{t_{n,j}}
(\tilde{X}^\infty_{t_{n,j}})
\big)
\big\|^2
\Big]\\
& =
h^2 
(2L+1)^2
\,
\mathbb{E}
\big[
\|
e_{t_n}\|^2
\big]
+
h^2
\,
\mathbb{E}
\Big[
\big\|
\sum_{j=0}^{k-1}
{\color{black}\tfrac{j}{k}
\mathcal{M}_{t_{n,j}}}
\big\|^2
\Big].
\end{aligned}
\end{equation}
As a result, 
\begin{equation}
\begin{aligned}
2h
\,
\mathbb{\,E} 
\Big[
\big  \langle
e_{t_n},
\sum_{j=0}^{k-1}
\mathcal{M}_{t_{n,j}}
\big \rangle
\Big]
& \leq
-h
\min_{0\leq j\leq k-1} (2C_{T-{t_{n,j}}}+1)
\mathbb{\,E}
\big[
\|e_{t_n}\|^2
\big]\\
& \quad +
h^2
(2L+1)^2
\Big(
1+2\big(1 \vee (2\alpha -1)\big)
\Big)
\,
\mathbb{E}
\big[
\|e_{t_n}\|^2
\big]\\
& \quad +
h^2
\,
\mathbb{E}
\Big[
\big\|
\sum_{j=0}^{k-1}
{\color{black}\tfrac{j}{k}\mathcal{M}_{t_{n,j}}}
\big\|^2
\Big].
\end{aligned}
\end{equation}

\textbf{Case II:} $(n+1)k> N_h$.
In the case $(n+1)k> N_h$, we use the Young inequality to obtain
\begin{equation}
\begin{aligned}
2h
\,
\mathbb{\,E} 
\Big[
\big  \langle
e_{t_n},
\sum_{j=0}^{k-1}
\mathcal{M}_{t_{n,j}}
\big \rangle
\Big]
& \leq
h
\,
\mathbb{E}
\big[
\|e_{t_n}\|^2
\big]
+
h
\,
\mathbb{E}
\Big[
\big\|
\sum_{j=0}^{k-1}
\mathcal{M}_{t_{n,j}}
\big\|^2
\Big].
\end{aligned}
\end{equation}
Next, we cope with $\mathbb{E}
\big[
\|\sum_{j=0}^{k-1}
\mathcal{M}_{t_{n,j}}\|^2
\big]$.
{\color{black}
For any $j\in [k-1]_0$ and $n\in [N]_0$, $\E[B_{t_{n,j}}]=\E[B^2_{t_{n,j}}]=1/k$, one obtains
}
\begin{align}
\mathbb{E}
\Big[
\big\|
\sum_{j=0}^{k-1}
\mathcal{M}_{t_{n,j}}
\big\|^2
\Big]
& =
\mathbb{E}
\Big[
\big\|
\sum_{j=0}^{k-1}
B_{t_{n,j}}
\big(
b_{t_{n,j}}
(\tilde{X}^N_{t_{n,j}})
-
b_{t_{n,j}}
(\tilde{X}^\infty_{t_{n,j}})
\big)
\big\|^2
\Big]\nonumber\\
& =
\underbrace{
\tfrac{2}{k^2}
\sum_{0\leq i<j\leq k-1}
\mathbb{E}
\Big[
\big\langle 
b_{t_{n,j}}
(\tilde{X}^N_{t_{n,j}})
-
b_{t_{n,j}}
(\tilde{X}^\infty_{t_{n,j}}),
b_{t_{n,j}}
(\tilde{X}^N_{t_{n,j}})
-
b_{t_{n,j}}
(\tilde{X}^\infty_{t_{n,j}})
\big\rangle
\Big]}_{=:\mathbb{T}_3}\nonumber\\
& \quad +
\underbrace{
\tfrac{1}{k}
\sum_{j=0}^{k-1}
\mathbb{E}
\Big[
\big\|
b_{t_{n,j}}
(\tilde{X}^N_{t_{n,j}})
-
b_{t_{n,j}}
(\tilde{X}^\infty_{t_{n,j}})
\big\|^2
\Big]}_{=:\mathbb{T}_4}.
\end{align}
In the following we deal with the above two items separately. 
Using the Cauchy-Schwarz inequality and the Lipschitz condition of 
$b_t(\cdot)$, one can derive that
\begin{equation}
\label{eq:T3-est}
\begin{aligned}
\mathbb{T}_3 
& \leq 
\tfrac{2}{k^2}
\sum_{0\leq i<j\leq k-1}
\mathbb{E}
\Big[
\big\|
b_{t_{n,i}}
(\tilde{X}^N_{t_{n,i}})
-
b_{t_{n,i}}
(\tilde{X}^\infty_{t_{n,i}})
\big\|
\cdot
\big\|
b_{t_{n,j}}
(\tilde{X}^N_{t_{n,j}})
-
b_{t_{n,j}}
(\tilde{X}^\infty_{t_{n,j}})
\big\|
\Big]\\
& \leq
\sup_{0\leq i \leq k-1}
\mathbb{E}
\Big[
\big\|
b_{t_{n,i}}
(\tilde{X}^N_{t_{n,i}})
-
b_{t_{n,i}}
(\tilde{X}^\infty_{t_{n,i}})
\big\|^2
\Big]\\
& {\color{black}\leq
(2L+1)^2 
\sup_{0\leq i \leq k-1}
\mathbb{E}
\Big[
\big\|
e_{t_n}
+
\tfrac{ih}{k}
\big(
b_{t_{n}}
({X}^N_{t_{n}})
-
b_{t_{n}}
({X}^\infty_{t_{n}})
\big)
\big\|^2
\Big]}\\
& \leq
2(2L+1)^2 
\big(1+h^2(2L+1)^2\big)
\,
\mathbb{E}
\big[
\|
e_{t_n}\|^2
\big].
\end{aligned}
\end{equation}
Regarding $\mathbb{T}_4$, we use the Lipschitz condition of 
$b_t(\cdot)$ to obtain 
\begin{equation}
\label{eq:T4-est}
\begin{aligned}
\mathbb{T}_4 
\leq 
2(2L+1)^2 
\big(1+h^2(2L+1)^2\big)
\,
\mathbb{E}
\big[
\|
e_{t_n}\|^2
\big].
\end{aligned}
\end{equation}
Equipped with estimates \eqref{eq:T3-est} and \eqref{eq:T4-est}, one  derives that
\begin{align}
\mathbb{E}
\Big[
\big\|
\sum_{j=0}^{k-1}
\mathcal{M}_{t_{n,j}}
\big\|^2
\Big]
\leq
4
\big((2L+1)^2+1 
\big)^2
\,
\mathbb{E}
\big[
\|
e_{t_n}\|^2
\big].
\end{align}
{\color{black}
For the weighted sum, the same expansion and the bounds \(0\leq i/k,j/k\leq1\) give
}
\begin{align}
&{\color{black}\mathbb{E}
\Big[
\big\|
\sum_{j=0}^{k-1}
\tfrac{j}{k}
\mathcal{M}_{t_{n,j}}
\big\|^2
\Big]}\nonumber\\
&{\color{black}\leq
\tfrac{2}{k^2}
\sum_{0\leq i<j\leq k-1}
\tfrac{ij}{k^2} \
\mathbb{E}
\Big[
\big\|
b_{t_{n,i}}
(\tilde{X}^N_{t_{n,i}})
-
b_{t_{n,i}}
(\tilde{X}^\infty_{t_{n,i}})
\big\|
\cdot
\big\|
b_{t_{n,j}}
(\tilde{X}^N_{t_{n,j}})
-
b_{t_{n,j}}
(\tilde{X}^\infty_{t_{n,j}})
\big\|
\Big]}\nonumber\\
&\quad
{\color{black}
+
\tfrac{1}{k}
\sum_{j=0}^{k-1}
\big(\tfrac{j}{k}\big)^2
\mathbb{E}
\Big[
\big\|
b_{t_{n,j}}
(\tilde{X}^N_{t_{n,j}})
-
b_{t_{n,j}}
(\tilde{X}^\infty_{t_{n,j}})
\big\|^2
\Big]}\nonumber\\
&{\color{black}\leq
4
\big((2L+1)^2+1 
\big)^2
\,
\mathbb{E}
\big[
\|
e_{t_n}\|^2
\big].}
\end{align}
Based on the above estimates, we obtain
\begin{equation}
\E
\big[
\|
e_{t_{N}}
\|^2
\big]
\leq
\prod_{\ell=0}^{N-1}
{\color{black}\Gamma_{t_\ell}} 
\E
\big[
\|
e_0
\|^2
\big],
\end{equation}
where 
\begin{equation}
\begin{aligned}
\Gamma_{t_n}
:=
\begin{cases}
1+
\big(C(L,\alpha) h-\min_{0\leq j\leq k-1} (2 C_{T-{t_{n,j}}}+1)
\big)h,
& \text { for } (n+1)k \leq N_h 
\\[5pt]
1+
\tilde{C}(L)h,
& \text { for } (n+1)k > N_h 
\end{cases}
\end{aligned}
\end{equation}
and 
\begin{equation}
\label{eq:C-def}
C(L,\alpha)
:=
\big((2L+1)^2+1
\big)^2
\big(
15+2(1 \vee 2(\alpha-1))
\big), 
\quad
\tilde{C}(L)
:=
14
\big((2L+1)^2+1
\big).
\end{equation}
With regard to 
$\E
\big[
\|
e_0
\|^2
\big]$, it follows from \cite[Proposition C.2]{strasman2025analysis} that,
\begin{equation}
\E
\big[
\|
e_0
\|^2
\big]
=
\mathcal{W}^2_2
\big(
\mathcal{L}(\overrightarrow{X}_T), \gamma^d
\big)
\leq 
\mathrm{e}^{-2T}
\mathcal{W}^2_2
(\mu,
\gamma^d).
\end{equation}
Since $h$ satisfies \eqref{eq:h-condition-initial-error}, we take $\rho_h^{(4)}:=C(L,\alpha)h $. It follows that $\Gamma_{t_n}\leq 1$ whenever $(n+1)k\leq N_h$.
{\color{black}
Using the definition of \(N_h\) with \(\rho=\rho_h^{(4)}\), we also have
\[
\big(N-\lfloor N_h/k\rfloor\big)h
\leq
T-\widehat{T}(\alpha,M,\rho_h^{(4)})+h
=
\eta(\alpha,M,\rho_h^{(4)})+h .
\]
}
Then
\begin{equation}
\label{eq:exp}
\prod_{\ell=0}^{N-1}
\Gamma_{t_\ell}
\leq
\prod_{\ell=N_h}^{N-1}
\Gamma_{t_\ell}
\leq
\big(
1+\tilde{C}(L)  h
\big)^{N-N_h}
\leq
\mathrm{e}^{\tilde{C}(L) \eta(\alpha, M,\rho_h^{(4)})}.
\end{equation}
All together, we have
\begin{equation}
\mathcal{W}_2
\big(
\mathcal{L}(X^N_{t_N}),
\mathcal{L}(X^{\infty}_{t_N})
\big)
\leq 
C
\mathrm{e}^{\tilde{C}(L) 
\eta(\alpha, M,\rho_h^{(4)})}
\mathrm{e}^{-T}
\mathcal{W}_2(\mathrm{\mu},\gamma^d).
\end{equation}
The proof of Proposition \ref{prop:pi-X_T} is completed.
\end{proof}

\subsection{Initialization error of EM method}
The aim of this subsection is to prove the initialization error of EM method, by utilizing the same technique in Proposition \ref{prop:pi-X_T}.

\begin{prop}
[Initialization error of EM method]
\label{prop:em-pi-X_T}
Suppose that Assumptions \ref{ass:weak-convex} and \ref{ass:one-side-lipschitz} hold. For any $n \in [N]_0$ and 
let $\{X^N_{t_{n}}\}$ and $\{X^\infty_{t_{n}}\}$ be produced by \eqref{eq:em-aux} using the uniform step-size 
\begin{equation}
h \leq
\tfrac{2}{(2L+1)^2}.
\end{equation}
Then there exists a constant $C$, independent of $d$ and $h$, such that,
\begin{equation}
\mathcal{W}_2
\big(
\mathcal{L}(X^N_{t_N}),
\mathcal{L}(X^{\infty}_{t_N})
\big)
\leq 
\mathrm{e}^{C
\eta(\alpha, M,(2L+1)^2h/2)}
\mathrm{e}^{-T}
\mathcal{W}_2(\mathrm{\mu},\gamma^d),
\end{equation}
with $\eta(\alpha, M,(2L+1)^2h/2)$ as defined in \eqref{eq:eta-def}, where $L$ comes from Lemma \ref{lem:score-lipschitz}.
\end{prop}

\begin{proof} 
By following the proof of Proposition \ref{prop:pi-X_T}, using the same notation and arguments as in the estimate of \eqref{eq:exp}, and taking $\rho^{(5)}_h:= (2L+1)^2h/2$, one will arrive at
\begin{equation}
\begin{aligned}
\E
\big[
\|
e_{t_{N }}
\|^2
\big]
& \leq
\prod_{\ell=0}^{N-1}
\Gamma_{t_\ell} 
\E
\big[
\|
e_{0}
\|^2
\big]\\
& \leq 
\mathrm{e}^{(2L+3)^2
\eta(\alpha, M,\rho^{(5)}_h)}
\mathrm{e}^{-2T}
\mathcal{W}^2_2(\mathrm{\mu},\gamma^d).
\end{aligned}
\end{equation}
where 
\begin{equation}
\begin{aligned}
\Gamma_{t_n}
:=
\begin{cases}
1+
\big((2L+1)^2 h-2(2 C_{T-{t_{n}}}+1)
\big)h,
& \text { for } n \leq N_h 
\\[5pt]
1+
(2L+1)(2L+3)h,
& \text { for } n > N_h.
\end{cases}
\end{aligned}
\end{equation}
The proof is completed.
\end{proof}

\section{Bound on score matching error }

\subsection{Score matching error of PM method}
\begin{prop}
[Score matching error of PM method]
\label{prop:score-error}
Suppose that Assumptions \ref{ass:weak-convex}, \ref{ass:one-side-lipschitz} and \ref{ass:score-function-approximation} hold. 
For any $n \in [N]_0, \,j\in[k-1]_0$ and 
let $\{X^\infty_{t_{n,j}}\}$ and $\{Y_{t_{n,j}}\}$ be produced by \eqref{eq:poisson-aux} and \eqref{eq:poisson-method} using the uniform step-size 
\[
h \leq 
\tfrac{1}{C(L,\alpha)},
\]
where $C(L,\alpha)$ comes from \eqref{eq:C-def}.
Then there exists a constant $C$, independent of $d$ and $h$, such that,
\begin{equation}
\mathcal{W}_2
\big(
\mathcal{L}(X^{\infty}_{t_N}),
\mathcal{L}(Y_{t_N})
\big)
\leq 
C
\mathrm{e}^{{C}
\eta(\alpha, M,C(L,\alpha)h)}
T
\epsilon_0,
\end{equation}
with $\eta(\alpha, M,C(L,\alpha)h)$ as defined in \eqref{eq:eta-def}, where $L$ comes from Lemma \ref{lem:score-lipschitz}.
\end{prop}

\begin{proof}

First, we introduce a new auxiliary process as, for any $ n \in [N]_0, 
\, i \in [k-1]_0$,
\begin{equation}
\widehat{Y}_{t_{n,i}}
=
{Y}_{t_n}
+
\tfrac{ih}{k}
b_{t_n}({Y}_{t_n})
+
\sum_{j=0}^{i-1}
\sqrt{\tfrac{2h}{k}} 
Z_{t_{n,j}}.
\end{equation}
Next, we recast the second term of \eqref{eq:poisson-method} as
\begin{align}
\label{eq:Y-rewritten-score-error}
Y_{t_{n+1}}
& =
Y_{t_n}
+
h
\big(
1
-
\sum_{j=0}^{k-1}
B_{t_{n,j}}
\big)
\big(
-Y_{t_n}
+
2\tilde{s}_{\theta^\star}(T-t_n,Y_{t_n})
\big)\nonumber\\
& \quad +
h
\sum_{j=0}^{k-1}
B_{t_{n,j}}
\Big(
-\tilde{Y}_{t_{n,j}}
+
2\tilde{s}_{\theta^\star}(T-t_{n,j},\tilde{Y}_{t_{n,j}})
\Big)
+
\sum_{j=0}^{k-1}
\sqrt{\tfrac{2h}{k}}
Z_{t_{n,j}}
\nonumber\\
& =
Y_{t_n}
+
h
\big(
1
-
\sum_{j=0}^{k-1}
B_{t_{n,j}}
\big)
b_{t_n}(Y_{t_n})
+
h
\sum_{j=0}^{k-1}
B_{t_{n,j}}
b_{t_{n,j}}
(\widehat{Y}_{t_{n,j}})
+
\sum_{j=0}^{k-1}
\sqrt{\tfrac{2h}{k}}
Z_{t_{n,j}}
\nonumber\\
& \quad +
h
\big(
1
-
\sum_{j=0}^{k-1}
B_{t_{n,j}}
\big)
\Big(
2\tilde{s}_{\theta^\star}(T-t_{n},{Y}_{t_n})
-
2 \nabla \log \tilde{p}_{T-t_n}(Y_{t_n})
\Big)\nonumber\\
& \quad +
h
\sum_{j=0}^{k-1}
B_{t_{n,j}}
\Big(
-\tilde{Y}_{t_{n,j}}
+
2\tilde{s}_{\theta^\star}(T-t_{n,j},\tilde{Y}_{t_{n,j}})
-
b_{t_{n,j}}
(\widehat{Y}_{t_{n,j}})
\Big).
\end{align}
For simplicity, we denote 
\[
\mathcal{E}_{t_n}
:=
X^\infty_{t_{n}}
-
Y_{t_{n}},
\quad 
 n \in [N]_0.
\]
Subtracting \eqref{eq:poisson-x_0_pi-em} from \eqref{eq:Y-rewritten-score-error}, taking the norm of both sides of the inequality, and using the triangle inequality yield
\begin{align}
\big\|
\mathcal{E}_{t_{n+1}}
\big\|_{L^2(\Omega;\mathbb{R}^d)}
& =
\Big\|
\mathcal{E}_{t_n}
+
h
\big(
1
-
\sum_{j=0}^{k-1}
B_{t_{n,j}}
\big)
\Big(
b_{t_n}
(X^\infty_{t_n})
-
b_{t_n}
(Y_{t_n})
\Big)\nonumber\\
& \quad +
h
\sum_{j=0}^{k-1}
B_{t_{n,j}}
\Big(
b_{t_{n,j}}
(\tilde{X}_{t_{n,j}}^\infty)
-
b_{t_{n,j}}
(\widehat{Y}_{t_{n,j}})
\Big)
\Big\|_{L^2(\Omega;\mathbb{R}^d)}
\nonumber\\
& \quad +
\Big\|
h
\big(
1
-
\sum_{j=0}^{k-1}
B_{t_{n,j}}
\big)
\Big(
2\tilde{s}_{\theta^\star}(T-t_n,{Y}_{t_n})
-
2 \nabla \log \tilde{p}_{T-t_n}(Y_{t_n})
\Big)
\Big\|_{L^2(\Omega;\mathbb{R}^d)}
\nonumber\\
& \quad +
\Big\|
h
\sum_{j=0}^{k-1}
B_{t_{n,j}}
\Big(
-\tilde{Y}_{t_{n,j}}
+
2\tilde{s}_{\theta^\star}(T-t_{n,j},\tilde{Y}_{t_{n,j}})
-
b_{t_{n,j}}
(\widehat{Y}_{t_{n,j}})
\Big)
\Big\|_{L^2(\Omega;\mathbb{R}^d)}
\nonumber\\
& =:
A_{1}+A_2+A_3.
\end{align}
For the estimate of $A_2$, using Assumption \ref{ass:score-function-approximation}, one gets
{\color{black}
\[
\begin{aligned}
A_2^2
&\leq
4h^2
\E
\Big[
\big(
1-\sum_{j=0}^{k-1}B_{t_{n,j}}
\big)^2
\Big]
\epsilon_0^2
\leq
4h^2\epsilon_0^2,
\end{aligned}
\]
and hence \(A_2\leq 2h\epsilon_0\).
}
In the following, we cope with $A_3$.
Applying the triangle inequality, the Lipschitz condition of $b_t(\cdot)$ and Assumption \ref{ass:score-function-approximation} yields 
\begin{align}
A_3^2
& \leq
2h^2 
\,
\E
\bigg[
\Big\|
\sum_{j=0}^{k-1}
B_{t_{n,j}}
\big(
-\tilde{Y}_{t_{n,j}}
+
2\tilde{s}_{\theta^\star}(T-t_{n,j},\tilde{Y}_{t_{n,j}})
-
b_{t_{n,j}}
(\tilde{Y}_{t_{n,j}})
\big)
\Big\|^2
\bigg]
\nonumber\\
& \quad +
2h^2 
\,
\E
\bigg[
\Big\|
\sum_{j=0}^{k-1}
B_{t_{n,j}}
\big(
b_{t_{n,j}}
(\tilde{Y}_{t_{n,j}})
-
b_{t_{n,j}}
(\widehat{Y}_{t_{n,j}})
\big)
\Big\|^2
\bigg]\nonumber\\
& {\color{black}\leq} 
{\color{black}4 h^2}
\sup_{0\leq j\leq k-1}
\E
\Big[
\big\|
2\tilde{s}_{\theta^\star}(T-t_{n,j},\tilde{Y}_{t_{n,j}})
-
2 {\color{black}\nabla \log \tilde{p}_{T-t_{n,j}}(\tilde{Y}_{t_{n,j}})}
\big\|^2
\Big]\nonumber\\
& \quad +
{\color{black}4 (2L+1)^2 h^2}
\sup_{0\leq j\leq k-1}
\E
\Big[
\big\|
\tilde{Y}_{t_{n,j}}
-
\widehat{Y}_{t_{n,j}}
\big\|^2
\Big]\nonumber\\
& {\color{black}\leq}
{\color{black}16}
\big(
1+(2L+1)^2\big)
h^2 \epsilon_0^2,
\end{align}
{\color{black}
where we use the elementary estimate
\[
\mathbb{E}
\Big[
\big\|\sum_{j=0}^{k-1}B_{t_{n,j}}V_j\big\|^2
\Big]
\leq
2\sup_{0\leq j\leq k-1}
\mathbb{E}\big[\|V_j\|^2\big].
\]
}
Following the same argument as in Proposition \ref{prop:pi-X_T}, we get 
\begin{equation}
\|\mathcal{E}_{t_{n+1}}\|_{L^2(\Omega;\mathbb{R}^d)}
\leq 
\zeta_{t_n} 
\|\mathcal{E}_{t_n}\|_{L^2(\Omega;\mathbb{R}^d)}
+
\big(
{\color{black}2+4\sqrt{1+(2L+1)^2}}
\big)
h
\epsilon_0,
\end{equation}
where 
\begin{equation}
\begin{aligned}
\zeta_{t_n}
:=
\begin{cases}
\Big(1+
\big(C(L,\alpha) h-\min_{0\leq j\leq k-1} 
(2 C_{T-{t_{n,j}}}+1)
\big)h
\Big)^{1/2},
& \text { for } (n+1)k \leq N_h 
\\[5pt]
1+
\tilde{C}(L)h,
& \text { for } (n+1)k > N_h.
\end{cases}
\end{aligned}
\end{equation}
Recalling that we already know $\zeta_{t_n}\leq 1$ for $(n+1)k\leq N_h$,
we obtain by iteration that
\begin{equation}
\begin{aligned}
\big\|
\mathcal{E}_{t_N}
\big\|_{L^2(\Omega;\mathbb{R}^d)}
& \leq 
\prod_{\ell=0}^{N-1}
\zeta_{t_\ell} 
\big\|
\mathcal{E}_{0}
\big\|_{L^2(\Omega;\mathbb{R}^d)}
+
\Big(
\sum_{m=1}^{N-1}
\prod_{\ell=m}^{N-1}
\zeta_{t_\ell}
+
1\Big)
\Big(
2+2\sqrt{2(2L+1)^2+1}
\Big)
h
\epsilon_0,
\end{aligned}
\end{equation}
Now, by virtue of \eqref{eq:recusion} and the fact that $X_0^\infty-Y_0=0$, we get 
\begin{equation}
\mathcal{W}_2
\big(
\mathcal{L}(X^{\infty}_{t_N}),
\mathcal{L}(Y_{t_N})
\big)
\leq 
C
\left(
\hat{T}(\alpha,M,0)
+
\tfrac{1}
{\tilde{C}(L)}
\right) 
\mathrm{e}^{\tilde{C}(L)
\eta(\alpha, M,\rho^{(4)}_h)}
\epsilon_0.
\end{equation}
The proof of Proposition \ref{prop:score-error}
is finished.
\end{proof}

\subsection{Score matching error of EM method}
\begin{prop}
[Score matching error of EM method]
\label{prop:em-score-error}
Suppose that Assumptions \ref{ass:weak-convex}, \ref{ass:one-side-lipschitz} and \ref{ass:score-function-approximation} hold. 
For any $n \in [N]_0$ and 
let $\{X^\infty_{t_{n}}\}$ and $\{Y_{t_{n}}\}$ be produced by \eqref{eq:em-aux} and \eqref{eq:EM} using the uniform step-size 
\[
h \leq 
\tfrac{2}{(2L+1)^2}.
\]
Then there exists a constant $C$, independent of $d$ and $h$, such that,
\begin{equation}
\mathcal{W}_2
\big(
\mathcal{L}(X^{\infty}_{t_N}),
\mathcal{L}(Y_{t_N})
\big)
\leq 
C 
\mathrm{e}^{C
\eta(\alpha, M,(2L+1)^2h/2)}
T
\epsilon_0,
\end{equation}
with $\eta(\alpha, M,(2L+1)^2h/2)$ as defined in \eqref{eq:eta-def}, where $L$ comes from Lemma \ref{lem:score-lipschitz}.
\end{prop}

\begin{proof}
By employing the same notation and arguments as those used in the proof of Proposition \ref{prop:score-error}, along with the triangle inequality, one derives
\begin{equation}
\begin{aligned}
\big\|
\mathcal{E}_{t_{n+1}}
\big\|_{L^2(\Omega;\mathbb{R}^d)}
& =
\underbrace{
\big\|
\mathcal{E}_{t_{n}}
+
h
\big(
b_{t_n}
(X^\infty_{t_n})
-
b_{t_n}
(Y_{t_n})
\big)
\Big\|_{L^2(\Omega;\mathbb{R}^d)}
}_{=:\mathcal{A}_1}\\
& \quad +
\underbrace{
\Big\|
h
\big(
2\tilde{s}_{\theta^\star}(T-t_n,{Y}_{t_n})
-
2 \nabla \log \tilde{p}_{T-t_n}
(Y_{t_n})
\big)
\Big\|_{L^2(\Omega;\mathbb{R}^d)}
}_{=:\mathcal{A}_2}.
\end{aligned}
\end{equation}
By Assumption \ref{ass:score-function-approximation}, we have 
\[
\mathcal{A}_2 
\leq 
2h\epsilon_0.
\]
With regard to $\mathcal{A}_1$, by revisiting the proof of Proposition \ref{prop:em-pi-X_T} and invoking the above estimation together with the fact that $X^\infty_{0} = Y_{0} \sim \gamma^d$, we deduce
\begin{equation}
\begin{aligned}
\big\|
\mathcal{E}_{t_{N}}
\big\|_{L^2(\Omega;\mathbb{R}^d)}
& \leq 
\prod_{\ell=0}^{N-1}
\zeta_{t_\ell} 
\|\mathcal{E}_{0}\|_{L^2(\Omega;\mathbb{R}^d)}
+
\Big(
\sum_{m=1}^{N-1}
\prod_{\ell=m}^{N-1}
\zeta 
_{t_\ell}
+
1\Big)
2
h
\epsilon_0\\
& \leq 
\Big(
\sum_{m=1}^{N-1}
\prod_{\ell=m}^{N-1}
\zeta 
_{t_\ell}
+
1\Big)
2
h
\epsilon_0 \\
& \leq 
C
\left(
\hat{T}(\alpha,M,0)
+
\tfrac{1}
{(2L+1)}
\right) 
\mathrm{e}^{(2L+1)
\eta(\alpha, M,\rho^{(5)}_h)}
\epsilon_0,
\end{aligned}
\end{equation}
with 
\begin{equation}
\begin{aligned}
\zeta_{t_n}
:=
\begin{cases}
\big(1+
\big((2L+1)^2h-2(2C_{T-t_n}+1)\big)
h \big)^{1/2},
& \text { for } n \leq N_h 
\\[5pt]
1+
(2L+1) h,
& \text { for } n > N_h.
\end{cases}
\end{aligned}
\end{equation}
The proof is completed.
\end{proof}

{\color{black}
\section{Proof of Proposition \ref{prop:time-lipschitz-gaussian-mixture}}
\label{app:proof-lipschitz-t}
\begin{proof}
The explicit solution of the OU process \eqref{eq:forward-SDE} gives
\[
\overrightarrow{X}_t
\overset{\mathrm{d}}{=}
\mathrm{e}^{-t}\overrightarrow{X}_0
+
\sqrt{1-\mathrm{e}^{-2t}}\,Z,
\quad 
Z \sim \gamma^d,
\]
where $Z$ is independent of $\overrightarrow{X}_0$. 
For the sake of notation, we denote
\[
a_r:=e^{-r},
\qquad
v_r
:=
1+e^{-2r}
(\sigma^2-1)
\geq  
v
:=\min\{1,\sigma^2\}>0 .
\]
Then, we have that the density distribution of this law is
\begin{equation}
\begin{aligned}
p_r(x)
& =
\sum_{i=1}^{m}
\beta_i
\frac{1}{(2\pi v_r)^{d/2}}
\exp\left(
-\frac{\|x-a_r\mu_i\|^2}{2v_r}
\right).
\end{aligned}
\end{equation}
Consequently, the score function of $p_r$ is given by
\begin{equation}
\begin{aligned}
\nabla\log p_r(x)
& =
-\frac{x}{v_r}
+
\frac{a_r}{v_r}
\frac{
\sum_{i=1}^m
\beta_i 
\mu_i
\exp
\left(
-\frac{\|x-a_r\mu_i\|^2}{2v_r}
\right)
}
{
\sum_{i=1}^m
\beta_i
\exp\left(
-\frac{\|x-a_r\mu_i\|^2}{2v_r}
\right)
}\\
& =
-\frac{x}{v_r}
+
\frac{a_r}{v_r}
\frac{
\sum_{i=1}^m
\beta_i 
\mu_i
\exp\left(
-\frac{a_r^2\|\mu_i\|^2}{2 v_r}
+
\frac{a_r \langle\mu_i,x\rangle}{v_r}
\right)
}{
\sum_{i=1}^m
\beta_i
\exp\left(
-\frac{a_r^2\|\mu_i\|^2}{2 v_r}
+
\frac{a_r \langle\mu_i,x\rangle}{v_r}
\right)}\\
& =:
-\frac{x}{v_r}
+
\frac{a_r}{v_r}
\frac{
\sum_{i=1}^m
\beta_i 
\mu_i \psi_i(r,x) }
{\sum_{i=1}^m
\beta_i \psi_i(r,x) },
\end{aligned}
\end{equation}
where 
\[
\psi_i(r,x)
=
\exp 
\left(
\tfrac{a_r}{v_r}
\langle\mu_i,x\rangle
-
\tfrac{a_r^2}{2v_r}
\|\mu_i\|^2 
\right).
\]
Define
\[
\omega_i(r,x)
:=
\frac{\beta_i\psi_i(r,x)}
{\sum_{j=1}^m\beta_j\psi_j(r,x)},
\qquad
\overline{\mu}_r(x)
:=
\sum_{i=1}^m\omega_i(r,x)\mu_i.
\]
As a consequence,
\[
\nabla\log p_r(x)
=
-\frac{x}{v_r}
+
\frac{a_r}{v_r}
\overline{\mu}_r(x).
\]
It follows that
\begin{equation}
\begin{aligned}
\partial_r\overline{\mu}_r(x)
&=
\sum_{i=1}^m
\mu_i\partial_r\omega_i(r,x)
=
\sum_{i=1}^m
\omega_i(r,x)
\bigl(
\mu_i-\overline{\mu}_r(x)
\bigr)
\partial_r\ell_i(r,x),
\end{aligned}
\end{equation}
where 
\[
\ell_i(r,x)
:=
\log\psi_i(r,x)
=
\tfrac{a_r}{v_r}
\langle\mu_i,x\rangle
-
\tfrac{a_r^2}{2v_r}
\|\mu_i\|^2.
\]
Before proceeding further, we give some estimate. First, we have 
\begin{align*}
\Big|
\big(\tfrac{a_r}{v_r}\big)'
\Big|
=
\Big|
\tfrac{a_r'}{v_r}
-
\tfrac{a_r v_r'}{v_r^2}
\Big|
\leq 
\tfrac{1}{v}
+
\tfrac{2|\sigma^2-1|}{v^2}
=: C_1, 
\end{align*}
and
\begin{align*}
\Big|
\big(\tfrac{a_r^2}{2v_r}\big)'
\Big|
&=
\left|
\tfrac{a_ra_r'}{v_r}
-
\tfrac{a_r^2v_r'}{2v_r^2}
\right|
\le
\tfrac{1}{v}
+
\tfrac{|\sigma^2-1|}{v^2}
=:C_2.
\end{align*}
Define
$
R:=\max_{1\leq i\leq m}\|\mu_i\|,
$
which is independent of $d$. By using the Cauchy-Schwarz inequality and the above estimates, we can derive 
\[
\big|
\partial_r\ell_i(r,x)
\big|
=
\Big|
\big(\tfrac{a_r}{v_r}\big)'
\langle\mu_i,x\rangle
-
\big(\tfrac{a_r^2}{2v_r}\big)'\|\mu_i\|^2
\Big|
\le
C_1 R\|x\|
+
C_2R^2.
\]
Since \(\overline\mu_r(x)\) is a convex combination of the vectors
\(\mu_i\), we have
$
\|\overline\mu_r(x)\|\leq R.
$
Using
$
\|\mu_i-\overline{\mu}_r(x)\|
\le
\|\mu_i\|+\|\overline{\mu}_r(x)\|
\le 2R
$, 
we obtain
\[
\begin{aligned}
\|\partial_r\overline{\mu}_r(x)\|
&\le
\sum_{i=1}^m
\omega_i(r,x)
\big\|
\mu_i-\overline{\mu}_r(x)
\big\|
\cdot
|\partial_r\ell_i(r,x)|
\\
&\le
2C_1R^2\|x\|
+
2C_2R^3.
\end{aligned}
\]
All together, we obtain
\begin{equation}
\begin{aligned}
\big\|
\partial_r\nabla\log p_r(x)
\big\|
&\le
\frac{\|v_r' x\|}{v_r^2}
+
\Big|
\big(\tfrac{a_r}{v_r}\big)'
\Big|
\cdot
\big\|\overline{\mu}_r(x)
\big\|
+
\big|
\tfrac{a_r}{v_r}
\big|
\cdot
\big\|
\partial_r\overline{\mu}_r(x)
\big\|
\\
&\le
\tilde L(1+\|x\|),
\end{aligned}
\end{equation}
where 
$\tilde{L}
:= 
\max\left\{
\tfrac{2|\sigma^2-1|}{v^2}
+
\tfrac{2C_1R^2}{v},
\,
C_1R
+
\tfrac{2C_2R^3}{v}
\right\},$ is independent of $d$.
Therefore, for any $t_n\leq t \leq s\leq t_{n+1}$, 
\begin{equation}
\begin{aligned}
\big\|
\nabla\log p_{T-t}(x)
-
\nabla\log p_{T-s}(x)
\big\|
&\le
\int_{T-s}^{T-t}
\big\|
\partial_r\nabla\log p_r(x)
\big\|
\,\mathrm dr
\\
&\le
\tilde L h(1+\|x\|).
\end{aligned}
\end{equation}
This concludes the proof.
\end{proof}

\section{Proof of Corollary \ref{cor:PM}}
\label{app:proof-mixting-time}
\begin{proof}
Given an error tolerance $\epsilon>0$, Theorem \ref{thm:poisson} tells that, for $T$ being large enough, $h$ being small enough and $\epsilon_0$ being small enough such that
\begin{equation}
\label{appendix-eq:mixing-time-three-term}
\begin{aligned}
    T^3\big(
    \tfrac{\sqrt{d}h}{k}
    +
    k d h^2
    +
    dh^2
    +
    \sqrt{d}h^{3/2}
    \big)
    \leq 
    \tfrac{\epsilon}{3},
    \quad 
    \mathrm{e}^{-T}
    \mathcal{W}_2
    (\mu,
    \gamma^d)
    \leq 
    \tfrac{\epsilon}{3},
    \quad
    \text{and }
    \quad
    \epsilon_0 T
    \leq 
    \tfrac{\epsilon}{3},
\end{aligned}
\end{equation}
one can arrive at
\[
\mathcal{W}_2
    \left(\mu, \mathcal{L}\left(Y_{t_N}\right)\right)
\leq \epsilon.
\]
According to Lemma \ref{lem:target-p-order-moment}, we have $\mathcal W_2(\mu,\gamma^d)\le\sqrt{(\mathcal M+1)d}$. 
Then, the second and third part of inequality \eqref{appendix-eq:mixing-time-three-term} requires 
\[
T \geq \log\big(\tfrac{3\sqrt{(\mathcal M+1)d}}{\epsilon})\big),
\quad 
\epsilon_0 \leq 
\tfrac{\epsilon}{3T}
\le 
\tfrac{\epsilon}{3\log\big(3\epsilon^{-1}\sqrt{(\mathcal M+1)d})\big)}.
\]
Next, we deal with the first term. We first choose the batch size $k$ to be of order $\tilde{\mathcal{O}}(d^{1/2} h/\epsilon)$ to ensure that
\[
C d^{1/2}\,\tfrac{h}{k}
\leq 
\tfrac{\epsilon}{12}.
\]
Then, based on the choice of $k$, one has 
\begin{equation}
\label{eq:mixing-time-h}
    C d^{3/2} h^{3}
    \leq 
    \tfrac{\epsilon^{2}}{12}, 
    \quad 
    C
    d h^{2}
    \leq 
    \tfrac{\epsilon}{12}, 
    \quad
    \text{and}
    \quad
    C
    \sqrt d h^{3/2}
    \leq 
    \tfrac{\epsilon}{12},
\end{equation}
which ensure that the first inequality in \eqref{appendix-eq:mixing-time-three-term} holds. Rearranging \eqref{eq:mixing-time-h} gives
\[
h
\leq 
C
\min\!\left\{
\epsilon^{2/3}/d^{1/2},
\epsilon^{1/2}/d^{1/2},
\epsilon^{2/3}/d^{1/3}
\right\}. 
\]
This completes the proof.
\end{proof}
}

\end{document}